\documentclass[10pt]{paper}
\usepackage[active]{srcltx}
\usepackage{amsmath}\usepackage{amssymb}\usepackage{amscd}\usepackage{amsthm}\usepackage{amsfonts}
\usepackage{empheq}
\usepackage{cite}
\usepackage{fullpage}
\usepackage[pagebackref=false,bookmarks=true,colorlinks=true,linktoc=page,citecolor=black,linkcolor=black,urlcolor=black]{hyperref}

\newcommand{\R}{\mathbb{R}}
\newcommand{\Z}{\mathbb{Z}}
\newcommand{\eps}{\varepsilon}
\newcommand{\hH}{\widehat{H}}

\newcommand{\ti}{\textbf{i}}
\newcommand{\tR}{\textbf{R}}
\newcommand{\cI}{\mathcal{I}}
\newcommand{\bl}{\boldsymbol{\lambda}}
\newcommand{\blambda}{\boldsymbol{\lambda}}
\newcommand{\some}{\diamond}
\newcommand{\tQ}{\widetilde{Q}}
\newcommand{\oX}{\overline{X}}

\newcommand{\idemp}{e(\ti)}
\newcommand{\idmep}{e(\ti)}
\newcommand{\idmepj}{e(\textbf{j})}

\newtheorem{theorem}{\textbf{Theorem}}[section]

\newtheorem{lemma}[theorem]{\textbf{Lemma}}
\newtheorem{proposition}[theorem]{\textbf{Proposition}}
\newtheorem*{theo-intro}{\textbf{Theorem}}

\theoremstyle{definition}
\newtheorem{definition}[theorem]{\textbf{Definition}}

\theoremstyle{remark}
\newtheorem{remark}[theorem]{\textbf{Remark}}

\numberwithin{equation}{section}

\begin{document} 

\title{Affine Hecke algebras and generalisations of quiver Hecke algebras for type $B$}
\author{L. Poulain d'Andecy and R. Walker}

\maketitle

\begin{abstract}
\let\thefootnote\relax\footnote{The second author is supported in part by the European Research Council in the framework of The European Union H2020 with the Grant ERC 647353 QAffine.}We define and study cyclotomic quotients of affine Hecke algebras of type $B$. We establish an isomorphism between direct sums of blocks of these algebras and a generalisation, for type $B$, of cyclotomic quiver Hecke algebras which are a family of graded algebras closely related to algebras introduced by Varagnolo and Vasserot. Inspired by the work of Brundan and Kleshchev we first give a family of isomorphisms for the corresponding result in type $A$ which includes their original isomorphism. We then select a particular isomorphism from this family and use it to prove our result. 
\end{abstract}

\section{Introduction}\label{sec-intro}

KLR algebras, also called quiver Hecke algebras, are a family of graded algebras which have been introduced by Khovanov, Lauda \cite{KL} and Rouquier \cite{Rou} in order to categorify quantum groups. Namely, the finite-dimensional graded projective modules over KLR algebras categorify the negative half of the quantised universal enveloping algebra of the Kac-Moody algebra associated to some Cartan datum. KLR algebras have been the subject of intense study over the past decade. They depend upon a quiver, and when the quiver is of type A, their representation theory is intimately related to that of affine Hecke algebras of type A.

There exist finite-dimensional quotients of KLR algebras, known as cyclotomic KLR algebras, which were introduced in order to categorify an irreducible highest weight module of a quantum Kac-Moody algebra. On the other hand, there exist finite-dimensional quotients of affine Hecke algebras of type A, known as cyclotomic Hecke algebras, or Ariki--Koike algebras. In 2008, Brundan and Kleshchev \cite{BK} showed that there is an explicit isomorphism between blocks of cyclotomic Hecke algebras and cyclotomic KLR algebras (for quivers of type A). 

Among the interesting consequences of the isomorphism of Brundan and Kleshchev, a non-trivial $\mathbb{Z}$-grading was immediately established on the cyclotomic Hecke algebras (cyclotomic KLR algebras are naturally $\Z$-graded) which was not previously known. Another concerns the dependence of the cyclotomic Hecke algebras on their deformation parameter. Indeed the deformation parameter appears in the definition of the cyclotomic KLR algebras only through the quiver, which is determined once the order of the deformation parameter is known. 

\vskip .2cm
The main purpose of this paper is to give a generalisation of the isomorphism of Brundan and Kleshchev \cite{BK} in the context of affine Hecke algebras of type $B$. 
So we first define finite-dimensional cyclotomic quotients of affine Hecke algebras of type $B$ which are natural analogues of the cyclotomic Hecke algebras. We note that, by their definition, understanding the representation theory of all these cyclotomic quotients allows one to understand all finite-dimensional representations of affine Hecke algebras of type $B$.

We then look for a candidate to replace the cyclotomic KLR algebras appearing in the isomorphism of Brundan and Kleshchev. We note that despite their name, KLR algebras of type $B$ (that is, with a type $B$ quiver) are not relevant here. Instead, the natural candidate is found to be a family of algebras introduced by Varagnolo and Vasserot in \cite{VV} which they used to prove a conjecture of Enomoto and Kashiwara \cite{EK} concerning the category of representations of affine Hecke algebras of type $B$. They are a family of $\Z$-graded algebras which play, in this context, a role similar to that played by KLR algebras for affine Hecke algebras of type $A$. We extend this similarity by recovering them in our generalisation of the isomorphism of Brundan and Kleshchev.

We note that in the work of both Enomoto--Kashiwara \cite{EK} and Varagnolo--Vasserot \cite{VV}, a restriction was imposed on the representations of the affine Hecke algebras of type $B$. This restriction stated that $\pm1$ does not appear in the set of eigenvalues of certain elements (or equivalently, $\pm1$ does not appear in the vertex set of the quiver defining the algebras used by Varagnolo and Vasserot). Here we work with no restriction on the considered representations and thus we have to generalise the definition of the algebras of Varagnolo and Vasserot in order to include the possibility of $\pm1$ in the set of eigenvalues. We remark here that our convention of the roles of the deformation parameters $p$ and $q$ in this paper is the opposite convention to that used by both Varagnolo--Vasserot in \cite{VV} and Walker in \cite{Walker}. 

\paragraph{Statement of the main result.} Throughout the paper we work over a field $K$ of characteristic different from 2 and we fix $p,q\in K\setminus\{0\}$ such that $p^2, q^2\neq 1$. To state our main result, we introduce some notation. For any $\lambda\in K\setminus\{0\}$, let $I_{\lambda}:=\{\lambda^{\epsilon}q^{2l}\ |\ \epsilon\in\{-1,1\}\,,\ \ l\in\Z\}$. Fix an $l$-tuple $\bl=(\lambda_1,\dots,\lambda_l)$ of elements of $K\setminus\{0\}$ such that we have $I_{\lambda_a}\cap I_{\lambda_b}=\emptyset$ if $a\neq b$. Then we fix a multiplicity map
\begin{equation}\label{def-m-intro}
m\ :\ I_{\lambda_1}\cup\dots\cup I_{\lambda_l}\to \Z_{\geq0}\ ,
\end{equation}
with finite support. That is, such that only a finite number of multiplicities $m(i)$ are different from $0$. To such a multiplicity map we associate a cyclotomic quotient, denoted $H(B_n)_{\bl,m}$, of the affine Hecke algebra $\hH(B_n)$. Note again that we consider arbitrary eigenvalues and consequently any finite-dimensional representation of the affine Hecke algebra $\hH(B_n)$ factors through one of these cyclotomic quotients.

\vskip .2cm
Our main goal is to give a new presentation of the cyclotomic quotient $H(B_n)_{\blambda,m}$. To do this, we consider a quiver denoted $\Gamma_{\bl}$ associated to the choice of $\bl=(\lambda_1,\dots,\lambda_l)$. The vertices of this quiver are indexed by the set $I_{\lambda_1}\cup\dots\cup I_{\lambda_l}$ (and we will identify a vertex with its index). The arrows are given as follows. For every $i\in I_{\lambda_1}\cup\dots\cup I_{\lambda_l}$ there is an arrow with origin $i$ and target $q^2i$.

To this quiver, we associate an algebra denoted $V_{\bl}^{n}$. If $\pm1\notin I_{\lambda_1}\cup\dots\cup I_{\lambda_l}$ then these algebras were introduced by Varagnolo and Vasserot in \cite{VV}. Finally, to the multiplicity map $m$ above, we associate a cyclotomic quotient of this algebra, denoted $V_{\bl,m}^{n}$. The algebras $V_{\bl,m}^{n}$ are analogues of the cyclotomic KLR algebras.

All the precise definitions are given in Section \ref{sec-defV}. We can now state our main result.

\begin{theorem}\label{theo}
The algebra $H(B_n)_{\bl,m}$ is isomorphic to the algebra $V_{\bl,m}^{n}$.
\end{theorem}

In fact, we prove Theorem \ref{theo} by first decomposing $H(B_n)_{\blambda,m}$ as a direct sum of subalgebras. Let $\beta$ be an orbit for the natural action (by inversion and permutation) of the Weyl group of type $B_n$ on the set $\bigl(I_{\lambda_1}\cup\dots\cup I_{\lambda_l}\bigr)^n$. An element $e_{\beta}$ of $H(B_n)_{\blambda,m}$ is naturally associated to such an orbit which is, if non-zero, a central idempotent. The algebra $H(B_n)_{\blambda,m}$ then decomposes as a direct sum whose summands are the non-zero $e_{\beta}H(B_n)_{\blambda,m}$.

\vskip .2cm
The algebras $e_{\beta}H(B_n)_{\blambda,m}$ are in fact our main objects of study. They are the analogues for type B of the blocks of cyclotomic quotients of type A studied in \cite{BK}. In general, the non-zero subalgebras $e_{\beta}H(B_n)_{\blambda,m}$ are direct sums of blocks of $H(B_n)_{\blambda,m}$. Note that here we do not know if the non-zero idempotents $e_{\beta}$ are primitive in general (whereas the analogous statement is known for cyclotomic quotients of type A). However, as already noted in \cite{BK} for type A, the primitivity of $e_{\beta}$ plays no role in this work.

\vskip .2cm
An orbit $\beta$ in $\bigl(I_{\lambda_1}\cup\dots\cup I_{\lambda_l})^n$ corresponds to a dimension vector for the quiver $\Gamma_{\bl}$ (satisfying a compatibility condition with the involution $i\mapsto i^{-1}$ on the set of vertices), and we have associated algebras denoted by $V_{\bl}^{\beta}$ and their cyclotomic quotients $V_{\bl,m}^{\beta}$. By definition
\[V_{\bl}^{n}:=\bigoplus_\beta V_{\bl}^{\beta}\ \ \ \ \text{and}\ \ \ \ V_{\bl,m}^{n}:=\bigoplus_\beta V_{\bl,m}^{\beta}\ ,\]
where the direct sums are over the set of orbits $\beta$ in $\bigl(I_{\lambda_1}\cup\dots\cup I_{\lambda_l})^n$. What we prove to obtain Theorem \ref{theo} is that, for any orbit $\beta$, the algebras $e_{\beta}H(B_n)_{\blambda,m}$ and $V_{\bl,m}^{\beta}$ are isomorphic.

\paragraph{Consequences.} Since the algebras $V_{\bl,m}^{\beta}$ (and in turn $V_{\bl,m}^{n}$) are naturally $\Z$-graded, the theorem provides a non-trivial $\Z$-grading on the cyclotomic quotients $H(B_n)_{\bl,m}$ . It suggests in particular that one can study a graded representation theory for the affine Hecke algebra of type B.

\vskip .2cm
Moreover, from the definitions, it is clear that the algebra $V_{\bl}^{n}$ depends only on the following information: the quiver $\Gamma_{\bl}$, the involution $i\mapsto i^{-1}$ on the set of vertices of $\Gamma_{\bl}$, and the possible presence and position of $\pm p$ in $\Gamma_{\bl}$. This has immediate consequences concerning the values of $\bl$ for which it is enough to consider, and concerning the dependence of the algebras $H(B_n)_{\bl,m}$ on the deformation parameters $p$ and $q$. To be more precise, if $q$ is a root of unity, let $e$ be the smallest non-zero positive integer such that $q^{2e}=1$; if $q$ is not a root of unity, set $e:=\infty$. Then, if $p\in\pm q^{\mathbb{Z}}$ define $e'$ in $\mathbb{Z}/e\mathbb{Z}$ (by convention, $e'\in\mathbb{Z}$ if $e=+\infty$) by $p=\pm q^{e'}$. If $p\notin\pm q^{\mathbb{Z}}$ let $e':=\infty$. As a direct consequence of Theorem \ref{theo}, the cyclotomic quotient $H(B_n)_{\bl,m}$ depends on $q$ and $p$ only through $e$ and $e'$.

\vskip .2cm
For the study of irreducible representations of the affine Hecke algebra $\hH(B_n)$, it is noted in \cite{EK} that it is enough to consider the situation $l=1$, where $\bl$ is a single $\lambda\in K\setminus\{0\}$. However, for a complete study of the category of representations of $\hH(B_n)$, one needs a priori to consider arbitrary cyclotomic quotients, even if one can expect to reduce everything to the situation $l=1$, in the spirit of the Dipper--Mathas Morita equivalence for type A \cite{DM}; see also \cite[\S 3.4]{Ro2} (this will be the subject of another paper). 

Assume in this paragraph that we have fixed $p,q \in K\setminus\{ 0 \}$ such that $p^2,q^2\neq 1$ and that $\bl$ consists of a single $\lambda$. A direct consequence of the theorem is that, as we indicated above, the only relevant information about $\lambda$ is contained in the quiver $\Gamma_{\lambda}$, the involution $i\mapsto i^{-1}$ on the set of vertices of $\Gamma_{\lambda}$, and the possible presence and position of $\pm p$ in $\Gamma_{\lambda}$. Clearly, it is enough to consider only one $\lambda$ in each orbit in $K\setminus\{0\}$ under inversion and multiplication by powers of $q^2$. Moreover, it is easy to see that $\lambda'=-\lambda$ gives equivalent data and therefore results in isomorphic algebras $V_{\lambda}^{n}$. So finally, we obtain that it is enough to consider one of the following four situations for the choice of $\lambda$:
\[\mathbf{(a)}\ \lambda=1\ \qquad \mathbf{(b)}\ \lambda=q\ \qquad \mathbf{(c)}\ \lambda=p\ \qquad \mathbf{(d)}\ \lambda\notin\{\pm q^{\mathbb{Z}}\,,\,\pm p^{\pm1}q^{2\mathbb{Z}}\}\ .\]
Situation \textbf{(d)} corresponds to a quiver with two disconnected components exchanged by the involution $i\mapsto i^{-1}$, and not containing any of the special points $\pm p$. All $\lambda\notin\{\pm q^{\mathbb{Z}}\,,\,\pm p^{\pm1}q^{2\mathbb{Z}}\}$ result then in isomorphic algebras $V_{\lambda}^{n}$.

If $p\in \pm q^{\mathbb{Z}}$ then Case \textbf{(c)} amounts to studying one of the Cases \textbf{(a)} or \textbf{(b)}. If $p\notin \pm q^{\mathbb{Z}}$ then situation \textbf{(c)} corresponds to a quiver with two disconnected components related by the involution $i\mapsto i^{-1}$, and containing at least one special point (it contains the two special points $\pm p$ only if $q^{2N}=-1$ for some $N$, namely, only if $e$ is even). 

If $q$ is an odd root of unity then Case \textbf{(b)} reduces to Case \textbf{(a)}. If $q$ is not an odd root of unity then situation \textbf{(b)} corresponds to a quiver with a single connected component, stable under the involution $i\mapsto i^{-1}$, and with no fixed point. Finally, Case \textbf{(a)} corresponds to a quiver with a single connected component which is stable under the involution $i\mapsto i^{-1}$ and with at least one fixed point (it contains the two fixed points $\pm1$ only if $q^{2N}=-1$ for some $N$, namely, only if $e$ is even).

\paragraph{Remarks on the proof.}
Our approach for the proof of Theorem \ref{theo} is inspired by the proof of Brundan and Kleshchev for the type A setting \cite{BK}. In particular, the isomorphism is given explicitly. In fact, we split the proof of Theorem \ref{theo} into two steps by introducing an intermediary algebra isomorphic to both sides of the isomorphism statement. Roughly speaking, the first step consists in replacing the Coxeter generators of $H(B_n)_{\bl,m}$ by modified intertwining elements and introducing a complete family of orthogonal idempotents. The second step consists in replacing the invertible commuting elements of $H(B_n)_{\bl,m}$ by commuting nilpotent elements and in making a subtle renormalisation of the intertwining elements in order to obtain the desired ``nice'' defining relations of $V_{\bl,m}^{n}$ (we note that these relations are defined over $\mathbb{Z}$).

\vskip .2cm
We emphasize that the restriction of our isomorphism to type A does not correspond to the isomorphism of Brundan--Kleshchev in \cite{BK}. Indeed it seems that their isomorphism can not be extended directly to the type B situation.

So we start by giving a whole family of isomorphisms between blocks of cyclotomic quotients of type A and cyclotomic KLR algebras, one of them being the isomorphism described in \cite{BK}. Let us say a few words about this family of isomorphisms. Let $X$ be an invertible endomorphism of a finite-dimensional vector space with a single eigenvalue $i$. One can obtain a nilpotent element by setting $y=1-i^{-1}X$, and this is used by Brundan and Kleshchev to construct the nilpotent elements required for the isomorphism in type A. Our generalisation relies on the remark that a more general possibility is to consider the following nilpotent element
\[y=f(1-i^{-1}X)\ \ \ \ \ \text{where $f\in K[[z]]$ with no constant term.}\]
So we consider a formal power series $f$ with no constant term and find that we are able to generalise the isomorphism obtained by Brundan and Kleshchev using the above construction of nilpotent elements. The only condition on $f$, in addition to having no constant term, is that it must have a composition inverse (that is, $f$ has a non-zero coefficient in degree 1). We note that the existence of not only one but a whole family of isomorphisms reflects the existence of a family of automorphisms of cyclotomic KLR algebras, that we describe explicitly.

Finally it turns out that the isomorphism that we are able to extend to the type B setting is the one using the following construction of nilpotent elements:
\[y=iX^{-1}-i^{-1}X=f(1-i^{-1}X)\ ,\]
where $f$ is the following power series: $f(z)=z+\displaystyle\frac{z}{1-z}\,$. Note that $f$ is indeed invertible for the composition since its coefficient in degree 1 is equal to 2 (and this is where we use that the characteristic of $K$ is different from 2).

\vskip .2cm
One can expect that a similar construction can be used for more general affine Hecke algebras. In particular we conjecture that one can construct finite-dimensional cyclotomic quotients of affine Hecke algebras of type D which relate in this way to quotients of the algebras introduced by Shan-Varagnolo-Vasserot \cite{SVV}.

\paragraph{Organisation of the paper.} In Section \ref{sec-defV} we give the definition of the algebras $V_{\bl}^{\beta}$, $V_{\bl}^{n}$ and their cyclotomic quotients. In Section \ref{sec-def}, we recall the definition of the affine Hecke algebras $\hH(B_n)$ and state some well-know properties. The definition of the cyclotomic quotients of $\hH(B_n)$ and of the main objects of study $e_{\beta}H(B_n)_{\bl,m}$ are given in Section \ref{sec-cyc}. In Section \ref{sec-inter}, we introduce an algebra which will play a role in our proof as an intermediary between $e_{\beta}H(B_n)_{\bl,m}$ and $V_{\bl,m}^{\beta}$. The generalised isomorphism in the type A setting is stated and proved in Section \ref{sec-A}, while the proof of our main result, Theorem \ref{theo}, is concluded in Section \ref{sec-proof}.

\section{The algebra $V_{\bl}^{\beta}$ and its cyclotomic quotients $V_{\bl,m}^{\beta}$}\label{sec-defV}

Fix an $l$-tuple $\bl=(\lambda_1,\dots,\lambda_l)$ of elements of $K\setminus\{0\}$ such that we have $I_{\lambda_a}\cap I_{\lambda_b}=\emptyset$ if $a\neq b$. We recall that $I_{\lambda}:=\{\lambda^{\epsilon}q^{2l}\ |\ \epsilon\in\{-1,1\}\,,\ \ l\in\Z\}$ for $\lambda\in K\setminus\{0\}$. Let
\[S:=I_{\lambda_1}\cup\dots\cup I_{\lambda_l}\ .\]
We construct a quiver denoted $\Gamma_{\bl}$. The vertices of this quiver are indexed by $S$ (and we will identify a vertex with its index) and the arrows are given as follows. For every $i\in S$ there is an arrow with origin $i$ and target $q^2i$. We note that this convention is the opposite to that used in \cite{VV} and in \cite{Walker}.

\vskip .2cm
For $n\geq1$, the formulas
\begin{equation*}
\begin{split}
r_0\cdot(i_1,\ldots,i_n) &=(i_1^{-1},i_2,\ldots,i_n) \\
r_k\cdot(\ldots,i_k,i_{k+1},\ldots) &=(\ldots,i_{k+1},i_k,\ldots)\ \ \quad\text{for $k=1,\dots,n-1$,}
\end{split}
\end{equation*}
provide an action of the Weyl group $W(B_n)$ of type $B_n$ on $S^n$. 
We fix $\beta$ an orbit in $S^n$ for this action.

\vskip .2cm
We will use the following notation, for $i,j\in S$:
\begin{equation*}
\begin{split}
i \nleftrightarrow j\quad &\textrm{ when } i\neq j \text{ and $j\notin\{q^2i,q^{-2}i\}$\,,} \\
i \rightarrow j \quad &\textrm{ when $j=q^2i\neq q^{-2}i$\,,}  \\
i \leftarrow j \quad &\textrm{ when $j=q^{-2}i\neq q^{2}i$\,,}  \\
i \leftrightarrow j \quad &\textrm{ when $j=q^{2}i=q^{-2}i$\,.}
\end{split}
\end{equation*}
Note that $i\neq j$ in all these cases (recall that $q^2\neq 1$) and that the fourth case means that $q^2=-1$ and $i=-j$. Similarly we write
\begin{equation*}
\begin{split}
i^{-1} \stackrel{p}{\nleftrightarrow} i &\textrm{ when } i\neq i^{-1} \text{ and $i\notin\{\pm p^{\pm1}\}$\,,} \\
i^{-1} \stackrel{p}{\longrightarrow} i &\textrm{ when } i\in\{\pm p\}\ \text{and}\ i\notin\{\pm p^{-1}\}\,,\\
i^{-1} \stackrel{p}{\longleftarrow} i &\textrm{ when } i\notin\{\pm p\}\ \text{and}\ i\in\{\pm p^{-1}\}\,,\\
i^{-1} \stackrel{p}{\longleftrightarrow} i &\textrm{ when } i\in \{\pm p\}\cap \{\pm p^{-1}\}.
\end{split}
\end{equation*} 
Note that $i\neq i^{-1}$ in all these cases (recall that $p^2\neq 1$) and that the fourth case means that $p^2=-1$ and $i^{-1}=-i$.

\begin{definition} \label{definition:vv algebra}
The algebra $V_{\bl}^{\beta}$ is the $\Z$-graded $K$-algebra generated by elements 
\begin{equation*}
\{ \psi_a \}_{0\leq a\leq n-1} \cup \{ y_j \}_{1\leq j\leq n} \cup \{ \idemp \}_{\textbf{i} \in \beta}
\end{equation*}
which satisfy the following defining relations. 
\begin{align}
\sum_{\textbf{i} \in \beta} \idmep & = 1\,,\ \ \quad \idmep\idmepj = \delta_{\ti,\textbf{j}}\idmep\,,\ \ \ \forall\ti,\textbf{j}\in \beta\,, \label{Rel:V1} \\ 
y_iy_j&=y_jy_i\,,\ \ \quad y_{i}\idmep=\idmep y_{i}\,,\ \ \ \ \forall i,j\in\{1,\ldots,n\} \text{ and } \forall\ti\in\beta\,,\label{Rel:V2}
\end{align}
and, for $a,b\in\{0,1,\dots,n-1\}$ with $b\neq 0$, for $j\in\{1,\dots,n\}$, and for $\ti=(i_1,\dots,i_n)\in\beta$,
\begin{align}
\psi_{a}\idmep &= e(r_{a}(\ti))\psi_a\,,\label{Rel:V3}\\
(\psi_b y_j - y_{s_b(j)}\psi_b)\idmep&=\left\{
\begin{array}{l l l}
			-\idmep & \quad \textrm{ if } j=b,\ i_b=i_{b+1}\,,\\
			\idmep & \quad \textrm{ if } j=b+1,\ i_b=i_{b+1}\,,\\
			0 & \quad \textrm{ else,}
\end{array} \right. \label{Rel:V4}
\end{align}
where $s_b$ is the transposition of $b$ and $b+1$ acting on $\{1,\dots,n\}$;
\begin{align}
\psi_a\psi_b&=\psi_b\psi_a \hspace{0.3em} \textrm{ if } |a-b|>1\,, \label{Rel:V5} \\
\psi_b^2\idmep &= \left\{
\begin{array}{l l l l l}
0 & \quad \textrm{if } i_b=i_{b+1}\,, \\[0.2em]
\idmep & \quad \textrm{if } i_{b} \nleftrightarrow i_{b+1}\,, \\[0.2em]
(y_{b+1} - y_b)\idmep & \quad \textrm{if } i_b \rightarrow i_{b+1}\,,\\[0.2em]
(y_b - y_{b+1})\idmep & \quad \textrm{if } i_b \leftarrow i_{b+1}\,, \\[0.2em]
(y_{b+1}-y_b)(y_b-y_{b+1})\idemp & \quad \textrm{if } i_b \leftrightarrow i_{b+1}\,, 
\end{array} \right. \label{Rel:V6}\\
(\psi_{b}\psi_{b+1}\psi_{b} - \psi_{b+1}\psi_{b}\psi_{b+1})\idemp &= \left\{
\begin{array}{l l l l}
\idmep & \quad i_b=i_{b+2}\rightarrow i_{b+1}\,\\[0.2em]
-\idmep & \quad i_b=i_{b+2}\leftarrow i_{b+1}\, \\[0.2em]
(y_{b+2}-2y_{b+1}+y_b)\idmep & \quad i_b=i_{b+2} \leftrightarrow i_{b+1}\,,\\[0.2em]
0 & \quad \text{else\,,}\,
\end{array} \right. \label{Rel:V7}
\end{align}

\begin{align}
(\psi_0 y_1 + y_1\psi_0)\idmep&=\left\{
\begin{array}{l l}
			0 & \textrm{ if } i_1^{-1}\neq i_1\,,\\[0.2em]
			2\idmep & \textrm{ if } i_1^{-1}= i_1\,,
\end{array} \right. \label{Rel:V8}\\
\psi_0 y_j &= y_j \psi_0 \textrm{ if } j > 1\,, \label{Rel:V9}\\
\psi_0^2\idemp&=\left\lbrace
\begin{array}{ll}
0 & \textrm{ if } i_1^{-1}=i_1\,, \\[0.2em]
\idmep & \textrm{ if } i_1^{-1} \stackrel{p}{\nleftrightarrow} i_1\,, \\[0.2em]
y_1\idmep & \textrm{ if } i_1^{-1} \stackrel{p}{\longrightarrow} i_1\,, \\[0.2em]
-y_1\idmep & \textrm{ if } i_1^{-1} \stackrel{p}{\longleftarrow} i_1\,, \\[0.2em]
-y_1^2\idmep & \textrm{ if } i_1^{-1} \stackrel{p}{\longleftrightarrow} i_1\,,
\end{array}
\right. \label{Rel:V10} 
\end{align}

\begin{align}
((\psi_0 \psi_1)^2 - (\psi_1\psi_0)^2)\idmep&=\left\lbrace
\begin{array}{ll}
2\psi_0\idmep & \textrm{ if } i_1^{-1}\neq i_1\rightarrow i_2=i_2^{-1}\,, \\[0.2em]
-2\psi_0\idmep & \textrm{ if } i_1^{-1}\neq i_1\leftarrow i_2=i_2^{-1}\,, \\[0.2em]
4(\psi_0y_1-1)\idmep & \textrm{ if } i_1^{-1}= i_1\leftrightarrow i_2=i_2^{-1}\,, \\[0.2em]
-\psi_1\idmep & \textrm{ if }i_2\neq i_1 \stackrel{p}{\longleftarrow} i_1^{-1}=i_2\,,\\[0.2em]
\psi_1\idmep & \textrm{ if } i_2\neq i_1 \stackrel{p}{\longrightarrow} i_1^{-1}=i_2\,,\\[0.2em]
\psi_1(y_1-y_2)\idmep & \textrm{ if } i_2\neq i_1 \stackrel{p}{\longleftrightarrow} i_1^{-1}=i_2\,, \\[0.2em]
0 & \textrm{ else.}
\end{array}\right.
\label{Rel:V11}
\end{align}
The $\Z$-grading on $V_{\bl}^{\beta}$ is given as follows,
\begin{equation*}
\begin{split}
&\textrm{deg}(\idmep)=0, \\ 
&\textrm{deg}(y_j \idmep)=2, \\
&\textrm{deg}(\psi_0\idmep)=\left\{
\begin{array}{l l}	
|\{i_1\}\cap \{\pm p\}|+|\{i_1\}\cap \{\pm p^{-1}\}| & \quad \textrm{ if  } i_1^{-1} \neq i_1\,, \\[0.2em]
-2 & \quad \textrm{ if  } i_1^{-1} = i_1\,,	
\end{array} \right. \\
&\textrm{deg}(\psi_b \idmep)=\left\{
\begin{array}{l l l}	
|i_b \rightarrow i_{b+1}|+|i_b \leftarrow i_{b+1}| & \quad \textrm{ if  } i_b \neq i_{b+1}\,,\\
-2 & \quad \textrm{ if  } 	i_b = i_{b+1}.	
\end{array} \right. 
\end{split}
\end{equation*}
where $|i \rightarrow j|$ represents the number of arrows in $\Gamma_{\bl}$ which have origin $i$ and target $j$.
\end{definition}

\begin{definition}
Let $m\ :\ S \longrightarrow \Z_{\geq0}$ be a multiplicity map with finite support, i.e. $m(i)=0$ for all but finitely many $i \in S$. We define the cyclotomic quotient, denoted by $V_{\bl,m}^{\beta}$, to be the quotient of $V_{\bl}^{\beta}$ by the relation
\begin{equation}\label{cycV}
y_1^{m(i_1)}\idmep=0 \ \ \ \textrm{ for every } \textbf{i} \in \beta\ .
\end{equation}
\end{definition}

Finally, we define algebras by summing over all possible orbits $\beta\subset S^n$ (alternatively, over the set of dimension vectors of height $n$ satisfying certain conditions; see Remark \ref{rem-orb}).
\begin{definition}
Let $\mathcal{O}_{S^n}$ be the set of orbits in $S^n$ under the action of the Weyl group $W(B_n)$. We set:
\[V_{\bl}^{n}:=\bigoplus_{\beta\in\mathcal{O}_{S^n}}V_{\bl}^{\beta}\ \ \ \ \text{and}\ \ \ \ V_{\bl,m}^{n}:=\bigoplus_{\beta\in\mathcal{O}_{S^n}}V_{\bl,m}^{\beta}\ .\]
\end{definition}

Let $\tR^{\beta}_{\bl}$ denote the algebra defined by generators as in Definition \ref{definition:vv algebra} with $\psi_0$ removed and with defining relations (\ref{Rel:V1})--(\ref{Rel:V7}). The algebra $\tR^{\beta}_{\bl}$ is the KLR algebra associated to the quiver $\Gamma_{\bl}$ and the dimension vector $\beta$. By definition, we have that the subalgebra of $V_{\bl}^{\beta}$ generated by all the generators except $\psi_0$ is a quotient of the algebra $\tR^{\beta}_{\bl}$ (and a similar immediate statement for $V_{\bl,m}^{\beta}$ and the cyclotomic KLR algebra $\tR^{\beta}_{\bl,m}$, which is the quotient of $\tR^{\beta}_{\bl}$ by (\ref{cycV})).
\\
\\
We have the following fundamental lemma, which is proved in \cite[Lemma 2.1]{BK} (more precisely, it is proved in \cite{BK} in the context of cyclotomic KLR algebras; however, all defining relations of cyclotomic KLR algebras are present in $V_{\bl,m}^{\beta}$, so the proof in \cite{BK} can be repeated here verbatim).
\begin{lemma} \label{lem-nily}
The elements $y_i\in V_{\bl,m}^{\beta}$ are nilpotent, for all $1 \leq i \leq n$.
\end{lemma}

\paragraph{Notation $\boldsymbol{{}^{r_a}P}$.} Let $b\in\{1,\dots,n-1\}$ and let $P\in K[[z_1,\dots,z_n]]$ be a formal power series in the nilpotent commuting variables $z_1,\dots,z_n$. We note that the following formulas, 
\[{}^{r_0}P(z_1,z_2,\dots,z_n)=P(-z_1,z_2,\dots,z_n)\ \ \ \ \text{and}\ \ \ \ {}^{r_b}P(z_1,\dots,z_n)=P(z_1,\dots,z_{b+1},z_b,\dots,z_n)\]
define an action of the Weyl group $W(B_n)$ on $K[[z_1,\dots,z_n]]$. With this notation, Formulas (\ref{Rel:V4}) and (\ref{Rel:V8}) are equivalent, respectively, to the following formulas
\begin{equation}\label{rel:V4'}
(\psi_bP-{}^{r_b}P\psi_b)e(\ti)=\left\{\begin{array}{ll}
0 & \ \ \text{if $r_b(\ti)\neq\ti$,}\\[0.2em]
\displaystyle\frac{P-{}^{r_b}P}{y_{b+1}-y_b}e(\ti) &\ \ \text{if $r_b(\ti)=\ti$,}
\end{array}\right.\ \ \ \ \text{for $P\in K[[y_1,\dots,y_n]]$,}
\end{equation}
\begin{equation}\label{rel:V8'}
(\psi_0P-{}^{r_0}P\psi_0)e(\ti)=\left\{\begin{array}{ll}
0 & \ \ \text{if $r_0(\ti)\neq\ti$,}\\[0.2em]
\displaystyle\frac{P-{}^{r_0}P}{y_1}e(\ti) &\ \ \text{if $r_0(\ti)=\ti$,}
\end{array}\right.\ \ \ \ \text{for $P\in K[[y_1,\dots,y_n]]$\,.}
\end{equation}

\begin{remark}\label{rem-orb}
We explain here that there is a bijection between the orbits of $W(B_n)$ in $S^n$ and the set of dimension vectors, for the quiver $\Gamma_{\bl}$, of height $n$ which satisfy a certain compatibility condition with respect to the involution $\theta\ :\ i\mapsto i^{-1}$.

Let ${^\theta}\mathbb{N}S$ be the set of maps $\tilde{\beta}$ from $S$ to $\Z_{\geq0}$, with finite support, such that $\tilde{\beta}(i^{-1})=\tilde{\beta}(i)$ for every $i\in S$. Using the standard notation of formal sums, we have
\begin{equation*}
{^\theta}\mathbb{N}S:=\left\lbrace \left. \tilde{\beta} = \sum_{i\in S}\tilde{\beta}_i i \hspace{0.4em}\right|\hspace{0.4em} \tilde{\beta}_i \in \Z_{\geq0},\ \tilde{\beta}_{i^{-1}}= \tilde{\beta}_{i} \hspace{0.5em} \forall i\in S\,, \textrm{ and } |\textrm{supp}(\tilde{\beta})|<\infty \right\rbrace.
\end{equation*}
For an element $\tilde{\beta}\in{^\theta}\mathbb{N}S$, let its height be the positive integer $n$ given by
\begin{equation*}
n=\frac{1}{2}\sum_{\substack{i\in S \\ i^2\neq 1}} \tilde{\beta}_i + \sum_{\substack{i\in S \\ i^2=1}} \tilde{\beta}_i\ .
\end{equation*}
From an orbit $\beta$, construct $\tilde{\beta}\in{^\theta}\mathbb{N}S$ of height $n$ as follows. Pick $(i_1,\dots,i_n)\in\beta$ and then, for $i\in S$, set $\tilde{\beta}(i)$ to be the number of elements $i_k$ among $i_1,\dots,i_n$ such that $i_k\in\{i,i^{-1}\}$.

Conversely, to $\tilde{\beta} \in {^\theta}\mathbb{N}S$ of height $n$, we associate an orbit $\beta$ as follows:
\begin{equation*}
\beta=\left\lbrace \textbf{i}=(i_1, \ldots, i_n) \in S^n \hspace{0.4em} \left| \hspace{0.4em}  \sum_{\substack{k=1 \\ i_k^2\neq 1}}^n \bigl(i_k+i_k^{-1}\bigr) + \sum_{\substack{k=1 \\ i_k^2=1}}^n i_k=\tilde{\beta} \right\rbrace\right.\ .
\end{equation*}
It is easy to check that these procedures give the desired bijection between the orbits of $W(B_n)$ in $S^n$ and the set of dimension vectors in ${^\theta}\mathbb{N}S$ of height $n$.
\hfill$\triangle$
\end{remark}

\section{Affine Hecke algebra $\hH(B_n)$}\label{sec-def}

\subsection{Root datum}

Let $\{\eps_i\}_{i=1,\dots,n}$ be an orthonormal basis of the Euclidean space $\R^n$ with scalar product $(.,.)$ and let
\[L:=\bigoplus_{i=1}^n\Z\eps_i\ ,\ \ \ \ \alpha_0=2\eps_1\ \ \ \ \ \text{and}\ \ \ \ \ \alpha_i=\eps_{i+1}-\eps_i\,,\ i=1,\dots,n-1\ .\]
The set $\{\alpha_i\}_{i=0,\dots,n-1}$ is a set of simple roots for the root system $R=\{\pm2\eps_i,\pm\eps_i\pm\eps_j\}_{i,j=1,\dots,n}$ of type B. For a root $\alpha\in R$, we identify its coroot $\alpha^{\vee}$ with the following element of $\text{Hom}_{\Z}(L,\Z)$:
\[\alpha^{\vee}\ :\ x\mapsto 2\frac{(\alpha,x)}{(\alpha,\alpha)}\ .\]
To each root $\alpha\in R$ is associated the element $r_{\alpha}\in\text{End}_{\Z}(L)$ defined by
\begin{equation}\label{act-W0}
r_{\alpha}(x):=x-\alpha^{\vee}(x)\,\alpha\ \ \ \ \ \ \text{for any $x\in L$.}
\end{equation}
The group generated by $r_{\alpha}$, $\alpha\in R$, is identified with the Weyl group $W(B_n)$ associated to $R$. Let $r_0,r_1,\dots,r_{n-1}$ be the elements of $W(B_n)$ corresponding to the simple roots $\alpha_0,\alpha_1,\dots,\alpha_{n-1}$.

\subsection{Finite Hecke algebras} 
Set $q_0:=p$ and $q_i:=q$ for $i=1,\dots,n-1$. The finite Hecke algebra $H(B_n)$ is the associative $K$-algebra with unit generated by elements $g_0,g_1,\dots,g_{n-1}$ satisfying the defining relations:
\begin{equation}\label{rel-H0}
\begin{array}{ll}
g_i^2=(q_i-q_i^{-1})g_i+1 & \text{for}\ i\in\{0,\dots,n-1\}\,,\\[0.2em]
g_0g_1g_0g_1=g_1g_0g_1g_0\,,\ \ \  & \\[0.2em]
g_ig_{i+1}g_i=g_{i+1}g_ig_{i+1}\ \ \  & \text{for $i\in\{1,\dots,n-2\}$}\,,\\[0.2em]
g_ig_j=g_jg_i\ \ \  & \text{for $i,j\in\{0,\dots,n-1\}$ such that $|i-j|>1$}\,.
\end{array}
\end{equation} 

For any element $w\in W(B_n)$, let $w=r_{a_1}\dots r_{a_k}$ be a reduced expression for $w$ in terms of the generators $r_0,\dots,r_{n-1}$. We define $g_w:=g_{a_1}\dots g_{a_k}\in H(B_n)$. This definition does not depend on the reduced expression for $w$ and and it is a standard fact that the set of elements
\begin{equation}\label{base-H0}
\{g_w\,,\ \ \ w\in W(B_n)\}
\end{equation} 
forms a $K$--basis for $H(B_n)$.

\subsection{Affine Hecke algebra of type B}

We denote by $\hH(B_n)$ the affine Hecke algebra associated to the above root datum. The algebra $\hH(B_n)$ is the associative $K$-algebra generated by elements 
$$g_0,g_1,\dots,g_{n-1}\ \ \ \ \text{and}\ \ \ \ X^x,\ \ x\in L\ ,$$
subject to the defining relations $X^0=1$, $X^xX^{x'}=X^{x+x'}$ for any $x,x'\in L$, and Relations (\ref{rel-H0}) together with
\begin{equation}\label{rel-Lu}
g_iX^x-X^{r_i(x)}g_i=(q_i-q_i^{-1})\frac{X^x-X^{r_i(x)}}{1-X^{-\alpha_i}}\ ,
\end{equation}
for any $x\in L$ and $i=0,1,\dots,n-1$. Note that $\displaystyle\frac{X^x-X^{r_i(x)}}{1-X^{-\alpha_i}}\in \hH(B_n)$ since $r_i(x)=x-k\alpha_i$ for some $k\in\mathbb{Z}$.

Let $X_i:=X^{\eps_i}$ for $i=1,\dots,n$, so that $\{X^x\}_{x\in L}=\{X_1^{a_1}\dots X_n^{a_n}\}_{a_1,\dots,a_n\in\Z}$. Then the algebra $\hH(B_n)$ is generated by elements 
$$g_0,g_1,\dots,g_{n-1},X_1^{\pm1},\dots,X_n^{\pm1}\,,$$
and it is easy to check that a set of defining relations is (\ref{rel-H0}) together with
\begin{equation}\label{rel-H}
\begin{array}{ll}
X_iX_j=X_jX_i & \text{for}\ i,j\in\{1,\dots,n\}\,,\\[0.2em]
g_0X_1^{-1}g_0=X_1\,, & \\[0.2em]
g_iX_ig_i=X_{i+1}\ \ \  & \text{for $i\in\{1,\dots,n-1\}$}\,,\\[0.2em]
g_iX_j=X_jg_i\ \ \  & \text{for $i\in\{0,\dots,n-1\}$ and $j\in\{1,\dots,n\}$ such that $j\neq i,i+1$}\,.
\end{array}
\end{equation} 
It is a standard fact that the following sets of elements are $K$-bases of $\hH(B_n)$:
\begin{equation}\label{base-aff}
\{X^xg_w\}_{x\in L,\,w\in W(B_n)}\ \ \ \ \text{and}\ \ \ \ \{g_wX^x\}_{x\in L,\,w\in W(B_n)} .
\end{equation}
It follows in particular that the (multiplicative) group formed by elements $X^x$, $x\in L$, of $\hH(B_n)$ can be identified with the (additive) group $L$. It follows also that the subalgebra generated by $g_0,\dots,g_{n-1}$ is isomorphic to the finite Hecke algebra $H(B_n)$.

\paragraph{Notation $\boldsymbol{{}^{r_a} P}$.} Let $b\in\{1,\dots,n-1\}$ and $P\in K[Z_1^{\pm1},\dots,Z_n^{\pm1}]$ a Laurent polynomial in the commuting invertible variables $Z_1,\dots,Z_n$. We note that the following formulas, 
\[{}^{r_0}P=P(Z_1^{-1},Z_2,\dots,Z_n)\ \ \ \ \text{and}\ \ \ \ {}^{r_b}P=P(Z_1,\dots,Z_{b+1},Z_b,\dots,Z_n)\]
define an action of the Weyl group $W(B_n)$ on $K[Z_1^{\pm1},\dots,Z_n^{\pm1}]$. With this notation, Formula (\ref{rel-Lu}) is equivalent to
\begin{equation}\label{rel-Lu'}
g_iP-{}^{r_i}Pg_i=(q_i-q_i^{-1})\frac{P-{}^{r_i}P}{1-X^{-\alpha_i}}\ ,\ \  \ \ \ \text{for $i=0,1,\dots,n-1$ and $P\in K[X_1^{\pm1},\dots,X_n^{\pm1}]$.}
\end{equation}

\paragraph{Affine Hecke algebra of type A.} The subalgebra of $\hH(B_n)$ generated by $g_1,\dots,g_{n-1},X_1^{\pm1},\dots,X^{\pm1}_n$ is an affine Hecke algebra of type A (associated to the root system $A_{n-1}$). This is the one most studied, corresponding to the reductive group $GL_n$. We denote it $\hH(A_n)$.

A presentation of the algebra $\hH(A_n)$ is given by generators (we keep the same names, this should not lead to any confusion)
\[g_1,\dots,g_{n-1},X_1^{\pm1},\dots,X_n^{\pm1}\ ,\]
satisfying the defining relations of $\hH(B_n)$ which do not involve the generator $g_0$.

\subsection{Intertwining elements}\label{subsec-int}

 For $i\in\{0,1\dots,n-1\}$, let
\begin{equation}\label{def-U}
U_i:=g_i(1-X^{-\alpha_i})-(q_i-q_i^{-1})\in\hH(B_n)\ .
\end{equation}
Using (\ref{rel-Lu}), alternative expressions for elements $U_i$ are
\[U_i=g_i^{-1}-g_iX^{-\alpha_i}=(1-X^{\alpha_i})g_i+(q_i-q_i^{-1})X^{\alpha_i}=g_i-X^{\alpha_i}g_i^{-1}\ .\]
It is well-known that the elements $U_i$ satisfy the following relations:
\begin{equation}\label{rel-Ui}
\begin{array}{ll}
U_iX^x=X^{r_i(x)}U_i & \text{for $i\in\{0,\dots,n-1\}$ and $x\in L$}\,,\\[0.2em]
U_0U_1U_0U_1=U_1U_0U_1U_0\,,\ \ \  & \\[0.2em]
U_iU_{i+1}U_i=U_{i+1}U_iU_{i+1}\ \ \  & \text{for $i\in\{1,\dots,n-2\}$}\,,\\[0.2em]
U_iU_j=U_jU_i\ \ \  & \text{for $i,j\in\{0,\dots,n-1\}$ such that $|i-j|>1$}\,,\\[0.2em]
U_i^2=(q_i-X^{\alpha_i}q_i^{-1})(q_i-X^{-\alpha_i}q_i^{-1}) & \text{for}\ i\in\{0,\dots,n-1\}\,.
\end{array}
\end{equation} 

Let $i\in\{0,1,\dots,n-1\}$ and $h\in\hH(B_n)$. It is easy to check using the basis in (\ref{base-aff}) of $\hH(B_n)$ that $(1-X^{-\alpha_i})h=0$ implies that $h=0$, and similarly for $h(1-X^{-\alpha_i})$. So let $\hH(B_n)^{loc}$ denote the localisation of the algebra $\hH(B_n)$ over the multiplicative set generated by the elements $1-X^{-\alpha_i}$ for $i=0,1,\dots,n-1$. Setting
\[U'_i:=U_i(1-X^{-\alpha_i})^{-1}\in\hH(B_n)^{loc}\ ,\]
one checks easily that the elements $U'_i$ satisfy the same relations as in (\ref{rel-Ui}) except the last one which now becomes
\[(U'_i)^2=\frac{(q_i-q_i^{-1}X^{\alpha_i})(q_i-q_i^{-1}X^{-\alpha_i})}{(1-X^{\alpha_i})(1-X^{-\alpha_i})}\ \ \ \ \text{for}\ i\in\{0,1\dots,n-1\}\ .\]

\subsection{Eigenvalues of $X_1,\dots,X_n$ in $\hH(B_n)$-modules}

Let $I$ be a finite subset $I\subset K\backslash\{0\}$ and, for $k\in\{1,\dots,n\}$, let $I_k=\{\lambda q^{2l}\ |\ \lambda\in I\,,\ \ 0\leq |l|\leq k-1\ \}$. Note in the following propositions that $M$ need not be finite-dimensional.

\begin{proposition}\label{prop-eig1}
Let $M$ be a $\hH(A_n)$-module, and let $P(X)=\displaystyle\prod_{\lambda \in I}(X-\lambda)^{m(\lambda)}$ for some $m(\lambda)\in\Z_{>0}$. Assume that $P(X_1)$ acts as $0$ on $M$.
Then, for $k\in\{1,\dots,n\}$, there is a polynomial of the form
\[Q_k(X)=\prod_{\mu\in I_k}(X-\mu)^{m_k(\mu)}\ ,\ \ \ \text{for some $m_k(\mu)\in\Z_{\geq0}$,}\]
such that $Q_k(X_k)$ acts as $0$ on $M$. In particular, as a $K[X_1^{\pm1},\dots,X_n^{\pm1}]$-module, we have
\[M=\bigoplus_{\textbf{i}\in I_1\times\dots\times I_n} M_{\textbf{i}}\ ,\]
where $M_{\textbf{i}}:=\{v\in M\ |\ \forall k=1,\dots,n\,,\ \bigl(X_k-i_k\bigr)^Nv=0\ \text{for some $N>0$}\}$  for $\textbf{i}=(i_1,\dots,i_n)$.
\end{proposition}
\begin{proof}
We use induction on $k$. Let $k\in\{1,\dots,n-1\}$ and assume that some polynomial $Q_k(X_k)=\prod_{\mu\in I_k}(X_k-\mu)^{m_k(\mu)}$ acts as $0$ on $M$ (this is true for $k=1$ by assumption). Then, as $X_k,X_{k+1}$ commute, the module $M$ decomposes for the action of $X_k$ and $X_{k+1}$ as
\[M=\bigoplus_{\mu\in I_k}\text{Ker}(X_k-\mu\text{Id}_M)^{m_k(\mu)}\ .\]
For $\mu\in I_k$, we set $P_{\mu}(X)=Q_k(X)(X-q^2\mu)(X-q^{-2}\mu)$ and we will show that $\bigl(P_{\mu}(X_{k+1})\bigr)^{m_k(\mu)}$ acts as $0$ on $\text{Ker}(X_k-\mu\text{Id}_M)^{m_k(\mu)}$.

First, let $v\in \text{Ker}(X_k-\mu\text{Id}_M)$. We have
\[P_{\mu}(X_{k+1})(v)=Q_k(X_{k+1})(X_{k+1}-q^2X_k)(X_{k+1}-q^{-2}X_k)(v)\ ,\]
since $X_k(v)=\mu v$ and $X_k$ and $X_{k+1}$ commute. We note that in $\hH(A_n)$, thanks to \eqref{rel-Ui}, we have
\[U_kQ_k(X_k)U_k=Q(X_{k+1})U_k^2=-Q(X_{k+1})(X_{k+1}-q^2X_k)(X-q^{-2}X_k)\ .\]
This shows that $P_{\mu}(X_{k+1})(v)$ is $0$ since $Q_k(X_k)$ acts as $0$ on $M$.

Then, let $v\in \text{Ker}(X_k-\mu\text{Id}_M)^j$ with $2\leq j\leq m_k(\mu)$. So we have $(X_k-\mu\text{Id}_M)^{j-1}(v)\in \text{Ker}(X_k-\mu\text{Id}_M)$ and therefore we have 
\[P_{\mu}(X_{k+1})(X_k-\mu\text{Id}_M)^{j-1}(v)=0\ ,\]
using the preceding step. As $X_k$ and $X_{k+1}$ commute, we find that $P_{\mu}(X_{k+1})(v)\in \text{Ker}(X_k-\mu\text{Id}_M)^{j-1}$ and, using induction on $j$, this yields $\bigl(P_{\mu}(X_{k+1})\bigr)^{j-1}P_{\mu}(X_{k+1})(v)=0$ which is the desired result.

The rest of the proposition is immediate.
\end{proof}

For a $\hH(B_n)$-module, we can use the presence of $g_0$ to obtain further information on the minimal polynomial of $X_1$.
\begin{proposition}\label{prop-eig2}
Let $M$ be a $\hH(B_n)$-module, and let $P(X)=\displaystyle\prod_{\lambda \in I}(X-\lambda)^{m(\lambda)}$ for some finite subset $I\subset K\backslash\{0\}$ and some $m(\lambda)\in\Z_{>0}$. Assume that $P(X_1)$ acts as $0$ on $M$. Then the minimal polynomial of $X_1$ on $M$ is of the form
\[Q(X)=\prod_{\lambda\in I}(X-\lambda)^{m'(\lambda)}\ ,\ \ \ \text{with}\ \ \left\{\begin{array}{ll}
m'(\lambda^{-1})=m'(\lambda) & \text{if $\lambda\notin\{\pm p^{\pm1}\}$;}\\[0.4em]
|m'(\lambda^{-1})-m'(\lambda)|\leq 1 & \text{if $\lambda\in\{\pm p^{\pm1}\}$.}
\end{array}\right. \]
\end{proposition}
\begin{proof}
In $\hH(B_n)$, we have
\[U_0P(X_1)U_0=U_0^2P(X_1^{-1})=(p-p^{-1}X_1^2)(p-p^{-1}X_1^{-2})\prod_{\lambda \in I}(X_1^{-1}-\lambda)^{m(\lambda)}\ .\]
This element must act as $0$ on $M$ since $P(X_1)$ acts as $0$ on $M$. Multiplying by the invertible element $p^2X_1^2\prod_{\lambda\in I}(\lambda^{-1}X_1)^{m(\lambda)}$, we get that the minimal polynomial of $X_1$ on $M$ divides the following polynomial
\[(X^2-p^2)(X^2-p^{-2})\prod_{\lambda \in I}(X-\lambda^{-1})^{m(\lambda)}\ .\]
Taking for $P$ the minimal polynomial of $X_1$ on $M$, this proves the proposition.
\end{proof}

\section{Cyclotomic quotients of $\hH(B_n)$ and $e_{\beta}H(B_n)_{\blambda,m}$}\label{sec-cyc}

\subsection{Definition}

For any $\lambda\in K\setminus\{0\}$, let $I_{\lambda}=\{\lambda^{\epsilon}q^{2l}\ |\ \epsilon\in\{-1,1\}\,,\ \ l\in\Z\}$. Fix an $l$-tuple $\bl=(\lambda_1,\dots,\lambda_l)$ of elements of $K\setminus\{0\}$ such that we have $I_{\lambda_a}\cap I_{\lambda_b}=\emptyset$ if $a\neq b$. Then fix a multiplicity map
\begin{equation}\label{def-m}
m\ :\ I_{\lambda_1}\cup\dots\cup I_{\lambda_l}\to \Z_{\geq0}\ ,
\end{equation}
such that only a finite number of multiplicities $m(i)$ are different from $0$.

We define the cyclotomic quotient denoted $H(B_n)_{\bl,m}$ to be the quotient of the algebra $\hH(B_n)$ over the relation
\begin{equation}\label{cyc-rel}
\prod_{i\in I_{\lambda_1}\cup\dots\cup I_{\lambda_l}}(X_1-i)^{m(i)}=0\ .
\end{equation}
\begin{remark}
\begin{itemize} \item It is explained in \cite{EK} that for the study of irreducible representations of $\hH(B_n)$, one can assume that the elements $X_1,\dots,X_n$ have all their eigenvalues included in a single set of the form $I_{\lambda}$, for an arbitrary $\lambda$. Thus, for the study of irreducible representations of $\hH(B_n)$, it would be enough to consider the particular situation $l=1$, namely where $\bl$ is a single $\lambda$.

\item Proposition  \ref{prop-eig1} has the following consequence. All the eigenvalues of $X_1,\dots,X_n$ are included in $I_{\lambda_1}\cup\dots\cup I_{\lambda_l}$. More precisely, if the eigenvalues of $X_1$ are included in $\{\lambda_a^{\epsilon}q^{2N}\ |\ a\in\{1,\dots,l\}\,,\ \epsilon\in\{-1,1\}\,,\ 0\leq |N|\leq L\}$ then all the eigenvalues of $X_1,\dots,X_n$ are included in the finite subset $\{\lambda_a^{\epsilon}q^{2N}\ |\ a\in\{1,\dots,l\}\,,\ \epsilon\in\{-1,1\}\,,\ 0\leq |N|\leq L+n-1\}$.

\item Proposition \ref{prop-eig2} has the following consequence; the cyclotomic quotient $H(B_n)_{\bl,m}$ is isomorphic to the cyclotomic quotient $H(B_n)_{\bl,m'}$, where the multiplicity map $m'$ is given as follows:
\[m'(i)=\left\{\begin{array}{ll}
\text{min}\bigl(m(i),m(i^{-1})\bigr) & \text{if $i\notin\{\pm p^{\pm1}\}$;}\\[0.4em]
\text{min}\bigl(m(i),m(i^{-1})+1\bigr) & \text{if $i\in\{\pm p^{\pm1}\}$.}
\end{array}\right.\]

\item We have that $H(B_n)_{\blambda,m}$ is of finite dimension. Indeed, recall that the set
\[\{X_1^{a_1}\dots X_n^{a_n}g_w\}_{a_1,\dots,a_n\in\Z,\,w\in B_n}\]
is a basis of $\hH(B_n)$ and therefore linearly spans the quotient $H(B_n)_{\blambda,m}$. Moreover,
as an immediate consequence of Proposition \ref{prop-eig1}, we have that each of the commuting elements $X_1,\dots,X_n$ satisfies a relation $P_k(X_k)=0$ for a polynomial $P_k\in K[X]$. So we can extract from the above spanning set a finite spanning subset (by taking $a_k\in\{0,\dots,\text{deg}(P_k)\}$).\hfill$\triangle$
\end{itemize}
\end{remark}

\subsection{Idempotents and blocks of $H(B_n)_{\blambda,m}$}\label{subsec-idem}

For $\textbf{i}=(i_1,\dots,i_n)\in \bigl(I_{\lambda_1}\cup\dots\cup I_{\lambda_l}\bigr)^n$, let 
$$M_{\textbf{i}}:=\{x\in H(B_n)_{\blambda,m}\ |\ \bigl(X_k-i_k\bigr)^Nx=0\ \ \text{for $k=1,\dots,n$ and some $N>0$}\}\ .$$
The subspace $M_{\textbf{i}}$ is a common generalised eigenspace for the action of the commuting elements $X_1,\dots,X_n$ by left multiplication on $H(B_n)_{\blambda,m}$. As a representation of the subalgebra generated by $X_1,\dots,X_n$, we have
\[H(B_n)_{\blambda,m}=\bigoplus_{\textbf{i}\in \bigl(I_{\lambda_1}\cup\dots\cup I_{\lambda_l}\bigr)^n} M_{\textbf{i}}\ ,\]
and there is a finite number of non-zero $M_{\textbf{i}}$.
Let 
\begin{equation}\label{def-ei}
\{e^H_{\textbf{i}}\}_{\textbf{i}\in \bigl(I_{\lambda_1}\cup\dots\cup I_{\lambda_l}\bigr)^n}\ ,
\end{equation}
be the associated set of mutually orthogonal idempotents. By definition, the idempotents $e^H_{\textbf{i}}$ belong to the commutative subalgebra generated by $X_1,\dots,X_n$.

\paragraph{Central idempotent and blocks.} Let $\textbf{i}=(i_1,\dots,i_n)\in \bigl(K\setminus\{0\}\bigr)^n$. It can be seen as a character $X_k\mapsto i_k$ of $K[X_1^{\pm1},\dots,X_n^{\pm1}]$. The Weyl group acts on $K[X_1^{\pm1},\dots,X_n^{\pm1}]$, and therefore on its characters. The action of the generators $r_0,r_1,\dots,r_{n-1}$ of $W(B_n)$ is given explicitly by
\[r_0(\textbf{i})=(i_1^{-1},i_2,\dots,i_n)\ \ \quad \text{and}\quad \ \ r_a(\textbf{i})=(i_1,\dots,i_{a+1},i_a,\dots,i_n)\,,\ \ a=1,\dots,n-1.\]
Let $\beta$ be an orbit of the action of $W(B_n)$ on $\bigl(I_{\lambda_1}\cup\dots\cup I_{\lambda_l}\bigr)^n$. We set
\begin{equation}\label{def-ea}
e_{\beta}=\sum_{\textbf{i}\in\beta}e^H(\textbf{i})\ .
\end{equation}
Thus, the element $e_{\beta}$ is a central idempotent of $H(B_n)_{\blambda,m}$. The set $e_{\beta}H(B_n)_{\blambda,m}$ is therefore either $\{0\}$, or a subalgebra (with unit $e_{\beta}$) which is a union of blocks of $H(B_n)_{\blambda,m}$.

\section{An intermediary presentation of $e_{\beta}H(B_n)_{\blambda,m}$}\label{sec-inter}

We define a new algebra $\cI_N$ in this section, associated to the choice of $\beta$, $\blambda$ and $m$, which will play the role of an intermediate step in our proof of the main result (Theorem \ref{theo}). Indeed we will show that it is isomorphic, on the one hand, to the algebra $V^{\beta}_{\bl,m}$ and, on the other hand, to the algebra $e_{\beta}H(B_n)_{\blambda,m}$. 

In the following definition, we make the convenient abuse of notation consisting of denoting again by 
$X_1^{\pm1},\dots,X_n^{\pm1}$ some of the generators of the algebra $\cI_N$; as they commute, elements $X^x$, for $x\in L$, are defined as well as in $\hH(B_n)$.

Let $Y,Z\in K[X^{\pm1}_1,\dots,X^{\pm1}_n]$ and $\textbf{i}\in\beta$. In the algebra $\cI_N$ defined below, we will use the following notation:
\[Y^{-1}e(\ti):=Ze(\ti)\ \ \ \ \quad \text{if $YZe(\ti)=e(\ti)$}\,,\]
that is, if $Ye(\ti)$ is invertible in the subalgebra $K[X^{\pm1}_1,\dots,X^{\pm1}_n]e(\ti)$ (with unit $e(\ti)$) of $\mathcal{I}_N$. For example, due to the defining relations (\ref{rel1-X1})-(\ref{rel1-Xk}), the  elements $(1-X^{-\alpha_a})^{-1}e(\ti)$ are defined if $r_a(\ti)\neq \ti$ (see Remark \ref{rem-defI2} below for more precision on the denominators appearing in the defining formulas).
\begin{definition}\label{defI}
Let $N>0$ be an integer. Let $\cI_N$ be the algebra generated by elements
\[\Phi_0,\Phi_1,\dots,\Phi_{n-1},\ X_1^{\pm1},\dots,X_n^{\pm1},\ \ e(\ti)\,, \ti\in\beta\ ,\]
with the following defining relations:
\begin{equation}\label{rel1-e}
\sum_{\ti\in\beta}e(\ti) =1\,,\ \ \ \ \ \  e(\ti)e(\textbf{j})=\delta_{\ti,\textbf{j}}e(\ti)\,,\ \forall \ti,\textbf{j}\in\beta\,,
\end{equation}
\begin{equation}\label{rel1-X}
X_iX_j=X_jX_i\,,\ \ \ \ \ \  X_ie(\ti)=e(\ti)X_i\,,\ \forall i,j\in\{1,\dots,n\}\ \text{and}\ \forall \ti\in\beta\,,
\end{equation}
and, for $a,b\in\{0,1,\dots,n-1\}$ with $b\neq 0$, for $x\in L$, and for $\ti=(i_1,\dots,i_n)\in\beta$,
\begin{empheq}{alignat=2}
(X_1-i_1)^{m(i_1)}e(\ti) & = 0\ , \label{rel1-X1}\\[0.5em]
(X_k-i_k)^{N}e(\ti) & = 0\ , \qquad  \text{for $k=2,\dots,n$,}\label{rel1-Xk}\\[0.5em]
\Phi_a e(\ti)& =  e(r_a(\ti))\Phi_a\ ,\label{rel1-Phi1}\\[0.5em]
(\Phi_a X^x-X^{r_a(x)}\Phi_a) e(\ti) & = \left\{\begin{array}{ll}
0 & \ \ \text{if $r_a(\ti)\neq\ti$,}\\[0.2em]
(q_a-q_a^{-1} X^{-\alpha_a})\displaystyle\frac{X^x-X^{r_a(x)}}{1-X^{-\alpha_a}}e(\ti) &\ \ \text{if $r_a(\ti)=\ti$,}
\end{array}\right. \label{rel1-Phi2}\\[0.5em]
\Phi_a\Phi_b & =\Phi_b\Phi_a  \qquad \text{if $|a-b|>1$\ ,}\label{rel1-Phi3}\\[0.5em]
\Phi^2_a e(\ti) & = \left\{\begin{array}{ll}
\Gamma_a e(\ti) & \ \ \text{if $r_a(\ti)\neq\ti$,}\\[0.4em]
(q_a+q_a^{-1})\Phi_ae(\ti) &\ \ \text{if $r_a(\ti)=\ti$,}
\end{array}\right.\label{rel1-Phi4}
\end{empheq}
where $\Gamma_a=\displaystyle\frac{(q_a-q_a^{-1}X^{\alpha_a})(q_a-q_a^{-1}X^{-\alpha_a})}{(1-X^{\alpha_a})(1-X^{-\alpha_a})}=1+\frac{(q_a-q_a^{-1})^2}{(1-X^{\alpha_a})(1-X^{-\alpha_a})}\ $;

\begin{empheq}{alignat=2}
\bigl(\Phi_b\Phi_{b+1}\Phi_b-\Phi_{b+1}\Phi_b\Phi_{b+1}\bigr) e(\ti) & = 
\left\{\begin{array}{ll}
\bigl(\Phi_b-\Phi_{b+1}\bigr)e(\ti) & \ \ \text{if $i_b=i_{b+1}=i_{b+2}$,}\\[0.4em]
Z_b e(\ti) & \ \ \text{if $i_b=i_{b+2}\neq i_{b+1}$,}\\[0.4em]
0 & \ \ \text{otherwise,}
\end{array}\right.\label{rel1-Phi5}
\end{empheq}
where $Z_b =\displaystyle(\Gamma_b-\Gamma_{b+1})\frac{q-q^{-1} X^{-\alpha_b-\alpha_{b+1}}}{1-X^{-\alpha_b-\alpha_{b+1}}}\ $;

\begin{empheq}{alignat=2}
\bigl(\Phi_0\Phi_1\Phi_0\Phi_1-\Phi_1\Phi_0\Phi_1\Phi_0\bigr) e(\ti) & = 
\left\{\begin{array}{ll}
(pq^{-1}+p^{-1}q)\bigl(\Phi_0\Phi_1-\Phi_1\Phi_0\bigr)e(\ti) & \ \ \text{case 1}\\[0.6em]
\Bigl(\displaystyle\Phi_0(1-X^{-\alpha_0})-(p-p^{-1}X^{-\alpha_0})\Bigr)Y_0e(\ti) & \ \ \text{case 2}\\[0.8em]
\Phi_0(1-X^{-\alpha_0}) Y_0e(\ti) & \ \ \text{case 3}\\[0.6em]
-\Phi_1(1-X^{-\alpha_1}) Y_1 e(\ti) & \ \ \text{case 4}\\[0.4em]
0 & \ \ \text{otherwise,}
\end{array}\right.\label{rel1-Phi6}
\end{empheq}
where 
$$Y_0 =\displaystyle(\Gamma_1-{}^{r_0}\Gamma_1)\frac{p-p^{-1}X^{-r_1(\alpha_0)}}{(1-X^{-\alpha_0})(1-X^{-r_1(\alpha_0)})}\,,\ \ \ \ \ Y_1 =\displaystyle(\Gamma_0-{}^{r_1}\Gamma_0)\frac{q-q^{-1}X^{-r_0(\alpha_1)}}{(1-X^{-\alpha_1})(1-X^{-r_0(\alpha_1)})}\,,$$
and the different cases for Relation (\ref{rel1-Phi6}) are listed as follows:
\begin{itemize}
\item Case 1 : $i_1=i_2$ and $i_1^2=i_2^2=1$;
\item Case 2 : $i_1\neq i_2$ and $i_1^2=i_2^2=1$;
\item Case 3 : $i_1\neq i_2$ and $i_1^2\neq 1$ and $i_2^2=1$;
\item Case 4 : $i_1\neq i_2$ and $i_1^2\neq1$ and $i_1^{-1}=i_2$.
\end{itemize}
\end{definition}

\begin{remark} The integer $N$ appears only in the defining relation (\ref{rel1-Xk}). We will mainly consider the algebra $\cI_N$ for large $N$, for which it will be isomorphic to both of the algebras $e_{\beta}H(B_n)_{\blambda,m}$ and $V^{\beta}_{\bl,m}$, which do not depend on $N$. In particular, this will show that for $N$ large enough, the algebra $\cI_N$ does not depend on $N$ and Relation (\ref{rel1-Xk}) is implied by the others. Nevertheless, we consider it as a defining relation for technical convenience.\hfill$\triangle$
\end{remark}

\begin{remark}\label{rem-defI2}
Let us justify the claim that the defining relations involve only well-defined elements of $\cI_N$, even though some denominators appear. We have to explain that some elements $Y^{-1}e(\ti)$ are defined in the sense recalled just before the definition. This will be a consequence of Relations (\ref{rel1-X1})-(\ref{rel1-Xk}). First note that straightforward calculations give:
\begin{equation}\label{Z-Y}
\begin{array}{l} Z_b = \displaystyle(q-q^{-1})^2\frac{(X^{-\alpha_{b+1}}-X^{-\alpha_{b}})(q-q^{-1}X^{-\alpha_b-\alpha_{b+1}})}{(1-X^{-\alpha_b})^2(1-X^{-\alpha_{b+1}})^2}\,,\\[1em]
Y_0 =-\displaystyle(q-q^{-1})^2\frac{X^{-\alpha_1}(p-p^{-1}X^{-r_1(\alpha_0)})}{(1-X^{-\alpha_1})^2(1-X^{-r_0(\alpha_1)})^2}\,,\\[1em]
Y_1 =-(p-p^{-1})^2\displaystyle\frac{X^{-\alpha_0}(1+X^{-\alpha_1})(1+X^{-r_0(\alpha_1})(q-q^{-1}X^{-r_0(\alpha_1)})}{(1-X^{-\alpha_0})^2(1-X^{-r_1(\alpha_0)})^2}\,.
\end{array}
\end{equation}
Then we have:
\begin{itemize}
\item[-] $\displaystyle\frac{X^x-X^{r_a(x)}}{1-X^{-\alpha_a}}\in K[X_1^{\pm1},\dots,X_n^{\pm1}]$ since $r_a(x)=x-k\alpha_a$ for some $k\in\mathbb{Z}$.
\item[-] $\Gamma_a e(\ti)\in \cI_N$ if $r_a(\ti)\neq\ti$ since $(1-X^{-\alpha_a})^{-1}e(\ti)$ is defined in this case.
\item[-] Assume that $i_b=i_{b+2}\neq i_{b+1}$. Then (\ref{Z-Y}) shows that $Z_be(\ti)\in\cI_N$ since $(1-X^{-\alpha_b})^{-1}e(\ti)$ and $(1-X^{-\alpha_{b+1}})^{-1}e(\ti)$ are defined.
\item[-] Assume that $i_1\neq i_{2}$ and $i_1^{-1}\neq i_2$. Then (\ref{Z-Y}) shows that $Y_0e(\ti)\in\cI_N$ since $(1-X^{-\alpha_1})^{-1}e(\ti)$ and $(1-X^{-r_0(\alpha_{1})})^{-1}e(\ti)=(1-X_1X_2)^{-1}e(\ti)$ are defined.
\item[-] Assume that $i_1^2\neq 1$ and $i_2^{2}\neq 1$. Then (\ref{Z-Y}) shows that $Y_1e(\ti)\in\cI_N$ since $(1-X^{-\alpha_0})^{-1}e(\ti)$ and $(1-X^{-r_1(\alpha_{0})})^{-1}e(\ti)=(1-X_2^{-2})^{-1}e(\ti)$ are defined. \hfill$\triangle$
\end{itemize}
\end{remark}

\section{A family of isomorphisms in type A}\label{sec-A}

\paragraph{Cyclotomic quotient of $\hH(A_n)$.} Let the notation be as in Section \ref{sec-cyc}. We define the cyclotomic quotient, denoted $H(A_n)_{\blambda,m}$, to be the quotient of the algebra $\hH(A_n)$ over the relation
\begin{equation}\label{cyc-relA}
\prod_{i\in I_{\lambda_1}\cup\dots\cup I_{\lambda_l}}(X_1-i)^{m(i)}=0\ .
\end{equation}
Exactly as in Section \ref{sec-cyc}, we consider the set of mutually orthogonal idempotents\begin{equation}\label{def-ei-A}
\{e^H_{\textbf{i}}\}_{\textbf{i}\in \bigl(I_{\lambda_1}\cup\dots\cup I_{\lambda_l}\bigr)^n}\subset H(A_n)_{\blambda,m}\ ,
\end{equation}
associated to the decomposition
\[H(A_n)_{\blambda,m}=\bigoplus_{\textbf{i}\in \bigl(I_{\lambda_1}\cup\dots\cup I_{\lambda_l}\bigr)^n} M_{\textbf{i}}\ ,\]
for the action of $X_1,\dots,X_n$ by left multiplication, where, for $\textbf{i}=(i_1,\dots,i_n)\in \bigl(I_{\lambda_1}\cup\dots\cup I_{\lambda_l}\bigr)^n$,
$$M_{\textbf{i}}:=\{x\in H(A_n)_{\blambda,m}\ |\ \bigl(X_k-i_k\bigr)^Nx=0\ \ \text{for $k=1,\dots,n$ and some $N>0$}\}\ .$$

\paragraph{Central idempotent and blocks.} Let $\alpha$ be a finite union of orbits for the action of the symmetric group $\mathfrak{S}_n$ on $\bigl(I_{\lambda_1}\cup\dots\cup I_{\lambda_l}\bigr)^n$. We set
\begin{equation}\label{def-eb}
e_{\alpha}=\sum_{\textbf{i}\in\alpha}e^H(\textbf{i})\in H(A_n)_{\blambda,m}\ .
\end{equation}
The element $e_{\alpha}$ is a central idempotent of $H(A_n)_{\blambda,m}$.

\begin{remark} If we take $l=1$, $\lambda_1=1$ and $\alpha$ a single orbit, then $e_{\alpha}H(A_n)_{\blambda,m}$ is a single block (if non-zero) of $H(A_n)_{\blambda,m}$ as considered in \cite{BK}. We need the more general situation, in particular the situation in which $\alpha$ is an orbit for the action of the Weyl group $W(B_n)$ on $\bigl(I_{\lambda_1}\cup\dots\cup I_{\lambda_l}\bigr)^n$.\hfill$\triangle$
\end{remark}

\subsection{Generalisation of the Brundan--Kleshchev isomorphism}\label{sec-proofA}

Let $\tR^{\alpha}_{\bl,m}$ be the algebra generated by elements
\[\{\psi_1,\dots,\psi_{n-1}\}\cup\{y_1,\dots,y_n\}\cup\{e(\ti)\,,\ \ \ti\in \alpha\}\ ,\]
with the same defining relations as for the algebra $V^{\beta}_{\bl,m}$, replacing $\beta$ by $\alpha$ and removing the relations involving $\psi_0$.

Let $N>0$ and let $\cI_{N,A}$  be the algebra generated by elements
\[\{\Phi_1,\dots,\Phi_{n-1}\}\cup\{X^{\pm1}_1,\dots,X^{\pm1}_n\}\cup\{e(\ti)\,,\ \ \ti\in \alpha\}\ ,\]
with the same defining relations as for the algebra $\cI_N$, replacing $\beta$ by $\alpha$ and removing the relations involving $\Phi_0$. The fact that some generators of different algebras share the same name should not lead to any confusion.

\begin{theorem}\label{theoA}
The algebra $e_{\alpha}H(A_n)_{\blambda,m}$ is isomorphic to $\tR^{\alpha}_{\blambda,m}$. In fact we have:
\begin{enumerate}
\item The algebra $\tR^{\alpha}_{\blambda,m}$ is isomorphic to the algebra $\cI_{N,A}$ for $N$ large enough.
\item The algebra $e_{\alpha}H(A_n)_{\blambda,m}$ is isomorphic to the algebra $\cI_{N,A}$ for $N$ large enough.
\end{enumerate}
\end{theorem}

\begin{remark}\label{rem-BK}
This Theorem is proved in \cite{BK} if we assume that $\alpha$ is a single orbit of the symmetric group and that $\bl=(\lambda_1)$ with $\lambda_1=1$. In fact the proof in \cite{BK} can be repeated with very minor modifications in our more general setting (this is done in \cite{Ro2}). 

However, we will give a different isomorphism to that given in \cite{BK}. In fact, we provide a whole family of isomorphic maps, indexed by a formal power series $f$ in one variable with zero constant term and non-zero coefficient in degree 1. The one used in \cite{BK} is simply $f(z)=z$. It turns out that we will need a different one for generalising this theorem to the type B situation. This is why we need to provide the details of the proof which, in any case, will be useful for the proof of our main theorem. We remark that the approach of the proof is largely inspired by the proof in \cite{BK}.
\hfill$\triangle$
\end{remark}

The remainder of this section is devoted to the proof of the theorem.

\vskip .2cm
\noindent We start by fixing an integer $N$ large enough. In $e_{\alpha}H(A_n)_{\blambda,m}$, by definition of the idempotents $e^H(\ti)$, the elements $(X_k-i_k)e^H(\ti)$ are nilpotent for any $k=1,\dots,n$ and any $\ti\in\alpha$. We assume $N$ to be larger than all the nilpotency indices of these elements.

In $\tR^{\alpha}_{\bl,m}$, it is known that the elements $y_ke(\ti)$ are nilpotent for all $k=1,\dots,n$ and all $\ti\in\alpha$ (see Lemma \ref{lem-nily}, or directly \cite[Lemma 2.1]{BK}). We also assume $N$ to be larger than all the nilpotency indices of these elements.

\paragraph{Proof of item 1.} First, we fix a formal power series $f=\sum_i a_iz^i\in K[[z]]$ such that $a_0=0$ and $a_1\neq 0$. Thus there exists a formal power series denoted $g=\sum_i b_iz^i\in K[[z]]$ with $b_0=0$ and $b_1\neq 0$ satisfying
$$f\bigl(g(z)\bigr)=g\bigl(f(z)\bigr)=z\ .$$

We now label the generators of both $\tR^{\alpha}_{\blambda,m}$ and $\cI_{N,A}$ in a way in which, at first, will appear redundant but which will be needed for clarity later on. We set
\begin{equation*}
y_k^R:=y_k\,,\ \ \psi_s^R:=\psi_s\ \ \ \text{in $\tR^{\alpha}_{\blambda,m}$,}\ \ \ \qquad\text{and}\ \ \ \qquad X_k^{\cI}:=X_k\,,\ \ \Phi_s^{\cI}:=\Phi_s \ \ \text{in $\cI_{N,A}$.}
\end{equation*}
for all $k=1,\ldots, n$ and all $s=1,\ldots,n-1$.

Let $k\in\{1,\dots,n\}$. We define elements of $\cI_{N,A}$:
\[y^{\cI}_k:=\sum_{\textbf{i} \in \alpha}f\bigl(1-i_k^{-1}X_k\bigr)\idmep\in\cI_{N,A}\ .\]
These elements are well-defined since $(1-i_k^{-1}X_k)\idmep$ is nilpotent for any $\ti\in\alpha$ thanks to the defining relation (\ref{rel1-Xk}) in $\cI_{N,A}$. Moreover, the elements $y^{\cI}_k$ are nilpotent since the formal power series $f$ has no constant term and they commute.

Note that by definition of $f$ and $g$, we have then that $X_k=\sum_{\textbf{i} \in \alpha}i_k\bigl(1-g(y^{\cI}_k)\bigr)e(\ti)$ in $\cI_{N,A}$. Therefore, conversely, we define elements of $\tR^{\alpha}_{\blambda,m}$;
\[X_k^R:=\sum_{\textbf{i} \in \alpha}i_k\bigl(1-g(y_k)\bigr)\idmep\in\tR^{\alpha}_{\blambda,m}\ .\]
These elements are well-defined since the elements $y_k$ are nilpotent. Moreover, the elements $X_k^R$ are invertible since $g$ has no constant term and they commute.

Now we introduce a symbol $\some$ which we will use as an upper index to denote both $R$ and $\cI$. Let $r\in\{1,\dots,n\}$, $\ti\in\alpha$ and $P$ any rational expression in $X_1,\dots,X_n$. We set
\begin{equation}\label{nota1}
X^{\some}_k(\ti)=X^{\some}_k(i_k):=i_k\bigl(1-g(y^{\some}_k)\bigr)\ \ \ \ \text{and}\ \ \ \ P^{\some}(\ti)=P\bigl(X^{\some}_1(i_1),\dots,X^{\some}_n(i_n)\bigr)\ .
\end{equation}
Note that if $P^{\some}$ is a rational expression in $X^{\some}_1,\dots,X^{\some}_n$ such that $P^{\some}e(\ti)$ is defined (we recall again that this means that the denominator of $P^{\some}$ is invertible in the subalgebra $K[X^{\some}_1,\dots,X^{\some}_n]e(\ti)$) then, by construction, we have $P^{\some}e(\ti)=P^{\some}(\ti)e(\ti)$, where $P^{\some}(\ti)$ is a formal power series in $y^{\some}_1,\dots,y^{\some}_n$. For example, we have $X^{\some}_ke(\ti)=X^{\some}_k(\ti)\idmep$.

\vskip .2cm
Then we fix a family of invertible power series $Q_k(i,j) \in K[[z, z']]$, where $k=1,\dots,n-1$ and $(i,j)\in K^2$, and we set for $\some \in \{ R, \cI \}$:
$$Q^{\some}_k(i,j):=Q_k(i,j)\bigl(y^{\some}_k,y^{\some}_{k+1}\bigr) \in K[[y^{\some}_k,y^{\some}_{k+1}]]\ .$$
We will often use the following notation:
\[Q^{\some}_k(\ti)=Q^{\some}_k(i_k,i_{k+1})\ \ \ \quad\text{and}\quad\ \ \ \tQ^{\some}_k(\textbf{i})=\tQ^{\some}_k(i_k,i_{k+1})=\bigl(Q^{\some}_k(\textbf{i})\bigr)^{-1}\ .\]

We assume that the following properties are satisfied (where $\Gamma_k$ is defined in Definition \ref{defI} and $r_k$ acts on formal power series in $y^{\some}_1,\dots,y^{\some}_n$ by exchanging $y^{\some}_k$ and $y^{\some}_{k+1}$), for any $\ti\in\alpha$ and $k=1,\dots,n-1$,
\begin{equation}\label{eq-Q1}
Q_k^{\some}(\ti)=i_k^{-1}\big( q^{-1}X^{\some}_k(\ti)-qX^{\some}_{k+1}(\ti) \big)\displaystyle\frac{y^{\some}_{k+1}-y^{\some}_k}{g(y^{\some}_{k+1})-g(y^{\some}_k)}\ \ \ \ \ \textrm{ if } i_k=i_{k+1}\,,
\end{equation}
\begin{equation}\label{eq-Q2}
{^{r_k}}Q^{\some}_k\bigl(r_k(\ti)\bigr)Q^{\some}_k(\ti)=\left\lbrace
\begin{array}{ll}
\Gamma_k^{\some}(\ti) & \textrm{ if }i_k \nleftrightarrow i_{k+1} \vspace*{0.6em}\\
\Gamma_k^{\some}(\ti)\big( y_{k+1}^{\some}-y_k^{\some} \big)^{-1} & \textrm{ if }i_k \leftarrow i_{k+1} \vspace*{0.6em}\\
\Gamma_k^{\some}(\ti)\big( y_{k}^{\some}-y_{k+1}^{\some} \big)^{-1} & \textrm{ if }i_k \rightarrow i_{k+1} \vspace*{0.6em}\\
\Gamma_k^{\some}(\ti)\big( y_{k+1}^{\some}-y_k^{\some} \big)^{-1}\big( y_k^{\some}-y_{k+1}^{\some} \big)^{-1} & \textrm{ if }i_k \leftrightarrow i_{k+1} 
\end{array}
\right.
\end{equation}
and
\begin{equation}\label{eq-Q3}
{^{r_k}}Q^{\some}_{k+1}\bigl(r_{k+1}r_k(\textbf{i})\bigr)={^{r_{k+1}}}Q^{\some}_k\bigl(r_kr_{k+1}(\textbf{i})\bigr)\ .
\end{equation}

\begin{lemma}\label{lem-QA}
There exists such a family satisfying Conditions (\ref{eq-Q1})--(\ref{eq-Q3}).
\end{lemma}
\begin{proof} First, we note that $y^{\some}_{k+1}-y^{\some}_k$ divides $g(y^{\some}_{k+1})-g(y^{\some}_k)$ and moreover $\displaystyle\frac{g(y^{\some}_{k+1})-g(y^{\some}_k)}{y^{\some}_{k+1}-y^{\some}_k}$ is an invertible power series since $g$ has a non-zero coefficient in degree 1. Then it is easy to see that the formulas
\begin{equation}\label{formulaQk}
Q_k^{\some}(\textbf{i}):=\left\lbrace
\begin{array}{ll}
i_k^{-1}\big( q^{-1}X^{\some}_k(\ti)-qX^{\some}_{k+1}(\ti) \big)\displaystyle\frac{y^{\some}_{k+1}-y^{\some}_k}{g(y^{\some}_{k+1})-g(y^{\some}_k)} & \textrm{ when } i_k=i_{k+1} \vspace*{1em}\\
\displaystyle\frac{q X^{\some}_k(\ti)-q^{-1}X^{\some}_{k+1}(\ti)}{X^{\some}_k(\ti)-X^{\some}_{k+1}(\ti)} & \textrm{ when } i_{k+1}\nleftrightarrow i_k \vspace*{1em}\\
\displaystyle\frac{(q X^{\some}_{k+1}(\ti)-q^{-1}X^{\some}_{k}(\ti))\,q\,i_k}{(X^{\some}_k(\ti)-X^{\some}_{k+1}(\ti))(X^{\some}_{k+1}(\ti)-X^{\some}_{k}(\ti))}\cdot\frac{g(y^{\some}_{k+1})-g(y^{\some}_k)}{y^{\some}_{k+1}-y^{\some}_k} & \textrm{ when }i_k \rightarrow i_{k+1} \vspace*{1em}\\
1 & \textrm{ when }i_k \leftarrow i_{k+1} \vspace*{1em}\\
\displaystyle\frac{qi_k}{X^{\some}_k(\ti)-X^{\some}_{k+1}(\ti)}\cdot\frac{g(y^{\some}_{k+1})-g(y^{\some}_k)}{y^{\some}_{k+1}-y^{\some}_k} & \textrm{ when }i_k \leftrightarrow i_{k+1}
\end{array}\right.
\end{equation}
indeed define invertible power series (here we use that $q^2\neq 1$), where the dependence on $\ti$ of $Q_k^{\some}(\textbf{i})$ is only through $i_k$ and $i_{k+1}$. Moreover, one sees immediately that
$$\begin{array}{ll}
q X^{\some}_k(\ti)-q^{-1}X^{\some}_{k+1}(\ti)=q i_k \bigl(g(y_{k+1})-g(y_{k})\bigr) \quad & \text{if $i_{k+1}=q^2 i_k$,}\\[0.5em]
q X^{\some}_{k+1}(\ti)-q^{-1}X^{\some}_{k}(\ti)=q i_{k+1} \bigl(g(y_k)-g(y_{k+1})\bigr) \quad & \text{if $i_{k+1}=q^{-2} i_k$.}
\end{array}$$
So, using the defining formula in Definition \ref{defI} for $\Gamma_k$, it is now easy to see that the properties (\ref{eq-Q1}) and (\ref{eq-Q2}) are satisfied. The final property (\ref{eq-Q3}) follows immediately once one notices that:\\
$-$ in $r_{k+1}r_k(\textbf{i})$, we have $(i_{k+2},i_k)$ in positions $(k+1,k+2)$, and we have $(i_{k+2},i_k)$ as well in positions $(k,k+1)$ in $r_{k}r_{k+1}(\textbf{i})$;\\
$-$ the pair $(y_{k+1},y_{k+2})$ is sent by $r_k$ to $(y_{k},y_{k+2})$ which is the same as the image of $(y_k,y_{k+1})$ under the action of $r_{k+1}$.
\end{proof}

Now we set, for $k=1,\dots, n-1$,
\[\psi_k^{\cI}:=\sum_{\textbf{i}\in \alpha} \Phi_k^{\cI}\tQ_k^{\cI}(\textbf{i})\idmep\ \in\cI_{N,A} \ \ \ \ \ \text{and}\ \ \ \ \ \Phi_k^R:=\sum_{\textbf{i}\in \alpha} \psi_k^RQ^R_k(\textbf{i})\idmep\ \in\tR^{\alpha}_{\blambda,m}\ . \]
Finally we are ready to give the isomorphism. We consider the following maps, where $i\in\{1,\dots,n\}$, $k\in\{1,\dots,n-1\}$ and $\ti\in\alpha$.
\begin{equation}\label{def-F-GA}
F\ :\ \begin{array}{rcl}
\tR^{\alpha}_{\blambda,m} & \longrightarrow & \cI_{N,A} \\[0.5em]
\idmep &\mapsto & \idmep \\[0.2em]
y_i &\mapsto & y_i^{\cI} \\[0.2em]
\psi_k &\mapsto & \psi_k^{\cI}
\end{array}
\ \ \ \ \ \text{and}\ \ \ \ \ 
G\ :\ \begin{array}{rcl}
\cI_{N,A} &\longrightarrow & \tR^{\alpha}_{\blambda,m} \\[0.5em]
\idmep &\mapsto & \idmep \\[0.2em]
X_i &\mapsto & X_i^R\\[0.2em]
\Phi_k &\mapsto & \Phi_k^R
\end{array}\ .
\end{equation}
Assume that $F$ and $G$ extend to algebra homomorphisms. Then we have $F\circ G=\text{Id}_{\cI_{N,A}}$ and $G\circ F=\text{Id}_{\tR^{\alpha}_{\blambda,m}}$, since it is satisfied on the generators. This is immediate from the definitions for the generators $e(\ti)$, $\psi_k$ and $\Phi_k$. For generators $X_k,y_k$, we have, for any $\ti$, 
\[F\circ G(X_k)e(\ti)=F\Bigl(i_k\bigl(1-g(y_k)\bigr)\Bigr)e(\ti)=i_k\bigl(1-g\bigl(f(1-i_k^{-1}X_k)\bigr)\bigr)e(\ti)=X_ke(\ti)\ ,\]
since $f$ and $g$ are inverses for the composition. Similarly, we have $G\circ F(y_k)=y_k$. As a conclusion, the proof will be finished when we show that $F$ and $G$ extend to algebra homomorphisms.

\paragraph{Proof that $F$ and $G$ extend to algebra homomorphisms.}
We split the verification into several steps. Below we will often indicate only which relations we are checking without specifying for which elements we are doing so. Therefore, for convenience of the reader, let us recall that:
\begin{itemize}
\item[$-$] we must show Relations (\ref{Rel:V1})--(\ref{Rel:V7}) (without considering the ones with $\psi_0$), and the cyclotomic relations (\ref{cycV}) for the elements $\idmep,y_k^{\cI},\psi_k^{\cI}$ of $\cI_{N,A}$. 
\item[$-$] And, likewise, we must check Relations (\ref{rel1-e})--(\ref{rel1-Phi5}) (ignoring those involving $\Phi_0$) for the elements $\idmep, X_k^R, \Phi_k^R$ of $\tR^{\alpha}_{\blambda,m}$.
\end{itemize}

\begin{remark}\label{rem-symbol}
Moreover, by use of the symbol $\some$, which denotes either $\cI$ of $R$, we often calculate simultaneously in $\cI_{N,A}$ and in $\tR^{\alpha}_{\blambda,m}$. Our method of proof will often be the following. We will obtain an equality (or more generally an equivalence of some relations), for elements with the symbol $\some$. Roughly speaking, on one side of the equality, there will be a defining relation for $\cI_{N,A}$ and, on the other side, one of $\tR^{\alpha}_{\blambda,m}$. Then, by taking $\some=\cI$, we find that the defining relation in $\cI_{N,A}$ implies a relation in $\tR^{\alpha}_{\blambda,m}$ and conversely, by taking $\some=R$, we find that the defining relation in $\tR^{\alpha}_{\blambda,m}$ implies a relation in $\cI_{N,A}$.

Finally, note that, once we have checked that a defining relation of one of the two algebras $\cI_{N,A}$ or $\tR^{\alpha}_{\blambda,m}$ is satisfied, then we can use this defining relation with a symbol $\some$ on the elements.\hfill$\triangle$
\end{remark}

$\bullet$ \textbf{Relations (\ref{rel1-e})-(\ref{rel1-X}) and (\ref{Rel:V1})-(\ref{Rel:V2}).} In $\tR^{\alpha}_{\blambda,m}$, Relations (\ref{rel1-e})-(\ref{rel1-X}) are immediately seen to be satisfied by $\idmep,X_k^R$, and similarly for Relations (\ref{Rel:V1})-(\ref{Rel:V2}) for elements $\idmep,y_k^{\cI}$ in $\cI_{N,A}$.

\vskip .1cm
$\bullet$ \textbf{Relations (\ref{rel1-X1})-(\ref{rel1-Xk}) and (\ref{cycV}).} Let $k\in\{1,\dots,n\}$ and $\ti\in\alpha$. In $\tR^{\alpha}_{\blambda,m}$, we have 
$$(X^R_k-i_k)^{N}\idmep=(-i_kg(y_k))^Ne(\ti)=0\,,$$
because $g$ has no constant term and because $y_k^Ne(\ti)=0$ by assumption on $N$. Moreover, if $k=1$, this is true for $N=m(i_1)$ since $y_1^{m(i_1)}e(\ti)=0$ is the defining relation (\ref{cycV}) of $\tR^{\alpha}_{\blambda,m}$.

Conversely, in $\cI_{N,A}$, we have 
$$\bigl(y_1^{\cI}\bigr)^{m(i_1)}e(\ti)=\bigl(f(1-i_1^{-1}X_1)\bigr)^{m(i_1)}e(\ti)=0\ ,$$
since $f$ has no constant term and $(X_1-i_1)^{m(i_1)}e(\ti)=0$ in $\cI_{N,A}$.

\vskip .1cm
$\bullet$ \textbf{Relations (\ref{rel1-Phi1}) and (\ref{Rel:V3}).} Writing $\psi_a^{\some}$ in terms of $\Phi^{\some}_a$, we have
\[\Bigl(\psi^{\some}_ae(\ti)-e(r_a(\ti))\psi^{\some}_a\Bigr)\,e(\ti)=\Bigl(\Phi^{\some}_a\tQ^{\some}_a(\ti)e(\ti)-e(r_a(\ti))\Phi^{\some}_a\tQ^{\some}_a(\ti)\Bigr)\,e(\ti)\ .\]
By the previous steps, we know that, for $\some\in\{R,\cI\}$, the elements $Q^{\some}_k(\ti)$ commute with all the idempotents $e(\textbf{j})$, $\textbf{j}\in\alpha$. Using the invertibility of $\tQ^{\some}_a(\ti)$, this shows simultaneously (\ref{rel1-Phi1}) and (\ref{Rel:V3}) in front of the idempotent $e(\ti)$.

Then let $\textbf{j}\neq\ti$, and use that $e(\textbf{j})e(\ti)=0$ to note that
\begin{multline*}
\Bigl(\psi^{\some}_ae(\ti)-e(r_a(\ti))\psi^{\some}_a\Bigr)\,e(\textbf{j})=0\ \ \ \Leftrightarrow\ \ \ e(r_a(\ti))\psi^{\some}_ae(\textbf{j})=0\ \ \ \Leftrightarrow\ \ \ e(r_a(\ti))\Phi^{\some}_a\tQ^{\some}_a(\textbf{j})e(\textbf{j})=0\\
\ \ \ \Leftrightarrow\ \ \ e(r_a(\ti))\Phi^{\some}_ae(\textbf{j})=0\ \ \ \Leftrightarrow\ \ \ \Bigl(\Phi^{\some}_ae(\ti)-e(r_a(\ti))\Phi^{\some}_a\Bigr)\,e(\textbf{j})=0\ .
\end{multline*}
This shows (\ref{rel1-Phi1}) and (\ref{Rel:V3}) in front of the idempotent $e(\textbf{j})$ and therefore (\ref{rel1-Phi1}) and (\ref{Rel:V3}) are satisfied.

\vskip .1cm
$\bullet$ \textbf{Relations (\ref{rel1-Phi2}) and (\ref{Rel:V4}).} Here for clarity, we will treat each relation separately. We start with (\ref{Rel:V4}) in $\cI_{N,A}$. Recall that $s_b$ is the transposition $(b,b+1)$, and note that 
$$y^{\cI}_{s_b(j)} e(r_b(\ti))=f(1-i_j^{-1}X_{s_b(j)})e(r_b(\ti))={}^{r_b}f(1-i_j^{-1}X_{j})e(r_b(\ti))\ .$$
Therefore, we have in $\cI_{N,A}$
\begin{multline*}
\bigl(\psi^{\cI}_by^{\cI}_j-y^{\cI}_{s_b(j)}\psi_b^{\cI}\bigr)e(\ti)=\bigl(\Phi_bf(1-i_j^{-1}X_{j})-f(1-i_j^{-1}X_{s_b(j)})\Phi_b\bigr)\tQ^{\cI}_b(\ti)e(\ti)\\
=\delta_{i_b,i_{b+1}}(q X_{b+1}-q^{-1}X_b)\frac{f(1-i_j^{-1}X_{j})-f(1-i_j^{-1}X_{s_b(j)})}{X_{b+1}-X_b}\tQ^{\cI}_b(\ti)e(\ti)\ .
\end{multline*}
This is equal to $0$ if $s_b(j)=j$ or if $i_b\neq i_{b+1}$. Now assume that $i_b=i_{b+1}$. Then, from Condition (\ref{eq-Q1}) on $\tQ^{\cI}_b(\ti)$ and from the fact that $X_ke(\ti)=i_k\bigl(1-g(y^{\cI}_k)\bigr)e(\ti)$, we have:
\begin{equation}\label{eq2}
\tQ^{\cI}_b(\ti)\frac{q X_{b+1}-q^{-1}X_b}{X_{b+1}-X_b}e(\ti)=(y^{\cI}_{b+1}-y^{\cI}_b)^{-1}e(\ti)\ .
\end{equation}

Moreover, if $j\in\{b,b+1\}$, we find finally
\[\bigl(\psi^{\cI}_by_j-y_{s_b(j)}\psi_b^{\cI}\bigr)e(\ti)=\frac{y_j-y_{s_b(j)}}{y_{b+1}-y_b}e(\ti)\ .\]

Now we deal with (\ref{rel1-Phi2}) in $\tR^{\alpha}_{\blambda,m}$. Let $\Pi$ be a Laurent polynomial in $X^R_1,\dots, X_n^R$, and let $P$ be the formal power series in $y_1,\dots,y_n$ defined by $P:=\Pi\bigl((i_1\bigl(1-g(y_1)\bigr),\dots,i_n\bigl(1-g(y_n)\bigr)\bigr)$. Note that $\Pi e(\ti)=Pe(\ti)$ and 
\begin{multline*}
{}^{r_b}\Pi e(r_b(\ti))=\Pi(X^R_1,\dots,X^R_{b+1},X^R_b,\dots,X^R_n)e(r_b(\ti))\\
=\Pi\Bigl((i_1\bigl(1-g(y_1)\bigr),\dots,i_{b}\bigl(1-g(y_{b+1})\bigr),i_{b+1}\bigl(1-g(y_b)\bigr),\dots,i_n\bigl(1-g(y_n)\bigr)\Bigr)e(r_b(\ti))={}^{r_b}Pe(r_b(\ti))\ ,
\end{multline*}
where ${}^{r_b}\Pi$ stands for the transposition of $X^R_b$ and $X^R_{b+1}$ in the Laurent polynomial $\Pi$, while ${}^{r_b}P$ stands for the transposition of $y_b$ and $y_{b+1}$ in the power series $P$. Now, using Relation (\ref{rel:V4'}) in $\tR^{\alpha}_{\blambda,m}$, we find
\begin{multline*}
\bigl(\Phi^{R}_b\Pi-{}^{r_b}\Pi\Phi_b^{R}\bigr)e(\ti)=\bigl(\psi_bP-{}^{r_b}P\psi_b\bigr)Q^{R}_b(\ti)e(\ti)=\delta_{i_b,i_{b+1}}\frac{P-{}^{r_b}P}{y_{b+1}-y_b}Q^{R}_b(\ti)e(\ti)\\
=\delta_{i_b,i_{b+1}}(q X_{b+1}-q^{-1}X_b)\frac{\Pi-{}^{r_b}\Pi}{X_{b+1}-X_b}e(\ti)\ ,
\end{multline*}
where we note that $e(\ti)=e(r_b(\ti))$ if $i_b=i_{b+1}$, and we have again used Relation (\ref{eq2}) with upper indices $R$ instead of $\cI$.

At this point of the proof we can start using the following relation, for any $\some\in\{R,\cI\}$ and any formal power series $P$ in $y^{\some}_1,\dots,y^{\some}_n$:
\begin{equation}\label{eqproof}
\bigl(\Phi^{\some}_bP-{}^{r_b}P\Phi_b^{\some}\bigr)e(\ti)=\delta_{i_b,i_{b+1}}\frac{P-{}^{r_b}P}{y^{\some}_{b+1}-y^{\some}_b}Q^{\some}_b(\ti)e(\ti)\ .
\end{equation}

\vskip .1cm
$\bullet$ \textbf{Relations (\ref{rel1-Phi3}) and (\ref{Rel:V5}).} Let $a,b\in\{1,\dots,n-1\}$ such that $|a-b|>1$. Note that $Q^{\some}_a(\ti)\in K[[y^{\some}_a,y^{\some}_{a+1}]]$ so it commutes with $\Phi^{\some}_b$, and moreover the dependence on $\ti$ is only through $i_a$ and $i_{a+1}$ so we have $Q^{\some}_a(r_b(\ti))=Q^{\some}_a(\ti)$. Similar statements hold with $a$ and $b$ exchanged. Thus we have
\[(\psi^{\some}_a\psi^{\some}_b-\psi^{\some}_b\psi^{\some}_a)e(\ti)=\bigl(\Phi^{\some}_a\tQ^{\some}_a(r_b(\ti))\Phi^{\some}_b\tQ^{\some}_b(\ti)-\Phi^{\some}_b\tQ^{\some}_b(r_a(\ti))\Phi^{\some}_a\tQ^{\some}_a(\ti)\bigr)e(\ti)=(\Phi^{\some}_a\Phi^{\some}_b-\Phi^{\some}_b\Phi^{\some}_a)\tQ^{\some}_b(\ti)\tQ^{\some}_a(\ti)e(\ti)\ ,\]
and we note that $\tQ^{\some}_b(\ti)\tQ^{\some}_a(\ti)$ is invertible to conclude the verification.

\vskip .1cm
$\bullet$ \textbf{Relations (\ref{rel1-Phi4}) and (\ref{Rel:V6}).} Assume first that $i_b\neq i_{b+1}$. Using (\ref{eqproof}), we have
\[(\psi^{\some}_b)^2e(\ti)=\Phi^{\some}_b\tQ^{\some}_b(r_b(\ti))\,\Phi^{\some}_b\tQ^{\some}_b(\ti)\,e(\ti)=(\Phi^{\some}_b)^2.{}^{r_b}\tQ^{\some}_b(r_b(\ti))\tQ^{\some}_b(\ti)\,e(\ti)\ .\]
Therefore, we have:
\[\Bigl((\psi^{\some}_b)^2-\Gamma^{\some}_b(\ti){^{r_b}}\tQ^{\some}_b(r_b(\ti))\tQ^{\some}_b(\ti)\Bigr)e(\ti)
=0\ \ \ \ \Leftrightarrow\ \ \ \ \Bigl((\Phi^{\some}_b)^2-\Gamma^{\some}_b(\ti)\Bigr){^{r_b}}\tQ^{\some}_b(r_b(\ti))\tQ^{\some}_b(\ti)e(\ti)=0 .\]
We note that, by definition, $\Gamma^{\some}_b(\ti)e(\ti)=\Gamma^{\some}_be(\ti)$, so on the right we have (\ref{rel1-Phi4}) multiplied by ${^{r_b}}\tQ^{\some}_b(r_b(\ti))\tQ^{\some}_b(\ti)$ which is invertible. And from Condition (\ref{eq-Q2}) on the family $Q^{\some}_b(\ti)$, we have precisely (\ref{Rel:V6}) on the left (in all cases where $i_b\neq i_{b+1}$).

Now assume that $i_b= i_{b+1}$. From (\ref{eq-Q1}), a direct calculation shows that
\[\bigl({}^{r_b}Q^{\some}_b(\ti)-Q^{\some}_b(\ti)\bigr)e(\ti)=i_b^{-1}(q+q^{-1})(X^{\some}_{b+1}(\ti)-X^{\some}_b(\ti))\frac{y^{\some}_{b+1}-y^{\some}_{b}}{g(y^{\some}_{b+1})-g(y^{\some}_b)}e(\ti)=(q+q^{-1})(y^{\some}_{b+1}-y^{\some}_b)e(\ti)\ .\]
Therefore we have, using the commutation relation (\ref{eqproof}),
\begin{multline*}
(\psi^{\some}_b)^2e(\ti)=\Phi^{\some}_b\,\tQ^{\some}_b(\ti)\,\Phi^{\some}_b\,\tQ^{\some}_b(\ti)\,e(\ti)=\Phi^{\some}_b\Bigl(\Phi^{\some}_b{}^{r_b}\tQ^{\some}_b(\ti)-Q^{\some}_b(\ti)\frac{\tQ^{\some}_b(\ti)-{}^{r_b}\tQ^{\some}_b(\ti)}{y^{\some}_{b+1}-y^{\some}_b}\Bigr)\tQ^{\some}_b(\ti)\,e(\ti)\\
=\Phi^{\some}_b\Bigl(\Phi^{\some}_b-\frac{{}^{r_b}Q^{\some}_b(\ti)-Q^{\some}_b(\ti)}{y^{\some}_{b+1}-y^{\some}_b}\Bigr){}^{r_b}\tQ^{\some}_b(\ti)\tQ^{\some}_b(\ti)\,e(\ti)=\Phi^{\some}_b\Bigl(\Phi^{\some}_b-(q+q^{-1})\Bigr){}^{r_b}\tQ^{\some}_b(\ti)\tQ^{\some}_b(\ti)\,e(\ti)\ ,
\end{multline*}
and we note again, to conclude the proof of both (\ref{rel1-Phi4}) and (\ref{Rel:V6}), that ${}^{r_b}\tQ^{\some}_b(\ti)\tQ^{\some}_b(\ti)$ is invertible.

\vskip .1cm
$\bullet$ \textbf{Relations (\ref{rel1-Phi5}) and (\ref{Rel:V7}).} To treat the final braid relations we will write $\psi^{\some}_b$ in terms of $\Phi_b^{\some}$, $\psi^{\some}_{b+1}$ in terms of $\Phi_{b+1}^{\some}$, and use all the relations already proved up until this point, especially the commutation relation (\ref{eqproof}). To simplify notation, we note again that $Q^{\some}_b(\ti)$ depends on $\ti$ only through $i_b$ and $i_{b+1}$, so we will write $Q^{\some}_b(i_b,i_{b+1})$ instead (and similarly for $Q^{\some}_{b+1}(\ti)$). Besides, it is enough to take $b=1$ and we denote 
$$(i_1,i_2,i_3)=:(i,j,k)\ .$$
We will almost always omit the idempotent $e(\ti)$ on the right hand side of each line.

First we note that, by repeatedly using the commutation relation (\ref{eqproof}), we will obtain
\begin{equation}\label{dev1A}
\psi^{\some}_1\psi^{\some}_2\psi^{\some}_1=\Phi^{\some}_1\tQ^{\some}_1(jk)\ \Phi^{\some}_2\tQ^{\some}_2(ik)\ \Phi^{\some}_1\tQ^{\some}_1(ij)=\Phi^{\some}_1\Phi^{\some}_2\Phi^{\some}_1\ {}^{r_1r_2}\tQ^{\some}_1(jk){}^{r_1}\tQ^{\some}_2(ik)\tQ^{\some}_1(ij)\ +\ \dots
\end{equation}
\begin{equation}\label{dev2A}
\psi^{\some}_2\psi^{\some}_1\psi^{\some}_2=\Phi^{\some}_2\tQ^{\some}_2(ij)\ \Phi^{\some}_1\tQ^{\some}_1(ik)\ \Phi^{\some}_2\tQ^{\some}_2(jk)=\Phi^{\some}_2\Phi^{\some}_1\Phi^{\some}_2\ {}^{r_2r_1}\tQ^{\some}_2(ij){}^{r_2}\tQ^{\some}_1(ik)\tQ^{\some}_2(jk)\ +\ \dots
\end{equation}
where in each case, the dots indicate terms with at most two occurrences of $\Phi^{\some}_1,\Phi^{\some}_2$. The crucial fact here is that
\[{}^{r_1r_2}\tQ^{\some}_1(jk){}^{r_1}\tQ^{\some}_2(ik)\tQ^{\some}_1(ij)={}^{r_2r_1}\tQ^{\some}_2(ij){}^{r_2}\tQ^{\some}_1(ik)\tQ^{\some}_2(jk)=:\bold{Q}^{\some}\ .\]
This equality is obtained by using Condition (\ref{eq-Q3}) three times on the family $Q^{\some}_b(\ti)$. So now we set
\[B^{\some}e(\ti):=\Bigl(\psi^{\some}_1\psi^{\some}_2\psi^{\some}_1-\psi^{\some}_2\psi^{\some}_1\psi^{\some}_2-(\Phi^{\some}_1\Phi^{\some}_2\Phi^{\some}_1-\Phi^{\some}_2\Phi^{\some}_1\Phi^{\some}_2)\bold{Q}^{\some}\Bigr)e(\ti)\ ,\]
and our goal is to calculate $B^{\some}e(\ti)$ and to check that, for any $\ti$, this is equal to the right hand side of (\ref{Rel:V7}) minus the right hand side of (\ref{rel1-Phi5}) multiplied by $\bold{Q}^{\some}$ on the right. Since $\bold{Q}^{\some}$ is invertible, this will prove both (\ref{rel1-Phi5}) and (\ref{Rel:V7}) (by taking first $\some=R$ and then $\some=\cI$).

There are several cases to consider ($i\neq k$, $i=k\neq j$ and $i=j=k$), We will repeatedly use the commutation relation (\ref{eqproof}). The right hand side in this relation, which appear in the calculation depending on conditions on $\ti$, produce the ``correcting'' terms indicated by dots in (\ref{dev1A})-(\ref{dev2A}).
\\
\\
\textbf{Notation.} For convenience, we drop the symbol $\some$ until the end of this verification.
\\
\\
\textbf{Case 1:} $i\neq k$. In (\ref{dev1A}), correcting terms can appear only when applying the commutation relation to the last $\Phi_1$, while in (\ref{dev2A}), correcting terms can appear only when applying the commutation relation to the last $\Phi_2$. We reason as follows:
\[\psi_1\psi_2\psi_1-\psi_2\psi_1\psi_2=\Phi_1\Phi_2{}^{r_2}\tQ_1(jk)\tQ_2(ik)\,\Phi_1\tQ_1(ij)\ -\ \Phi_2\Phi_1{}^{r_1}\tQ_2(ij)\tQ_1(ik)\,\Phi_2\tQ_2(jk)\ .\]
For the first term, if $i\neq j$ then there is no correcting term while moving through the last $\Phi_1$. If $i=j$ then we have, due to Condition (\ref{eq-Q3}) on the $Q$'s, that ${}^{r_2}\tQ_1(jk)={}^{r_1}\tQ_2(ik)$ and therefore the factor ${}^{r_2}\tQ_1(jk)\tQ_2(ik)$ commutes with $\Phi_1$. So finally we never have a correcting term and the first summand above is $\Phi_1\Phi_2\Phi_1\textbf{Q}$.

A similar reasoning shows that the second summand is always equal to $\Phi_2\Phi_1\Phi_2\textbf{Q}$ and we conclude that in this case we have $B^\some\idmep=0$ as required.
\\
\\
\textbf{Case 2:} $i=k\neq j$. Here, in (\ref{dev1A}), correcting terms can appear only when applying the commutation relation to the middle $\Phi_2$ while, in (\ref{dev2A}), correcting terms can appear only when applying the commutation relation to the middle $\Phi_1$. We find
\begin{multline*}
B^\some\idmep = \Phi_1\frac{\tQ_1(ji)-{}^{r_2}\tQ_1(ji)}{y_3-y_2}\Phi_1\tQ_1(ij)\ -\  \Phi_2\frac{\tQ_2(ij)-{}^{r_1}\tQ_2(ij)}{y_2-y_1}\Phi_2\tQ_2(ji)\\
=\psi_1^2\frac{1}{y_3-y_1}-\Phi_1^2\frac{{}^{r_1r_2}\tQ_1(ji)\tQ_1(ij)}{y_3-y_1}\ -\ \psi_2^2\frac{1}{y_3-y_1}+\Phi_2^2\frac{{}^{r_2r_1}\tQ_2(ij)\tQ_2(ji)}{y_3-y_1}
\\
= (\psi_1^2-\psi_2^2)\frac{1}{y_3-y_1}\ -\ (\Gamma_1-\Gamma_2)\frac{{^{r_1}}Q_2(ii)}{y_3-y_1}\textbf{Q}\ .
\end{multline*}
We note, using Condition (\ref{eq-Q1}) on $Q_2(ii)$, that $\displaystyle\frac{{^{r_1}}Q_2(jj)}{y_3-y_1}=\frac{qX_3-q^{-1}X_1}{X_3-X_1}$ and therefore the coefficient in front of $\textbf{Q}$ is what is required. The verification is concluded with the following relations which are easy to check:
\[\big( \psi_1^2-\psi_2^2 \big)\frac{1}{y_3-y_1}e(\ti)=\left\{
\begin{array}{l l l l}
0 & \quad \text{if $i=k\nleftrightarrow j$}\,\\[0.2em]
\idmep & \quad \text{if $i=k\rightarrow j$}\,\\[0.2em]
-\idmep & \quad \text{if $i=k\leftarrow j$}\, \\[0.2em]
(y_{3}-2y_{2}+y_1)\idmep & \quad \text{if $i=k\leftrightarrow j$}\,.
\end{array} \right. \] 
\\
\\
\textbf{Case 3:} In this case we have $i=j=k$. This is the most difficult case computationally since correcting terms may appear for every commutation relation we use. We sketch the main steps of the calculation. To save space we set, for $a\in\{1,2\}$, $Q_a:=Q_a(ii)$, $\tQ_a=Q_a^{-1}$. We recall that in this case $Q_1$ and $Q_2$ have explicit expressions given in (\ref{eq-Q1}) which are equal to:
\[Q_1=\frac{q^{-1}X_1-qX_{2}}{X_1-X_2}(y_2-y_1)\ \ \ \ \ \text{and}\ \ \ \ \ Q_2=\frac{q^{-1}X_2-qX_{3}}{X_2-X_3}(y_3-y_2)\,.\]
We recall that we have in this case
\begin{equation}\label{C1A-a}
\psi_1^2=\Phi_1\tQ_1\Phi_1\tQ_1=0\ \ \ \ \ \text{and}\ \ \ \ \ \psi_2^2=\Phi_2\tQ_2\Phi_2\tQ_2=0\ .
\end{equation}
We will also use the following particular instance of the commutation relation (\ref{eqproof}),
\begin{equation}\label{C1A-b}
\frac{1}{y_3-y_2}\Phi_1-\Phi_1\frac{1}{y_3-y_1}=\frac{1}{(y_3-y_2)(y_3-y_1)}Q_1\ .
\end{equation}
Then we calculate as follows:
\[\psi_1\psi_2\psi_1=\Phi_1\tQ_1\Phi_2\tQ_2\Phi_1\tQ_1=\ (A:=)\ \Phi_1\Phi_2{}^{r_2}\tQ_1\tQ_2\Phi_1\tQ_1\ +\ (B:=)\ \Phi_1\frac{\tQ_1-{}^{r_2}\tQ_1}{y_3-y_2}\Phi_1\tQ_1\ .\]
Then, we use condition (\ref{eq-Q3}) which says that ${}^{r_2}\tQ_1={}^{r_1}\tQ_2$ and we find that  ${}^{r_2}\tQ_1\tQ_2$ commutes with $\Phi_1$. Therefore we have that $A=\Phi_1\Phi_2\Phi_1\textbf{Q}$.

To treat the term $B$, we start by applying (\ref{C1A-b}) to find:
\[B=\Phi_1(\tQ_1-{}^{r_2}\tQ_1)\Phi_1\frac{\tQ_1}{y_3-y_1}\,+\,\Phi_1\frac{\tQ_1-{}^{r_2}\tQ_1}{(y_3-y_2)(y_3-y_1)}\ .\]
On the first term, we apply (\ref{C1A-a}) and the fact that ${}^{r_2}\tQ_1={}^{r_1}\tQ_2$ to obtain:
\[B=-\Phi_1^2\frac{\tQ_2\tQ_1}{y_3-y_1}\,+\,\Phi_1\frac{\tQ_2-{}^{r_1}\tQ_2}{(y_2-y_1)(y_3-y_1)}\,+\,\Phi_1\frac{\tQ_1-{}^{r_2}\tQ_1}{(y_3-y_2)(y_3-y_1)}\ .\]
We recall that $\Phi_1^2=(q+q^{-1})\Phi_1$ and again that ${}^{r_2}\tQ_1={}^{r_1}\tQ_2$, and we find:
\[B=-(q+q^{-1})\Phi_1\textbf{Q}\frac{{}^{r_2}\tQ_1}{y_3-y_1}\,+\,\Phi_1\Bigl(\frac{\tQ_1}{(y_3-y_2)(y_3-y_1)}+\frac{\tQ_2}{(y_2-y_1)(y_3-y_1)}-\frac{{}^{r_2}\tQ_1}{(y_2-y_1)(y_3-y_2)}\Bigr)\ .\]
The big parenthesis is equal to
\[\textbf{Q}\Bigl(\frac{(q^{-1}X_2-q X_3)(q^{-1}X_1-q X_3)}{(X_2-X_3)(X_1-X_3)}+\frac{(q^{-1}X_1-q X_2)(q^{-1}X_1-q X_3)}{(X_1-X_2)(X_1-X_3)}-\frac{(q^{-1}X_1-q X_2)(q^{-1}X_2-q X_3)}{(X_1-X_2)(X_2-X_3)}\Bigr)\ .\]
Now putting everything above the denominator $(X_1-X_2)(X_2-X_3)(X_1-X_3)$, one sees immediately that $X_2=X_1$ and $X_2=X_3$ cancels the numerator. Then looking at the constant term with respect to $X_2$, one finds at once that the numerator is $(X_1-X_2)(X_2-X_3)(q^2X_1-q^{-2}X_3)$.
Thus, a short final calculation gives:
\[B=-(q+q^{-1})\Phi_1\textbf{Q}\frac{q X_1-q^{-1}X_3}{X_1-X_3}+\Phi_1\textbf{Q}\frac{q^2X_1-q^{-2}X_3}{X_1-X_3}=-\Phi_1\textbf{Q}\ .\]
Finally, summing $A$ and $B$, we conclude that:
\[\psi_1\psi_2\psi_1=\bigl(\Phi_1\Phi_2\Phi_1-\Phi_1\bigr)\textbf{Q}\ .\]
Then, a very similar calculation shows that
\[\psi_2\psi_1\psi_2=\bigl(\Phi_2\Phi_1\Phi_2-\Phi_2\bigr)\textbf{Q}\ .\]
which allows to conclude that the following required relation is satisfied:
\[B^\some e(\ti)=-(\Phi_1-\Phi_2)\textbf{Q}\ .\]

\paragraph{Proof of item 2.} Recall the notation $Y^{-1}e(\ti)$ introduced in Section \ref{sec-inter} before Definition \ref{defI}. We use the similar notation $Y^{-1}e^H(\ti)$ in $e_{\alpha}H(A_n)_{\blambda,m}$. Note for example that $(1-X^{-\alpha_a})^{-1}e^H(\ti)$ are defined if $r_a(\ti)\neq \ti$.

Let $b\in\{1,\dots,n-1\}$. We define elements of $e_{\alpha}H(A_n)_{\blambda,m}$:
\[\Phi^H_b=g_be_{\alpha}-\sum_{\begin{array}{c}\scriptstyle{\ti\in\alpha}\\[-0.2em] \scriptstyle{i_b\neq i_{b+1}}\end{array}}(q-q^{-1})\bigl(1-X_bX_{b+1}^{-1}\bigr)^{-1}e^H(\textbf{i})+\sum_{\begin{array}{c}\scriptstyle{\textbf{i}\in\alpha}\\[-0.2em] \scriptstyle{i_b=i_{b+1}}\end{array}}q^{-1}e^H(\textbf{i})\ ,\]
and elements of $\cI_{N,A}$:
\[g^{\cI}_b:=\Phi_b+\!\!\!\sum_{\begin{array}{c}\scriptstyle{\ti\in\alpha}\\[-0.2em] \scriptstyle{i_b\neq i_{b+1}}\end{array}}(q-q^{-1})\bigl(1-X_bX_{b+1}^{-1}\bigr)^{-1}e(\textbf{i})-\!\!\!\sum_{\begin{array}{c}\scriptstyle{\textbf{i}\in\alpha}\\[-0.2em] \scriptstyle{i_b=i_{b+1}}\end{array}}q^{-1}e(\textbf{i})\ .\]

We consider the following maps, where $i\in\{1,\dots,n\}$, $b\in\{1,\dots,n-1\}$ and $\ti\in\alpha$.
\begin{equation}\label{def-rho-sigmaA}
\rho\ :\ \ \begin{array}{rcl}
\cI_{N,A} & \to & e_{\alpha}H(A_n)_{\blambda,m}\\[0.5em]
e(\ti) & \mapsto & e^H(\ti)\\[0.2em]
X_i & \mapsto & X_i e_{\alpha}\\[0.2em]
\Phi_b & \mapsto & \Phi^H_b
\end{array}\ 
\ \ \ \ \text{and}\ \ \ \ \  \ 
\sigma\ :\ \ \begin{array}{rcl}
\hH(A_n) & \to & \cI_{N,A}\\[0.5em]
X_i & \mapsto & X_i\\[0.2em]
g_b & \mapsto & g_b^{\cI}
\end{array}\ .
\end{equation}
Item 2 follows immediately from the following two lemmas.
\begin{lemma}\label{lem1A}
The map $\rho$ extends to a morphism of algebras from $\cI_{N,A}$ to $e_{\alpha}H(A_n)_{\blambda,m}$. The map $\sigma$ extends to a morphism of algebras from $\hH(A_n)$ to $\cI_{N,A}$. 
\end{lemma}
We still denote by $\rho$ and  $\sigma$ the obtained morphisms of algebras. The canonical surjection from $\hH(A_n)$ to $e_{\alpha}H(A_n)_{\blambda,m}$ appearing in the following lemma is the composition of the standard surjection from $\hH(A_n)$ to its quotient $H(A_n)_{\blambda,m}$ followed by the multiplication by $e_{\alpha}$.
\begin{lemma}\label{lem2A}
The map $\sigma$ factors through the canonical surjection $\hH(A_n)\to e_{\alpha}H(A_n)_{\blambda,m}$, and the resulting map provides the two-sided inverse of $\rho$.
\end{lemma}

\begin{proof}[Proof of Lemma \ref{lem1A}]
For convenience, we will denote during the proof $\Phi_b^{\cI}:=\Phi_b$ and $e^{\cI}(\ti):=e(\ti)$ in $\cI_{N,A}$ and $g_b^H:=g_b$ in $\hH(A_n)$. We need to check that:
\begin{itemize}
\item[-] The defining relations of $\hH(A_n)$ are satisfied by the elements $g_b^{\cI}$ and $X_i$ in $\cI_{N,A}$. The relations to be checked are; Relation (\ref{rel-Lu}), the fact that the $X_i$'s commute and the relations among (\ref{rel-H0}) which do not involve $g_0$.
\item[-] The defining relations of $\cI_{N,A}$ are satisfied by the elements $\Phi_b^H$, $X_ie_{\alpha}$ and $e^H(\ti)$ in $e_{\alpha}H(A_n)_{\blambda,m}$. The relations to be checked are the relations among (\ref{rel1-e})--(\ref{rel1-Phi5}) which do not involve $\Phi_0$.
\end{itemize}

We split the proof into several steps. Below we often indicate only which relations we are checking without specifying for which elements we are doing so (the relations to be checked are explicitly recalled above, so this should not lead to any confusion). 

Moreover, let the symbol $\some$ denote either $\cI$ of $H$. We use this symbol in order to check simultaneously relations in $\cI_{N,A}$ and in $e_{\alpha}H(A_n)_{\lambda,m}$, in exactly the same spirit as explained in Remark \ref{rem-symbol}.

A last preliminary remark is that to prove that some defining relations of $\hH(A_n)$ are satisfied in $\cI_{N,A}$, we check it in front of all the idempotents $e(\ti)$.

\vskip .1cm
$\bullet$ \textbf{Relations (\ref{rel1-e})--(\ref{rel1-Xk}).} They are satisfied by $e^H(\ti)$ and $X_ie_{\alpha}$ in $e_{\alpha}H(A_n)_{\lambda,m}$ by construction, using our assumption on $N$ for (\ref{rel1-Xk}) and using the cyclotomic relation (\ref{cyc-relA}) for (\ref{rel1-X1}).

\vskip .1cm
$\bullet$ The $X_i$'s commute in $\cI_{N,A}$ by definition.

\vskip .1cm
$\bullet$ \textbf{Relations (\ref{rel1-Phi1}).} We use that in $H(A_n)_{\lambda,m}$, we have $U_a(M_{\ti})\subset M_{r_a(\ti)}$ where we recall that $M_{\ti}$ is the image of the projector $e^H(\ti)$ and $U_a$ is the intertwining element of Section \ref{subsec-int}. Indeed this follows from the first relation in (\ref{rel-Ui}). Then, if $r_a(\ti)\neq\ti$, this proves (\ref{rel1-Phi1}). While if $r_a(\ti)=\ti$, this implies that $g_a(M_{\ti})\subset M_{\ti}$ which is also enough to obtain  (\ref{rel1-Phi1}).

\vskip .1cm
$\bullet$ \textbf{Relations (\ref{rel1-Phi2}) and (\ref{rel-Lu}).} We use that the $X_i$'s commute. 

If $r_b(\ti)\neq\ti$, we have
\[(\Phi^{\some}_b X^x-X^{r_b(x)}\Phi^{\some}_b) e^{\some}(\ti) =\bigl(g^{\some}_bX^x - X^{r_b(x)}g^{\some}_b-(q-q^{-1})\frac{X^x-X^{r_a(x)}}{1-X^{-\alpha_a}}\bigr)e^{\some}(\ti)\]
If $r_b(\ti)=\ti$, we have, after a direct calculation,
\begin{multline*}
(\Phi^{\some}_b X^x-X^{r_b(x)}\Phi^{\some}_b) e^{\some}(\ti) =\left(g^{\some}_bX^x - X^{r_b(x)}g^{\some}_b+q^{-1}(X^x-X^{r_b(x)})\right)e^{\some}(\ti)\\
 =\left( g^{\some}_bX^x - X^{r_b(x)}g^{\some}_b-(q-q^{-1})\frac{X^x-X^{r_a(x)}}{1-X^{-\alpha_a}}+ (q-q^{-1} X^{-\alpha_a})\displaystyle\frac{X^x-X^{r_a(x)}}{1-X^{-\alpha_a}}\right)e^{\some}(\ti)\ .
 \end{multline*}

$\bullet$ \textbf{Relations (\ref{rel1-Phi3}) and $\boldsymbol{g_ag_b=g_bg_a}$ if $\boldsymbol{|a-b|>1}$.} From the previous step, we can use that $g^{\some}_aX^{\alpha_b}=X^{\alpha_b}g^{\some}_a$ and $g^{\some}_bX^{\alpha_a}=X^{\alpha_a}g^{\some}_b$ for $\some\in\{\cI,H\}$.

Let us also check that, for $\some\in\{\cI,H\}$, $g^{\some}_a$ commutes with
\[\sum_{\begin{array}{c}\scriptstyle{\ti\in\alpha}\\[-0.2em] \scriptstyle{i_b\neq i_{b+1}}\end{array}}e^{\some}(\textbf{i})\ \ \qquad\text{and}\qquad\ \ \sum_{\begin{array}{c}\scriptstyle{\textbf{i}\in\alpha}\\[-0.2em] \scriptstyle{i_b=i_{b+1}}\end{array}}e^{\some}(\textbf{i})\ .\]
In $\cI_{N,A}$ (that is, if $\some=\cI$), this follows from the fact that $\Phi_a$ commutes with these elements, which is a direct consequence of the defining relations (\ref{rel1-Phi1}) between $\Phi_a$ and the idempotents $e(\ti)$.

In $H(A_n)_{\lambda,m}$ (that is, if $\some=H$), note that these sums are, respectively, the projector onto the characteristic space for $X^{\alpha_b}$ for the eigenvalue 1, and the sum of projectors on characteristic spaces for $X^{\alpha_b}$ for eigenvalues different from 1. So they are defined only in terms of $X^{\alpha_b}$ and therefore commute with $g_a$.

With this, it is immediate to see that, for any $\ti$,
\[\left(\Phi^{\some}_a\Phi^{\some}_b-\Phi^{\some}_b\Phi^{\some}_a\right)e^{\some}(\ti)=\left(g^{\some}_ag^{\some}_b-g^{\some}_bg^{\some}_a\right)e^{\some}(\ti)\ .\]

$\bullet$ \textbf{Relations (\ref{rel1-Phi4}) and $\boldsymbol{g_b^2=(q-q^{-1})g_b+1}$.} We use the previous steps, namely the commuting relations between $g_b^{\some}$ and elements $X^x$, and the commuting relations between $\Phi_b^{\some}$ and $e^{\some}(\ti)$. 

We then have, if $r_b(\ti)=\ti$,
\begin{multline*}
\left((\Phi^{\some}_b)^2-(q+q^{-1})\Phi^{\some}_b\right)e^{\some}(\ti)=\left((g^{\some}_b+q^{-1})^2-(q+q^{-1})(g^{\some}_b+q^{-1})\right)e^{\some}(\ti) \\
=\left((g^{\some}_b)^2-(q-q^{-1})g^{\some}_b-1\right)e^{\some}(\ti)\ .
\end{multline*}
If $r_b(\ti)\neq\ti$, we have
\begin{multline*}
(\Phi^{\some}_b)^2e^{\some}(\ti)=\left((g^{\some}_b-\frac{(q-q^{-1})}{1-X^{-\alpha_b}}\right)^2e^{\some}(\ti) \\
=\left((g^{\some}_b)^2-(q-q^{-1})\left(g^{\some}_b\frac{1}{1-X^{-\alpha_b}}-\frac{1}{1-X^{-\alpha_b}}g^{\some}_b\right)+\frac{(q-q^{-1})^2}{(1-X^{-\alpha_b})^2}  \right)e^{\some}(\ti)\ .
\end{multline*}
This is a straightforward and well-known calculation for which, using the commutation relation between $g^{\some}_b$ and $X^x$, one obtains
\[
(\Phi^{\some}_b)^2e^{\some}(\ti) =\left((g^{\some}_b)^2-(q-q^{-1})g^{\some}_b-1-\Gamma_b  \right)e^{\some}(\ti)\ .\]

$\bullet$ \textbf{Relations (\ref{rel1-Phi5}) and $\boldsymbol{g_bg_{b+1}g_b=g_{b+1}g_bg_{b+1}}$.} We calculate 
$$(\Phi^{\some}_b\Phi^{\some}_{b+1}\Phi^{\some}_b -\Phi^{\some}_{b+1}\Phi^{\some}_b\Phi^{\some}_{b+1})e^{\some}(\ti) - (g^{\some}_bg^{\some}_{b+1}g^{\some}_b -g^{\some}_{b+1}g^{\some}_bg^{\some}_{b+1})e^{\some}(\ti)$$
using all the previously checked relations. This calculation is exactly the same in $\cI_{N,A}$ and in $e_{\alpha}H(A_n)_{\lambda,m}$ (that is, for $\some=\cI$ or $\some=H$). It was done in \cite{BK} and the result is
\[\left\{\begin{array}{ll}
\bigl(\Phi^{\some}_b-\Phi^{\some}_{b+1}\bigr)e^{\some}(\ti) & \ \ \text{if $i_b=i_{b+1}=i_{b+2}$,}\\[0.4em]
Z_b e^{\some}(\ti) & \ \ \text{if $i_b=i_{b+2}\neq i_{b+1}$,}\\[0.4em]
0 & \ \ \text{otherwise.}
\end{array}\right.\]
This concludes the proof of the lemma.
\end{proof}

\begin{proof}[Proof of Lemma \ref{lem2A}]
First, we show that $\sigma$ factors through the quotient $H(A_n)_{\blambda,m}$ of $\hH(A_n)$. To show this, we need to check that the cyclotomic relation
\[\prod_{i\in I_{\lambda_1}\cup\dots\cup I_{\lambda_l}}(X_1-i)^{m(i)}=0\ ,\]
is satisfied in $\cI_{N,A}$. This follows at once from the defining relation $(X_1-i_1)^{m(i_1)}e(\ti)=0$, valid for any $\ti\in\alpha$.

Then, to show that it factors through the canonical surjection to $e_{\alpha}H(A_n)_{\blambda,m}$, we must show that:
\begin{equation}\label{sig-e}
\forall\,\ti\in \bigl(I_{\lambda_1}\cup\dots\cup I_{\lambda_l}\bigr)^n\,,\ \ \ \ \ \sigma\bigl(e^H(\ti)\bigr)=\left\{\begin{array}{cc}
e(\ti) & \text{if $\ti\in\alpha$,}\\[0.3em]
0 & \text{otherwise.}
\end{array}\right.
\end{equation}
We can reproduce the same proof as in \cite[Lemma 3.4]{BK}. We sketch it for the convenience of the readers. At this point, $\sigma$ provides an algebra morphism from $H(A_n)_{\blambda,m}$ to $\cI_{A,N}$, thereby providing $\cI_{A,N}$ with the structure of $H(A_n)_{\blambda,m}$-module. As in Section \ref{subsec-idem}, since $\sigma(X_k)=X_k$, let, for $\textbf{i}=(i_1,\dots,i_n)\in \bigl(I_{\lambda_1}\cup\dots\cup I_{\lambda_l}\bigr)^n$,
\[M_{\textbf{i}}:=\{x\in \cI_{A,N}\ |\ \bigl(X_k-i_k\bigr)^rx=0\ \ \text{for $k=1,\dots,n$ and some $r>0$}\}\ ,\]
be the generalised common eigenspaces in the $H(A_n)_{\blambda,m}$-module $\cI_{A,N}$. We have $\cI_{A,N}=\bigoplus_{\ti} M_{\ti}$ and, by definition of $e^H(\ti)$ in $H(A_n)_{\blambda,m}$, the projection onto $M_{\textbf{i}}$ associated to this decomposition of $\cI_{A,N}$ is $\sigma\bigl(e^H(\ti)\bigr)$. On the other hand, from the defining relations (\ref{rel1-X1})-(\ref{rel1-Xk}) and (\ref{rel1-e}), this idempotent of $\cI_{A,N}$ is $e(\ti)$ if $\ti\in\alpha$ and $0$ otherwise. This concludes the verification of (\ref{sig-e}).

\vskip .2cm
Finally, still denoting $\sigma$ the resulting morphism from $e_{\alpha}H(A_n)_{\blambda,m}$ to $\cI_{A,N}$, the equalities of maps $\sigma\circ\rho=\text{Id}_{\cI_{A,N}}$ and $\rho\circ\sigma=\text{Id}_{e_{\alpha}H(A_n)_{\blambda,m}}$ only have to be checked on the generators of the algebras  and this is immediate from the definitions (\ref{def-rho-sigmaA}) of $\rho$ and $\sigma$ and using (\ref{sig-e}). 
\end{proof}

\subsection{Automorphisms of cyclotomic KLR algebras}\label{sec-autA}

During the proof of the statement that the block $e_{\alpha}H(A_n)_{\blambda,m}$ is isomorphic to the cyclotomic KLR algebra $\tR^{\alpha}_{\blambda,m}$, we exhibit in fact a family of isomorphisms, parametrised by formal power series $f\in K[[z]]$ with no constant term and invertible for the composition. Denote $F_f\ :\ \tR^{\alpha}_{\blambda,m}\to e_{\alpha}H(A_n)_{\blambda,m}$ the corresponding isomorphism of algebras. As a direct consequence, we obtain a family of automorphisms of the cyclotomic KLR algebra $\tR^{\alpha}_{\blambda,m}$ (the one given below is obtained as $F_1^{-1}\circ F_f$).
 
\begin{proposition}
Let $f=\sum_{k\geq0}a_kz^k\in K[[z]]$ be a formal power series with $a_0=0$ and $a_1\neq 0$. There is an automorphism of the algebra $\tR^{\alpha}_{\blambda,m}$ given on the generators by
\[e(\ti)\mapsto e(\ti)\,,\ \ \ \ \ \ \ y_i\mapsto f(y_i)\ \ \ \ \ \ \text{and}\ \ \ \ \ \ \psi_k\mapsto \sum_{\ti\in\alpha}\psi_k Q_k(\ti)e(\ti)\,,\ \ \text{for some $Q_k(\ti)\in K[[y_k,y_{k+1}]]$\,.}\]
\end{proposition}
In fact, using for example the explicit formulas given in (\ref{formulaQk}), we find easily that one can take for elements $Q_k(\ti)$:
\[
Q_k(\textbf{i}):=\left\lbrace
\begin{array}{ll}
\displaystyle\frac{y_{k+1}-y_k}{f(y_{k+1})-f(y_k)} & \textrm{ when } i_k=i_{k+1}\,, \vspace*{1em}\\
1 & \textrm{ when } i_{k+1}\nleftrightarrow i_k\ \text{or}\ i_k \leftarrow i_{k+1}\,,\vspace*{1em}\\
\displaystyle\frac{f(y_{k+1})-f(y_k)}{y_{k+1}-y_k} & \textrm{ when }i_k \rightarrow i_{k+1}\ \text{or}\ i_k \leftrightarrow i_{k+1}\,.
\end{array}\right.
\]
We note that the proposition could also be checked directly in a straightforward manner.

\section{Proof of the main result}\label{sec-proof}

Theorem \ref{theo} follows immediately from the following two propositions.

\begin{proposition}\label{prop-iso1}
The algebra $V_{\bl,m}^{\beta}$ is isomorphic to the algebra $\cI_N$ for $N$ large enough.
\end{proposition}

\begin{proposition}\label{prop-iso2}
The algebra $e_{\beta}H(B_n)_{\blambda,m}$ is isomorphic to the algebra $\cI_N$ for $N$ large enough.
\end{proposition}

We will use Theorem \ref{theoA} with $\alpha=\beta$. In fact, a large part of the proof of the two propositions has been already done in Theorem \ref{theoA} and we will only need to deal with the additional defining relations coming from the type B setting. The method and steps that we use are the same as for the proof of Theorem \ref{theoA}.

\vskip .1cm
First fix an integer $N$ as in the beginning of the proof of Theorem \ref{theoA}. Namely, we assume first that $N$ is larger than all the nilpotency indices of the elements $(X_k-i_k)e^H(\ti)$ in $e_{\beta}H(B_n)_{\blambda,m}$, for $k=1,\dots,n$ and $\ti\in\beta$ (these elements are nilpotent by construction of the idempotents $e^H(\ti)$). 

Second, in $V_{\bl,m}^{\beta}$, the elements $y_ke(\ti)$ are nilpotent for any $k=1,\dots,n$ and any $\ti\in\beta$, thanks to Lemma \ref{lem-nily}. We also assume $N$ to be larger than all the nilpotency indices of these elements.

\paragraph{Proof of Proposition \ref{prop-iso1}.} We repeat (and extend to the type B setting) the first step of the proof of \textbf{item 1} in Section \ref{sec-proofA}. So we set, for all $k=1,\ldots, n$ and all $s=0,1,\ldots,n-1$,
\begin{equation*}
y_k^V:=y_k\,,\ \ \psi_s^V:=\psi_s\ \ \ \text{in $V^{\beta}_{\blambda,m}$,}\ \ \ \qquad\text{and}\ \ \ \qquad X_k^{\cI}:=X_k\,,\ \ \Phi_s^{\cI}:=\Phi_s \ \ \text{in $\cI_{N}$.}
\end{equation*}
A convenient notation will also be, for all $k=1,\ldots, n$,
\begin{equation}\label{nota2}
\oX_k:=X_k^{-1} .
\end{equation}

Let $k\in\{1,\dots,n\}$. We define elements of $\cI_{N}$ as follows:
\[y^{\cI}_k:=\sum_{\textbf{i} \in \beta}(i_k \oX_k-i_k^{-1}X_k)\idmep\in\cI_{N}\ .\]
Note that
\[i_k \oX_k-i_k^{-1}X_k=\bigl(1-(1-i_k^{-1}X_k)\bigr)^{-1}-\bigl(1-(1-i_k^{-1}X_k)\bigr)\ ,\]
Therefore, considering the following formal power series $f\in K[[z]]$,
\[f(z):=\frac{1}{1-z}-(1-z)=z+\frac{z}{1-z}\ ,\]
we have 
\[y^{\cI}_k:=\sum_{\textbf{i} \in \beta}f(1-i_k^{-1}X_k)e(\ti)\ .\]
and we are therefore in the setting of the proof of Theorem \ref{theoA}. Indeed we obviously have that $f$ has no constant term and moreover its degree 1 term is $2z$ which is non-zero since $K$ has characteristic different from 2. 

As before, let $g$ denote the composition inverse of $f$ and define elements of $V^{\beta}_{\blambda,m}$ by
\[X_k^V:=\sum_{\textbf{i} \in \beta}i_k\bigl(1-g(y_k)\bigr)\idmep\in V^{\beta}_{\blambda,m}\ .\]
We will not need an explicit expression for $g$, but we will need the following information.
\begin{lemma}\label{lem-fg}
The formal power series $g$ satisfies
\[1-g(-z)=\frac{1}{1-g(z)}\ .\]
\end{lemma}
\begin{proof} We have the following sequence of equivalences:
\[1-g(-z)=\frac{1}{1-g(z)}\ \ \Leftrightarrow\ \ g(-z)=\frac{-g(z)}{1-g(z)}\ \ \Leftrightarrow\ \ -z=f\bigl(\frac{-g(z)}{1-g(z)}\bigr)\ \ \Leftrightarrow\ \ -f(z)=f\bigl(\frac{-z}{1-z}\bigr)\ ,\]
where we applied $f$ to both sides to obtain the second equivalence, and we change the variable $z\mapsto f(z)$ to obtain the last one. The last condition on $f$ is easy to check using the explicit expression for $f$.
\end{proof}

We again consider a symbol $\some$ which we will use as an upper index to denote both $V$ and $\cI$, and we use the same notation $P^{\some}(\ti)$ as defined before in (\ref{nota1}) for a rational expression $P$ in $X_1,\dots,X_n$.

Next, we fix a family of invertible power series $Q^{\some}_k(\ti) \in K[[y_k, y_{k+1}]]$, where $k=1,\dots,n-1$ and $\ti\in\beta$, as in the proof of Theorem \ref{theoA}, which satisfies the same conditions (\ref{eq-Q1})--(\ref{eq-Q3}) as earlier.  

We add a family of invertible power series $Q_0(i) \in K[[z]]$, where $i\in K$, and we set for $\some \in \{ V, \cI \}$:
$$Q^{\some}_0(i)=Q_0(i)\bigl(y^{\some}_1\bigr) \in K[[y^{\some}_1]]\ \ \ \quad\text{and}\quad\ \ \ \tQ^{\some}_0(\textbf{i}):=\bigl(Q^{\some}_0(\textbf{i})\bigr)^{-1}\ ,$$
and we will use the following notation:
$$Q^{\some}_0(\ti)=Q^{\some}_0(i_1)\ \ \ \quad\text{and}\quad\ \ \ \tQ^{\some}_0(\textbf{i})=\tQ^{\some}_0(i_1)=\bigl(Q^{\some}_0(i_1)\bigr)^{-1}\ .$$

We assume that the following additional conditions are satisfied (where $\Gamma_0$ is defined in Definition \ref{defI} and $r_0$ acting on formal power series in $y^{\some}_1$ replaces $y^{\some}_1$ by $-y^{\some}_{1}$):
\begin{equation}\label{eq-Q4}
Q_0^{\some}(\textbf{i})=i_1^{-1}\bigl( p^{-1}\oX^{\some}_1(\ti)-pX^{\some}_{1}(\ti) \bigr)\ \ \ \ \ \textrm{ if } i_1^{-1}=i_1\,,
\end{equation}
\begin{equation}\label{eq-Q5}
{^{r_0}}Q^{\some}_0\bigl(r_0(\textbf{i})\bigr)Q^{\some}_0(\textbf{i})=\left\lbrace
\begin{array}{ll}
\Gamma_0^{\some}(\ti) & \textrm{ if } i^{-1}_1 \stackrel{p}\nleftrightarrow i_{1} \vspace*{0.6em}\\
\Gamma_0^{\some}(\ti)\big( y_{1}^{\some}\big)^{-1} & \textrm{ if }i^{-1}_1 \stackrel{p}\longrightarrow i_{1} \vspace*{0.6em}\\
-\Gamma_0^{\some}(\ti)\big(y_{1}^{\some} \big)^{-1} & \textrm{ if }i^{-1}_1 \stackrel{p}\longleftarrow i_{1} \vspace*{0.6em}\\
-\Gamma_0^{\some}(\ti)\big( y_{1}^{\some}\big)^{-2} & \textrm{ if }i^{-1}_1 \stackrel{p}\longleftrightarrow i_{1} 
\end{array}
\right.
\end{equation}
\begin{equation}\label{eq-Q6}
{}^{r_1r_{0}r_1}Q^{\some}_0\bigl(\ti)=Q^{\some}_{0}\bigl(\ti)\ \ \ \ \ \text{and}\ \ \ \ \ \ {}^{r_0r_{1}r_0}Q^{\some}_1\bigl(r_0r_1r_0(\textbf{i})\bigr)=Q^{\some}_{1}(\textbf{i})\ .
\end{equation}

\begin{lemma}\label{lem-QB}
There exists such a family satisfying Conditions (\ref{eq-Q1})--(\ref{eq-Q3}) and (\ref{eq-Q4})--(\ref{eq-Q6}).
\end{lemma}
\begin{proof} For $Q^{\some}_k(\ti)$ with $k\neq 0$, we take the same formulas (\ref{formulaQk}) as in the proof of the corresponding lemma \ref{lem-QA} and we thus already know that Conditions (\ref{eq-Q1})--(\ref{eq-Q3}) are satisfied.

It is easy to see that the following formulas
\begin{equation*}
Q_0^{\some}(\textbf{i}):=\left\lbrace
\begin{array}{ll}
i_1^{-1}\bigl( p^{-1}\oX^{\some}_1(\ti)-pX^{\some}_{1}(\ti) \bigr) & \textrm{ when } i^{-1}_1=i_1 \vspace*{1em}\\
\displaystyle\frac{p \oX^{\some}_1(\ti)-p^{-1}X^{\some}_{1}(\ti)}{\oX^{\some}_1(\ti)-X^{\some}_{1}(\ti)} & \textrm{ when } i^{-1}_1 \stackrel{p}\nleftrightarrow i_{1} \vspace*{1em}\\
\displaystyle\frac{(p X^{\some}_{1}(\ti)-p^{-1}\oX^{\some}_{1}(\ti))\,p\,i_1^{-1}}{(\oX^{\some}_1(\ti)-X^{\some}_{1}(\ti))(X^{\some}_{1}(\ti)-\oX^{\some}_{1}(\ti))}\quad & \textrm{ when }i^{-1}_1 \stackrel{p}\longrightarrow i_{1} \vspace*{1em}\\
1 & \textrm{ when }i^{-1}_1 \stackrel{p}\longleftarrow i_{1} \vspace*{1em}\\
\displaystyle\frac{p\,i_1^{-1}}{\oX^{\some}_1(\ti)-X^{\some}_{1}(\ti)} & \textrm{ when }i^{-1}_1 \stackrel{p}\longleftrightarrow i_{1}
\end{array}\right.
\end{equation*}
indeed define invertible power series
(in particular we recall that $p^2\neq 1$), where the dependence on $\ti$ of $Q_0^{\some}(\textbf{i})$ is only through $i_1$. Moreover, one sees that, by definition of $g$, 
$$\begin{array}{ll}
p \oX^{\some}_1(\ti)-p^{-1}X^{\some}_{1}(\ti)\displaystyle=p^{-1}i_1 \Bigl(\frac{1}{1-g(y^{\some}_{1})}-(1-g(y^{\some}_1))\Bigr)=p^{-1}i_1y^{\some}_1 \quad & \text{if $i_{1}=p^2 i_1^{-1}$,}\\[0.5em]
p X^{\some}_1(\ti)-p^{-1}\oX^{\some}_{1}(\ti)\displaystyle=p^{-1}i^{-1}_1 \Bigl((1-g(y^{\some}_1))-\frac{1}{1-g(y^{\some}_{1})}\Bigr)=-p^{-1}i^{-1}_1y^{\some}_1 \quad & \text{if $i_{1}=p^{-2} i_1^{-1}$.}
\end{array}$$
So, using the defining formula in Definition \ref{defI} for $\Gamma_0$, it is now easy to see that the properties (\ref{eq-Q4}) and (\ref{eq-Q5}) are satisfied. The first relation in (\ref{eq-Q6}) is obvious since ${}^{r_1}Q^{\some}_0\bigl(\ti)$ is a power series in $y_2$ only and therefore is stable under the action of $r_0$.

The second equality in (\ref{eq-Q6}) is more difficult to check. First, note that the action of $r_0r_1r_0$ on $\ti$ transforms $(i_1,i_2)$ into $(i_2^{-1},i_1^{-1})$. Similarly, the action of $r_0r_1r_0$ on $K[[y_1,y_2]]$ replaces $y_1$ by $-y_2$ and $y_2$ by $-y_1$. Then, using Lemma \ref{lem-fg} for the formal power series $g$, we find the following formulas:
\begin{multline*}
{}^{r_0r_1r_0}\bigl(g(y_2)-g(y_1)\bigr)=g(-y_1)-g(-y_2)=1-g(-y_2)-(1-g(-y_1))\\
=\frac{1}{1-g(y_2)}-\frac{1}{1-g(y_1)}=\frac{g(y_2)-g(y_1)}{(1-g(y_1))(1-g(y_2))}\ ,
\end{multline*}
\[{}^{r_0r_1r_0}X_1(r_0r_1r_0(\ti))=i_2^{-1}(1-g(-y_2))=\frac{i_2^{-1}}{1-g(y_2)}=\frac{X_1(\ti)}{i_1i_2(1-g(y_1))(1-g(y_2))}\ ,\]
\[{}^{r_0r_1r_0}X_2(r_0r_1r_0(\ti))=i_1^{-1}(1-g(-y_1))=\frac{i_1^{-1}}{1-g(y_1)}=\frac{X_2(\ti)}{i_1i_2(1-g(y_1))(1-g(y_2))}\ .\]
With this, a direct case-by-case inspection of Formulas (\ref{formulaQk}) allows us to check the second property in (\ref{eq-Q6}).
\end{proof}

Now we set, for $k=0,1,\dots,n-1$,
\[\psi_k^{\cI}:=\sum_{\textbf{i}\in \beta} \Phi_k^{\cI}\tQ_k^{\cI}(\textbf{i})\idmep\ \in\cI_{N} \ \ \ \ \ \text{and}\ \ \ \ \ \Phi_k^V:=\sum_{\textbf{i}\in \beta} \psi_k^VQ^V_k(\textbf{i})\idmep\ \in V^{\beta}_{\blambda,m}\ . \]
Finally we are ready to give the isomorphism. We consider the following maps, where $i\in\{1,\dots,n\}$, $k\in\{0,1,\dots,n-1\}$ and $\ti\in\beta$.
\begin{equation}\label{def-F-G}
F\ :\ \begin{array}{rcl}
V^{\beta}_{\blambda,m} & \longrightarrow & \cI_{N} \\[0.5em]
\idmep &\mapsto & \idmep \\[0.2em]
y_i &\mapsto & y_i^{\cI} \\[0.2em]
\psi_k &\mapsto & \psi_k^{\cI}
\end{array}
\ \ \ \ \ \text{and}\ \ \ \ \ 
G\ :\ \begin{array}{rcl}
\cI_{N} &\longrightarrow & V^{\beta}_{\blambda,m} \\[0.5em]
\idmep &\mapsto & \idmep \\[0.2em]
X_i &\mapsto & X_i^V\\[0.2em]
\Phi_k &\mapsto & \Phi_k^V
\end{array}\ .
\end{equation}
Assuming that $F$ and $G$ extend to algebra homomorphisms, we have $F\circ G=\text{Id}_{\cI_{N}}$ and $G\circ F=\text{Id}_{V^{\beta}_{\blambda,m}}$, since it is satisfied on the generators (we know this for all generators except $\psi_0$ and $\Phi_0$ from the proof in type A, see (\ref{def-F-GA}); it is immediate for $\psi_0$ and $\Phi_0$). Thus the proof will be finished when we show that $F$ and $G$ extend to algebra homomorphisms.

\paragraph{Proof that $F$ and $G$ extend to algebra homomorphisms.} The method is the same as in the type A situation. Most of the defining relations have already been checked at the corresponding stage of the proof of Theorem \ref{theoA}. We use the symbol $\some\in\{V,\cI\}$ as explained in Remark \ref{rem-symbol}.

\vskip .1cm
$\bullet$ \textbf{Relations (\ref{rel1-Phi1}) and (\ref{Rel:V3}) for $\boldsymbol{a=0}$.} The proof of these relations with $a>0$ used for Theorem \ref{theoA}, item 1, can be repeated without modification.

\vskip .1cm
$\bullet$ \textbf{Relations (\ref{rel1-Phi2}) with $\boldsymbol{a=0}$ and (\ref{Rel:V8})-(\ref{Rel:V9}).} Here, for clarity, we will treat each relation separately. We start with (\ref{Rel:V9}) in $\cI_{N}$. We have that $\Phi_0$ commutes with $X_j$ if $j>1$ and therefore, in $\cI_N$, we have
\[\psi^{\cI}_0y^{\cI}_je(\ti)=\Phi_0f(1-i_j^{-1}X_{j})\tQ^{\cI}_0(\ti)e(\ti)=f(1-i_j^{-1}X_{j})\Phi_0\tQ^{\cI}_0(\ti)e(\ti)=y^{\cI}_j\psi^{\cI}_0e(\ti)\ ,
\]
where we notice that $f(1-i_j^{-1}X_{j})e(r_0(\ti))=y^{\cI}_je(r_0(\ti))$ since $j>1$.

Then, using here the explicit expression of the power series $f$, we have $y^{\cI}_{0} e(\ti)=(i_1X_1^{-1}-i_1^{-1}X_1)e(\ti)$ and
$$y^{\cI}_{0} e(r_0(\ti))=(i_1^{-1}X_1^{-1}-i_1X_1)e(r_0(\ti))=-{}^{r_0}(i_1X_1^{-1}-i_1^{-1}X_1)e(r_0(\ti))\ ,$$
where the action of $r_0$ is to invert $X_1$. Therefore, we have in $\cI_{N}$, using (\ref{rel1-Phi2}), 
\begin{multline*}
\bigl(\psi^{\cI}_0y^{\cI}_1+y^{\cI}_{1}\psi_0^{\cI}\bigr)e(\ti)=\bigl(\Phi_0(i_1X_1^{-1}-i_1^{-1}X_1)-{}^{r_0}(i_1X_1^{-1}-i_1^{-1}X_1)\Phi_0\bigr)\tQ^{\cI}_0(\ti)e(\ti)\\
=\delta_{i_1,i^{-1}_{1}}(p X_{1}-p^{-1}X^{-1}_1)\frac{(i_1X_1^{-1}-i_1^{-1}X_1)-{}^{r_0}(i_1X_1^{-1}-i_1^{-1}X_1)}{X_{1}-X_1^{-1}}\tQ^{\cI}_0(\ti)e(\ti)\ .
\end{multline*}
This is equal to $0$ if $i_1^{-1}\neq i_{1}$. Now assume that $i_1^{-1}=i_{1}$. From Condition (\ref{eq-Q4}), we have $(p X_{1}-p^{-1}X^{-1}_1)\tQ^{\cI}_0(\ti)e(\ti)=-i_1$ and thus we find finally that 
\[\bigl(\psi^{\cI}_0y^{\cI}_1+y^{\cI}_{1}\psi_0^{\cI}\bigr)e(\ti)=2e(\ti)\ .\]

Now we deal with (\ref{rel1-Phi2}) for $a=0$ in $V^{\beta}_{\blambda,m}$. Let $\Pi$ be a Laurent polynomial in $X^V_1$, and let $P$ be the formal power series in $y_1$ defined by $P:=\Pi\bigl(i_1\bigl(1-g(y_1)\bigr)\bigr)$. Note that $\Pi e(\ti)=Pe(\ti)$ and that, using Lemma \ref{lem-fg},
\begin{multline*}
{}^{r_0}\Pi e(r_0(\ti))=\Pi(X^{-1}_1)e(r_0(\ti))=\Pi\bigl(i_1^{-1}\bigl(1-g(y_1)\bigr)^{-1}\bigr)e(r_0(\ti))=\Pi\bigl(i_1^{-1}\bigl(1-g(-y_1)\bigr)\bigr)e(r_0(\ti))={}^{r_0}Pe(r_0(\ti))\ ,
\end{multline*}
where ${}^{r_0}\Pi$ stands for replacing $X^V_1$ by its inverse in the Laurent polynomial $\Pi$, while ${}^{r_0}P$ stands for replacing $y_1$ by $-y_1$ in the power series $P$. Now, using Relation (\ref{rel:V8'}) in $V^{\beta}_{\blambda,m}$, we find:
\begin{multline*}
\bigl(\Phi^{V}_0\Pi-{}^{r_0}\Pi\Phi_0^{V}\bigr)e(\ti)=\bigl(\psi_0P-{}^{r_0}P\psi_0\bigr)Q^{V}_0(\ti)e(\ti)=\delta_{i_1,i_1^{-1}}\frac{P-{}^{r_0}P}{y_1}Q^{V}_0(\ti)e(\ti)\\
=\delta_{i_1,i_1^{-1}}(p X_{1}-p^{-1}X_1^{-1})\frac{\Pi-{}^{r_0}\Pi}{X_{1}-X_1^{-1}}e(\ti)\ ,
\end{multline*}
where we have noted that $e(\ti)=e(r_0(\ti))$ if $i_1=i^{-1}_{1}$, and we have again used Relation (\ref{eq2}) with upper indices $V$ instead of $\cI$.

At this point of the proof , we can start using the following relations, for any $\some\in\{V,\cI\}$ and any formal power series $P$ in $y^{\some}_1,\dots,y^{\some}_n$:
\begin{equation}\label{eqproof1}
\bigl(\Phi^{\some}_bP-{}^{r_b}P\Phi_b^{\some}\bigr)e(\ti)=\delta_{i_b,i_{b+1}}\frac{P-{}^{r_b}P}{y^{\some}_{b+1}-y^{\some}_b}Q^{\some}_b(\ti)e(\ti)\ .
\end{equation}
\begin{equation}\label{eqproof2}
\bigl(\Phi^{\some}_0P-{}^{r_0}P\Phi_0^{\some}\bigr)e(\ti)=\delta_{i_1,i^{-1}_{1}}\frac{P-{}^{r_0}P}{y^{\some}_{1}}Q^{\some}_0(\ti)e(\ti)\ .
\end{equation}

\vskip .1cm
$\bullet$ \textbf{Relations (\ref{rel1-Phi3}) and (\ref{Rel:V5}) with $\boldsymbol{a=0}$.} Let $b\in\{2,\dots,n-1\}$. Note that $Q^{\some}_b(\ti)\in K[[y^{\some}_b,y^{\some}_{b+1}]]$ so it commutes with $\Phi^{\some}_0$, and moreover the dependence on $\ti$ is only through $i_b$ and $i_{b+1}$ so we have $Q^{\some}_b(r_0(\ti))=Q^{\some}_b(\ti)$. Similarly, $Q^{\some}_0(\ti)\in K[[y^{\some}_1]]$ so it commutes with $\Phi^{\some}_b$, and moreover the dependence on $\ti$ is only through $i_1$ and so we have $Q^{\some}_0(r_b(\ti))=Q^{\some}_0(\ti)$. Thus we have
\[(\psi^{\some}_0\psi^{\some}_b-\psi^{\some}_b\psi^{\some}_0)e(\ti)=\bigl(\Phi^{\some}_0\tQ^{\some}_0(r_b(\ti))\Phi^{\some}_b\tQ^{\some}_b(\ti)-\Phi^{\some}_b\tQ^{\some}_b(r_0(\ti))\Phi^{\some}_0\tQ^{\some}_0(\ti)\bigr)e(\ti)=(\Phi^{\some}_0\Phi^{\some}_b-\Phi^{\some}_b\Phi^{\some}_0)\tQ^{\some}_b(\ti)\tQ^{\some}_0(\ti)e(\ti)\ ,\]
and we note that $\tQ^{\some}_b(\ti)\tQ^{\some}_0(\ti)$ is invertible to conclude the verification.

\vskip .1cm
$\bullet$ \textbf{Relations (\ref{rel1-Phi4}) with $\boldsymbol{a=0}$ and (\ref{Rel:V10}).} Assume first that $i_1\neq i^{-1}_{1}$. Using (\ref{eqproof2}), we have
\[(\psi^{\some}_0)^2e(\ti)=\Phi^{\some}_0\tQ^{\some}_0(r_0(\ti))\,\Phi^{\some}_0\tQ^{\some}_0(\ti)\,e(\ti)=(\Phi^{\some}_0)^2.{}^{r_0}\tQ^{\some}_0(r_0(\ti))\tQ^{\some}_0(\ti)\,e(\ti)\ .\]
Therefore, we have:
\[\Bigl((\psi^{\some}_0)^2-\Gamma^{\some}_0(\ti){^{r_0}}\tQ^{\some}_0(r_0(\ti))\tQ^{\some}_0(\ti)\Bigr)e(\ti)
=0\ \ \ \ \Leftrightarrow\ \ \ \ \Bigl((\Phi^{\some}_0)^2-\Gamma^{\some}_0(\ti)\Bigr){^{r_0}}\tQ^{\some}_0(r_0(\ti))\tQ^{\some}_0(\ti)e(\ti)=0 .\]
We note that, by definition, $\Gamma^{\some}_0(\ti)e(\ti)=\Gamma^{\some}_0e(\ti)$, so on the right we have (\ref{rel1-Phi4}) for $a=0$ multiplied by ${^{r_0}}\tQ^{\some}_0(r_0(\ti))\tQ^{\some}_0(\ti)$ which is invertible. And from Condition (\ref{eq-Q5}) on the family $Q^{\some}_0(\ti)$, we have precisely (\ref{Rel:V10}) on the left (in all cases where $i_1\neq i^{-1}_{1}$).

Now assume that $i_1= i^{-1}_{1}$. From (\ref{eq-Q4}), a direct calculation shows that
\[\bigl({}^{r_0}Q^{\some}_0(\ti)-Q^{\some}_0(\ti)\bigr)e(\ti)=i_1^{-1}(p+p^{-1})y^{\some}_1e(\ti)\ .\]
Therefore we have, using the commutation relation (\ref{eqproof2}),
\begin{multline*}
(\psi^{\some}_0)^2e(\ti)=\Phi^{\some}_0\,\tQ^{\some}_0(\ti)\,\Phi^{\some}_0\,\tQ^{\some}_0(\ti)\,e(\ti)=\Phi^{\some}_0\Bigl(\Phi^{\some}_0{}^{r_0}\tQ^{\some}_0(\ti)-Q^{\some}_0(\ti)\frac{\tQ^{\some}_0(\ti)-{}^{r_0}\tQ^{\some}_0(\ti)}{y^{\some}_{1}}\Bigr)\tQ^{\some}_0(\ti)\,e(\ti)\\
=\Phi^{\some}_0\Bigl(\Phi^{\some}_0-\frac{{}^{r_0}Q^{\some}_0(\ti)-Q^{\some}_0(\ti)}{y^{\some}_{1}}\Bigr){}^{r_0}\tQ^{\some}_0(\ti)\tQ^{\some}_0(\ti)\,e(\ti)=\Phi^{\some}_0\Bigl(\Phi^{\some}_0-(p+p^{-1})\Bigr){}^{r_0}\tQ^{\some}_0(\ti)\tQ^{\some}_0(\ti)\,e(\ti)\ ,
\end{multline*}
and we note again, to conclude the proof of both (\ref{rel1-Phi4}) with $a=0$ and (\ref{Rel:V10}) that ${}^{r_0}\tQ^{\some}_0(\ti)\tQ^{\some}_0(\ti)$ is invertible.

\vskip .1cm
$\bullet$ \textbf{Relations (\ref{rel1-Phi6}) and (\ref{Rel:V11}).} To treat the final four-term braid relations, we use a method similar to the one used for the three-term braid relations in paragraph ``\textbf{Relations (\ref{rel1-Phi5}) and (\ref{Rel:V7})}'' during the proof of Theorem \ref{theoA}. 

We will write $\psi^{\some}_0$, $\psi^{\some}_1$ in terms of $\Phi_0^{\some}$, $\Phi^{\some}_1$ and use all the already proved relations, especially the commutation relations (\ref{eqproof1})-(\ref{eqproof2}). To simplify notation, we note that $Q^{\some}_1(\ti)$ depends on $\ti$ only through $i_1$ and $i_{2}$, so we will write $Q^{\some}_1(i_1,i_{2})$ instead. Similarly, $Q^{\some}_0(\ti)$ depends on $\ti$ only through $i_1$, so we will write $Q^{\some}_0(i_1)$. We denote 
$$(i_1,i_2)=:(i,j)\ ,$$
and we will almost always omit the idempotent $e(\ti)$ on the right hand side of each line.

First we note that, by repeatedly using the commutation relations (\ref{eqproof1})-(\ref{eqproof2}), we will obtain
\begin{multline}\label{dev1}
\psi^{\some}_0\psi^{\some}_1\psi^{\some}_0\psi^{\some}_1=\Phi^{\some}_0\tQ^{\some}_0(i)\ \Phi^{\some}_1\tQ^{\some}_1(j^{-1}i)\ \Phi^{\some}_0\tQ^{\some}_0(j)\ \Phi^{\some}_1\tQ^{\some}_1(ij)\\
=\Phi^{\some}_0\Phi^{\some}_1\Phi^{\some}_0\Phi^{\some}_1\ {}^{r_1r_0r_1}\tQ^{\some}_0(i)\,{}^{r_1r_0}\tQ^{\some}_1(j^{-1}i)\,{}^{r_1}\tQ^{\some}_0(j)\,\tQ^{\some}_1(ij)\ +\ \dots
\end{multline}
\begin{multline}\label{dev2}
\psi^{\some}_1\psi^{\some}_0\psi^{\some}_1\psi^{\some}_0=\Phi^{\some}_1\tQ^{\some}_1(j^{-1}i^{-1})\ \Phi^{\some}_0\tQ^{\some}_0(j)\ \Phi^{\some}_1\tQ^{\some}_1(i^{-1}j)\ \Phi^{\some}_0\tQ^{\some}_0(i)\\
=\Phi^{\some}_0\Phi^{\some}_1\Phi^{\some}_0\Phi^{\some}_1\ {}^{r_0r_1r_0}\tQ^{\some}_1(j^{-1}i^{-1})\,{}^{r_0r_1}\tQ^{\some}_0(j)\,{}^{r_0}\tQ^{\some}_1(i^{-1}j)\,\tQ^{\some}_0(i)\ +\ \dots
\end{multline}
where in each case, the dots indicate terms with at most three occurrences of $\Phi^{\some}_0,\Phi^{\some}_1$. The crucial fact here is that
\[{}^{r_1r_0r_1}\tQ^{\some}_0(i)\,{}^{r_1r_0}\tQ^{\some}_1(j^{-1}i)\,{}^{r_1}\tQ^{\some}_0(j)\,\tQ^{\some}_1(ij)={}^{r_0r_1r_0}\tQ^{\some}_1(j^{-1}i^{-1})\,{}^{r_0r_1}\tQ^{\some}_0(j)\,{}^{r_0}\tQ^{\some}_1(i^{-1}j)\,\tQ^{\some}_0(i)=:\bold{Q}^{\some}\ .\]
This equality is obtained by using the two conditions in (\ref{eq-Q6}) two times each. So now we set
\[B^{\some}e(\ti):=\Bigl(\psi^{\some}_0\psi^{\some}_1\psi^{\some}_0\psi^{\some}_1-\psi^{\some}_1\psi^{\some}_0\psi^{\some}_1\psi^{\some}_0-(\Phi^{\some}_0\Phi^{\some}_1\Phi^{\some}_0\Phi^{\some}_1-\Phi^{\some}_1\Phi^{\some}_0\Phi^{\some}_1\Phi^{\some}_0)\bold{Q}^{\some}\Bigr)e(\ti)\ ,\]
and our goal is to calculate $B^{\some}e(\ti)$ and to check that, for any $\ti$, this is equal to the right hand side of (\ref{Rel:V11}) minus the right hand side of (\ref{rel1-Phi6}) multiplied by $\bold{Q}^{\some}$ on the right. Since $\bold{Q}^{\some}$ is invertible, this will prove both (\ref{rel1-Phi6}) and (\ref{Rel:V11}) (by first taking $\some=V$ and then $\some=\cI$).

There are several cases to consider (listed below), and in each case, this is a straightforward but lengthy calculation. Note that these calculations do not use an explicit expression for the elements $Q^{\some}_b(\ti)$ but only the imposed Conditions (\ref{eq-Q1})--(\ref{eq-Q3}) and (\ref{eq-Q4})--(\ref{eq-Q6}).
 
The different cases to consider are related to the different possible orbits of $\ti$ under the action of the Weyl group of type $B_2$ generated by $r_0$ and $r_1$. Here are the different possibilities for the orbit $O_{\ti}$ (all elements listed in a given orbit are assumed to be distinct):
\begin{enumerate}
\item $O_{\ti}=\{\ti\}$ (the stabiliser is of order 8) : this is Case 1;
\item $O_{\ti}=\{\ti, r_1(\ti)\}$ (the stabiliser of $\ti$ is generated by $r_0$ and $r_1r_0r_1$): this is Case 2;
\item $O_{\ti}=\{\ti, r_1(\ti),r_0(\ti),r_1r_0(\ti)\}$ (the stabiliser of $\ti$ is generated by $r_1r_0r_1$): this is Case 3;
\item $O_{\ti}=\{\ti, r_0(\ti),r_1(\ti),r_0r_1(\ti)\}$ (the stabiliser of $\ti$ is generated by $r_0r_1r_0$): this is Case 4;
\item $O_{\ti}=\{\ti, r_1(\ti),r_0r_1(\ti),r_1r_0r_1(\ti)\}$ (the stabiliser of $\ti$ is generated by $r_0$);
\item $O_{\ti}=\{\ti, r_0(\ti),r_1r_0(\ti),r_0r_1r_0(\ti)\}$ (the stabiliser of $\ti$ is generated by $r_1$);
\item $O_{\ti}$ contains $8$ elements.
\end{enumerate} 
Note that an orbit $O_{\ti}=\{\ti, r_0(\ti)\}$ (the stabiliser of $\ti$ is generated by $r_1$ and $r_0r_1r_0$) is impossible since this would mean that $i_1^2\neq 1$, while $i_1=i_2$ and $i_1^{-1}=i_2$.

We will repeatedly use the commutation relations (\ref{eqproof1})-(\ref{eqproof2}). The right hand sides in these relations, which appear in the calculation depending on conditions on $i$ and $j$, produce the ``correcting'' terms indicated by dots in (\ref{dev1})-(\ref{dev2}).
\\
\\
\textbf{Notation.} For convenience, we drop the symbol $\some$ until the end of this verification.
\\
\\
\textbf{Case 7: $O_{\ti}$ contains $8$ elements.} This case is exactly when there is no correcting term in (\ref{dev1})-(\ref{dev2}), so we find here immediately that $B^\some e(\ti)=0$.
\\
\\
\textbf{Case 6: $O_{\ti}=\{\ti, r_0(\ti),r_1r_0(\ti),r_0r_1r_0(\ti)\}$.} In this case we have $i=j$ and $i^2\neq 1$. Correcting terms can only possibly come when applying the commutation relation to the last $\Phi_1$ in (\ref{dev1}). However we have:
\[\psi_0\psi_1\psi_0\psi_1=\Phi_0\Phi_1\Phi_0\ {}^{r_0r_1}\tQ_0(i){}^{r_0}\tQ_1(i^{-1}i)\tQ_0(i)\ \Phi_1\tQ_1(ii)\ ,\]
and, due to the Conditions (\ref{eq-Q6}) satisfied by the $Q$'s, it turns out that ${}^{r_0r_1}\tQ_0(i){}^{r_0}\tQ_1(i^{-1}i)\tQ_0(i)$ is invariant by $r_1$ and therefore commutes with $\Phi_1$. So finally there is no correcting term and we have $B^\some e(\ti)=0$.
\\
\\
\textbf{Case 5: $O_{\ti}=\{\ti, r_1(\ti),r_0r_1(\ti),r_1r_0r_1(\ti)\}$.} In this case we have $i\neq j$, $i^2=1$, $j^2\neq 1$. Correcting terms can only possibly come when applying the commutation relation to the last $\Phi_0$ in (\ref{dev2}) and similarly to the previous case, one finds that there is no correcting term and we have $B^\some e(\ti)=0$.
\\
\\
\textbf{Case 4:} $O_{\ti}=\{\ti, r_0(\ti),r_1(\ti),r_0r_1(\ti)\}$. In this case we have $i=j^{-1}$ and $ i^2\neq 1$. In (\ref{dev1}), the only correcting term comes when applying the commutation relation to the first $\Phi_1$. We have:
\begin{multline*}
\psi_0\psi_1\psi_0\psi_1-\Phi_0\Phi_1\Phi_0\Phi_1\textbf{Q}=\Phi_0\frac{\tQ_0(i)-{}^{r_1}\tQ_0(i)}{y_2-y_1}\psi_0\psi_1\\
= \psi_0^2\psi_1\frac{1}{y_1+y_2}\ -\ \Phi_0^2\Phi_1\,{}^{r_1r_0r_1}\tQ_0(i){}^{r_1}\tQ_0(j)\tQ_1(ij)\frac{1}{y_1+y_2}= \psi_0^2\psi_1\frac{1}{y_1+y_2}\ -\ \Phi_1{}^{r_1}\Gamma_0\frac{{}^{r_0}Q_1(jj)}{y_1+y_2}\,\textbf{Q}\ .
\end{multline*}
In (\ref{dev2}), the only correcting term comes when applying the commutation relation to the second $\Phi_1$:
\begin{multline*}
\psi_1\psi_0\psi_1\psi_0-\Phi_1\Phi_0\Phi_1\Phi_0\textbf{Q}=\psi_1\Phi_0\frac{\tQ_0(j)-{}^{r_1}\tQ_0(j)}{y_2-y_1}\psi_0\\
= \psi_1\psi_0^2\frac{1}{y_1+y_2}\ -\ \Phi_1\Phi_0^2\,\tQ_1(ij){}^{r_0r_1}\tQ_0(j)\tQ_0(i)\frac{1}{y_1+y_2}= \psi_0^2\psi_1\frac{1}{y_1+y_2}\ -\ \Phi_1\Gamma_0\frac{{}^{r_0}Q_1(jj)}{y_1+y_2}\,\textbf{Q}\ .
\end{multline*}
Once we note, using Condition (\ref{eq-Q1}) on $Q_1(jj)$, that $\displaystyle\frac{{^{r_0}}Q_1(jj)}{y_1+y_2}=\frac{q^{-1}X^{-1}_2-qX_1}{X^{-1}_2-X_1}$, combining these two calculations, we find
\begin{equation*}
B^\some e(\ti) = \big( \psi_0^2\psi_1-\psi_1\psi_0^2 \big)\frac{1}{y_1+y_2}\ +\ \Phi_1(\Gamma_0-{^{r_1}}\Gamma_0) \frac{q^{-1}X^{-1}_2-qX_1}{X^{-1}_2-X_1}\textbf{Q}.
\end{equation*}
The coefficient in front of $\Phi_1$ is what is needed and the verification is concluded by the following formulas which are easy to check:
\[\big( \psi_0^2\psi_1-\psi_1\psi_0^2 \big)\frac{1}{y_1+y_2}e(\ti)=\left\lbrace
\begin{array}{ll}
0 & \textrm{ if } i\stackrel{p}{\nleftrightarrow} i^{-1}\,, \\[0.2em]
-\psi_1\idmep & \textrm{ if } i^{-1} \stackrel{p}{\longrightarrow} i\,,\\[0.2em]
\psi_1\idmep & \textrm{ if }  i^{-1} \stackrel{p}{\longleftarrow} i\,,\\[0.2em]
\psi_1(y_1-y_2)\idmep & \textrm{ if } i^{-1} \stackrel{p}{\longleftrightarrow} i\,.
\end{array}\right.\] 
\\
\\
\textbf{Case 2 and Case 3: $O_{\ti}=\{\ti, r_1(\ti)\}$ or $O_{\ti}=\{\ti, r_1(\ti),r_0(\ti),r_1r_0(\ti)\}$}. In these cases we have $i\neq j$ and $j^2=1$. The two cases differ by whether $i^2=1$ (Case 2) or $i^2\neq 1$ (Case 3).
In (\ref{dev1}), the only correcting term comes when applying the commutation relation to the second $\Phi_0$. Performing calculations in a similar way than just above, we find:
\[
\psi_0\psi_1\psi_0\psi_1-\Phi_0\Phi_1\Phi_0\Phi_1\textbf{Q}= \psi_0\psi_1^2\frac{1}{y_2}\ -\ \Phi_0\Gamma_1\frac{{}^{r_1}Q_0(j)}{y_2}\,\textbf{Q}\ .\]
In (\ref{dev2}), the correcting term can come when applying the commutation relation to both $\Phi_0$. We start the calculation as follows:
\[\psi_1\psi_0\psi_1\psi_0=\Phi_1\Phi_0\Phi_1{}^{r_1r_0}\tQ_1(ji^{-1}){}^{r_1}\tQ_0(ji^{-1})\tQ_1(i^{-1}j)\Phi_0\tQ_0(i)\ +\      \Phi_1\frac{\tQ_1(ji^{-1})-{}^{r_0}\tQ_1(ji^{-1})}{y_1}\psi_1\psi_0\ .\]
For the first term, we note that, if $i^2=1$, the factor between $\Phi_1\Phi_0\Phi_1$ and $\Phi_0$ commutes with $\Phi_0$ thanks to Conditions (\ref{eq-Q6}) so that there is no correcting term when moving it through $\Phi_0$. If $i^2\neq1$ there is no correcting term either and we conclude that, in both cases, this first term is equal to $\Phi_1\Phi_0\Phi_1\Phi_0\textbf{Q}$.

Therefore, we have
\[\psi_1\psi_0\psi_1\psi_0-\Phi_1\Phi_0\Phi_1\Phi_0\textbf{Q}=\Phi_1\frac{\tQ_1(ji^{-1})-{}^{r_0}\tQ_1(ji^{-1})}{y_1}\psi_1\psi_0
= \psi_1^2\psi_0\frac{1}{y_2}\,-\,\Phi_1^2\,\frac{{}^{r_1r_0}\tQ_1(ji^{-1})\tQ_1(i^{-1}j)}{y_2}\Phi_0\tQ_0(i)\ ,\]
and we note, as explained above, that in both cases ($i^2=1$ and $i^2\neq1$) moving ${}^{r_1r_0}\tQ_1(ji^{-1})\tQ_1(i^{-1}j)$ through $\Phi_0$ produces no correcting term. So we find
\begin{multline*}
\psi_1\psi_0\psi_1\psi_0-\Phi_1\Phi_0\Phi_1\Phi_0\textbf{Q} = \psi_1^2\psi_0\frac{1}{y_2}\ -\ \Gamma_1\,\Phi_0\frac{{}^{r_0r_1r_0}\tQ_1(ji^{-1})\tQ_1(i^{-1}j)\tQ_0(i)}{y_2}\\
=\psi_1^2\psi_0\frac{1}{y_2}-\ \Gamma_1\,\Phi_0\frac{{}^{r_1}Q_0(j)}{y_2}\textbf{Q}\\
= \psi_1^2\psi_0\frac{1}{y_2}\ -\ \Phi_0\,{}^{r_0}\Gamma_1\frac{{}^{r_1}Q_0(j)}{y_2}\textbf{Q}\ -\ 
\delta_{i,i^{-1}}\frac{\Gamma_1-{}^{r_0}\Gamma_1}{y_1}Q_0(i)\frac{{}^{r_1}Q_0(j)}{y_2}\textbf{Q}\ .
\end{multline*}
Once we note that, using Condition (\ref{eq-Q4}) on $Q_0(j)$ and $Q_0(i)$ if $i^2=1$, we have
\begin{equation*}
\frac{{^{r_1}}Q_0(j)}{y_2} = \frac{pX_2-p^{-1}X^{-1}_2}{X_2-X^{-1}_2}\ \ \ \ \ \text{and}\ \ \ \ \ \frac{Q_0(i)}{y_1}=\frac{pX_1-p^{-1}X_1^{-1}}{X_1-X_1^{-1}}\ \text{if $i^2=1$,}
\end{equation*}
we can combine the two calculations performed above to find
\begin{equation*}
B^\some e(\ti) = \big( \psi_0\psi_1^2-\psi_1^2\psi_0 \big)\frac{1}{y_2}\ -\ \Bigl(\bigl(\Phi_0+\delta_{i,i^{-1}}\frac{pX_1-p^{-1}X_1^{-1}}{X_1-X_1^{-1}}\bigr)(\Gamma_1-{}^{r_0}\Gamma_1) \frac{pX_2-p^{-1}X^{-1}_2}{X_2-X^{-1}_2}\Bigr)\textbf{Q}\ .
\end{equation*}
The factor in front of $\textbf{Q}$ is what is needed (in both Cases 2 and 3). 

Now assume that we have $i^2\neq1$ (Case 3). Note that here $i \leftrightarrow j$ (that is $q^2=-1$ and $i=-j$) is impossible since we have $i^2\neq1$ while $j^2=1$. Then, the verification is concluded in this case using the following formulas which are easy to check:
\[\big( \psi_0\psi_1^2-\psi_1^2\psi_0 \big)\frac{1}{y_2}e(\ti)=\left\lbrace
\begin{array}{ll}
0 & \textrm{ if } i\nleftrightarrow j\,, \\[0.2em]
2\psi_0\idmep & \textrm{ if } i \rightarrow j\,,\\[0.2em]
-2\psi_0\idmep & \textrm{ if }  i \leftarrow j\,.
\end{array}\right.\] 
Then assume that we have $i^2=1$ (Case 2). Then we have $i=-j$ and therefore, either $i\nleftrightarrow j$ or $i \leftrightarrow j$, and we have
\[\big( \psi_0\psi_1^2-\psi_1^2\psi_0 \big)\frac{1}{y_2}e(\ti)=\left\lbrace
\begin{array}{ll}
0 & \textrm{ if } i\nleftrightarrow j\,, \\[0.2em]
4(\psi_0y_1-1)\idmep & \textrm{ if }  i \leftrightarrow j\,.
\end{array}\right.\] 
This concludes the verification in these cases.
\\
\\
\textbf{Case 1:} $O_{\ti}=\{\ti\}$. In this case we have $i=j$ and $i^2=j^2=1$. This is the most difficult case computationally since correcting terms may appear for every commutation relation we use. We sketch the main steps of the calculation. To save space, we set $Q_0:=Q_0(i)$, $\tQ_0=Q_0^{-1}$ and $Q_1:=Q_1(ii)$, $\tQ_1=Q_1^{-1}$. We recall that in this case $Q_0$ and $Q_1$ have explicit expressions given in (\ref{eq-Q1}) and (\ref{eq-Q4}), that we recall here in terms of $X$'s (the expression for $Q_1$ comes from the explicit choice of the power series $f$ and $g$):
\[Q_0=i\bigl( p^{-1}X^{-1}_1-pX_{1} \bigr)\ \ \ \ \ \text{and}\ \ \ \ \ Q_1=i\big( q^{-1}X_1-qX_{2} \big)\bigl(1+X_1^{-1}X_2^{-1})\,.\]
We recall that we have in this case
\begin{equation}\label{C1-a}
\psi_1^2=\Phi_1\tQ_1\Phi_1\tQ_1=0\ \ \ \ \ \text{and}\ \ \ \ \ \psi_0^2=\Phi_0\tQ_0\Phi_0\tQ_0=0\ .
\end{equation}
We will also use the following particular instances of the commutation relation between $\Phi$'s and $X$'s, which are obtained directly from (\ref{rel1-Phi2}),
\begin{equation}\label{C1-b}
\tQ_1\Phi_0-\Phi_0{}^{r_0}\tQ_1=(q+q^{-1}){}^{r_0}\tQ_1\tQ_1Q_0\ ,
\end{equation}
\begin{equation}\label{C1-c}
\Phi_1\tQ_0-{}^{r_1}\tQ_0\Phi_1=\frac{p+p^{-1}X_1^{-1}X_2^{-1}}{1+X_1^{-1}X_2^{-1}}{}^{r_1}\tQ_0\tQ_0Q_1\ ,
\end{equation}
\begin{equation}\label{C1-d}
\frac{p+p^{-1}X_1^{-1}X_2^{-1}}{1+X_1^{-1}X_2^{-1}}\Phi_0  -  \Phi_0\frac{p+p^{-1}X_1X_2^{-1}}{1+X_1X_2^{-1}}=(p-p^{-1})\frac{q-q^{-1}X_1^{-1}X_2^{-1}}{1+X_1^{-1}X_2^{-1}}{}^{r_0}\tQ_1Q_0\ .
\end{equation}
Then we calculate as follows, applying first (\ref{C1-c}),
\begin{multline*}
\psi_0\psi_1\psi_0\psi_1=\Phi_0\tQ_0\Phi_1\tQ_1\Phi_0\tQ_0\Phi_1\tQ_1\\
=\ (A:=)\ \Phi_0\Phi_1{}^{r_1}\tQ_0\tQ_1\Phi_0\tQ_0\Phi_1\tQ_1\ \ +\ (B:=)\ \Phi_0\frac{p+p^{-1}X_1^{-1}X_2^{-1}}{1+X_1^{-1}X_2^{-1}}{}^{r_1}\tQ_0\tQ_0\Phi_0\tQ_0\Phi_1\tQ_1\ .
\end{multline*}
Then we use that ${}^{r_1}\tQ_0$ commutes with $\Phi_0$ thanks to the condition (\ref{eq-Q6}) on $Q_0$ and (\ref{C1-b}) to find
\[A=\Phi_0\Phi_1\Phi_0{}^{r_0}\tQ_1{}^{r_1}\tQ_0\tQ_0\Phi_1\tQ_1+(q+q^{-1})\Phi_0\Phi_1{}^{r_0}\tQ_1\tQ_1{}^{r_1}\tQ_0\Phi_1\tQ_1\ .\]
Using condition (\ref{eq-Q6}) on $Q_1$ we have that ${}^{r_0}\tQ_1$ commutes with $\Phi_1$ (as well as ${}^{r_1}\tQ_0\tQ_0$), so there is no correcting term coming from the first term in the sum above. For the second term, applying the commutation relation (\ref{C1-c}) between ${}^{r_1}\tQ_0$ and the second $\Phi_1$, the first resulting term contains a factor $\Phi_1\tQ_1\Phi_1\tQ_1$ equal to 0 thanks to (\ref{C1-a}) and therefore only the correcting term remains. Thus, we find:
\[A=\Phi_0\Phi_1\Phi_0\Phi_1\textbf{Q}-(q+q^{-1})\Phi_0\Phi_1\frac{p+p^{-1}X_1^{-1}X_2^{-1}}{1+X_1^{-1}X_2^{-1}}\textbf{Q}\ .\]
Concerning the term $B$, we note again that ${}^{r_1}\tQ_0$ commutes with $\Phi_0$ thanks to the condition (\ref{eq-Q6}) and then we apply the commutation relation (\ref{C1-d}) between $\displaystyle\frac{p+p^{-1}X_1^{-1}X_2^{-1}}{1+X_1^{-1}X_2^{-1}}$ and the second $\Phi_0$. In the resulting expression, the first term contains a factor $\Phi_0\tQ_0\Phi_0\tQ_0$ equal to 0 thanks to (\ref{C1-a}) and therefore only the correcting term remains. We find at the end:
\[B=(p-p^{-1})\Phi_0\,\frac{q-q^{-1}X_1^{-1}X_2^{-1}}{1+X_1^{-1}X_2^{-1}}{}^{r_0}\tQ_1{}^{r_1}\tQ_0\tQ_0\,\Phi_1\tQ_1=(p-p^{-1})\Phi_0\Phi_1\,\frac{q-q^{-1}X_1^{-1}X_2^{-1}}{1+X_1^{-1}X_2^{-1}}\textbf{Q}\,\]
where we noted for the last equality that the factor between $\Phi_0$ and $\Phi_1$ commutes with $\Phi_1$, thanks to Condition (\ref{eq-Q6}) again. Finally, summing $A$ and $B$, we conclude with a small calculation that:
\[\psi_0\psi_1\psi_0\psi_1=\Bigl(\Phi_0\Phi_1\Phi_0\Phi_1-(pq^{-1}+p^{-1}q)\Phi_0\Phi_1\Bigr)\textbf{Q}\ .\]
Then, a very similar calculation shows that
\[\psi_1\psi_0\psi_1\psi_0=\Bigl(\Phi_1\Phi_0\Phi_1\Phi_0-(pq^{-1}+p^{-1}q)\Phi_1\Phi_0\Bigr)\textbf{Q}\ ,\]
which allows us to conclude that the following required relation is satisfied:
\[B^\some e(\ti)=-(pq^{-1}+p^{-1}q)(\Phi_0\Phi_1-\Phi_1\Phi_0)\textbf{Q}\ .\]

\paragraph{Proof of Proposition \ref{prop-iso2}.}

We again use the notation $Y^{-1}e^{H}(\ti)$, similar to the one introduced before Definition \ref{defI} and used in the proof of Theorem \ref{theoA}. For example, $(1-X^{-\alpha_a})^{-1}e^H(\ti)$ is defined if $r_a(\ti)\neq \ti$.

Let $a\in\{0,1,\dots,n-1\}$. We define elements of $e_{\beta}H(B_n)_{\blambda,m}$:
\[\Phi^H_a=g_ae_{\beta}-\sum_{\begin{array}{c}\scriptstyle{\textbf{i}\in\beta}\\[-0.2em] \scriptstyle{r_a(\textbf{i})\neq \textbf{i}}\end{array}}(q_a-q_a^{-1})(1-X^{-\alpha_a})^{-1}e^H(\textbf{i})+\sum_{\begin{array}{c}\scriptstyle{\textbf{i}\in\beta}\\[-0.2em] \scriptstyle{r_a(\textbf{i})= \textbf{i}}\end{array}}q_a^{-1}e^H(\textbf{i})\]
and elements of $\cI_N$:
\[g^{\cI}_a:=\Phi_a+\!\!\!\sum_{\begin{array}{c}\scriptstyle{\textbf{i}\in\beta}\\[-0.2em] \scriptstyle{r_a(\textbf{i})\neq \textbf{i}}\end{array}}(q_a-q_a^{-1})(1-X^{-\alpha_a})^{-1}e(\textbf{i})-\!\!\!\sum_{\begin{array}{c}\scriptstyle{\textbf{i}\in\beta}\\[-0.2em] \scriptstyle{r_a(\textbf{i})= \textbf{i}}\end{array}}q_a^{-1}e(\textbf{i})\ .\]

We consider the following maps, where $i\in\{1,\dots,n\}$, $a\in\{0,1,\dots,n-1\}$ and $\ti\in\beta$.
\begin{equation}\label{def-rho-sigma}
\rho\ :\ \ \begin{array}{rcl}
\cI_N & \to & e_{\beta}H(B_n)_{\blambda,m}\\[0.5em]
e(\ti) & \mapsto & e^H(\ti)\\[0.2em]
X_i & \mapsto & X_i e_{\beta}\\[0.2em]
\Phi_a & \mapsto & \Phi^H_a
\end{array}\ ,
\ \ \ \ \ \ 
\sigma\ :\ \ \begin{array}{rcl}
\hH(B_n) & \to & \cI_N\\[0.5em]
X_i & \mapsto & X_i\\[0.2em]
g_a & \mapsto & g_a^{\cI}
\end{array}\ .
\end{equation}
The proposition follows immediately from the two following lemmas.
\begin{lemma}\label{lem1}
The map $\rho$ extends to a morphism of algebras from $\cI_N$ to $e_{\beta}H(B_n)_{\blambda,m}$. The map $\sigma$ extends to a morphism of algebras from $\hH(B_n)$ to $\cI_N$. 
\end{lemma}
We still denote by $\rho$ and  $\sigma$ the obtained morphisms of algebras. The canonical surjection from $\hH(B_n)$ to $e_{\beta}H(B_n)_{\blambda,m}$ appearing in the following lemma is the composition of the standard surjection from $\hH(B_n)$ to its quotient $H(B_n)_{\blambda,m}$ followed by the multiplication by $e_{\beta}$.
\begin{lemma}\label{lem2}
The map $\sigma$ factors through the canonical surjection $\hH(B_n)\to e_{\beta}H(B_n)_{\blambda,m}$, and the resulting map provides the two-sided inverse of $\rho$.
\end{lemma}
\begin{proof}[Proof of Lemma \ref{lem1}]
We will use the same method and notation as in the proof of Lemma \ref{lem1A}, the analogous lemma for type A. Thus, here as well, we will denote during the proof $\Phi_a^{\cI}:=\Phi_a$ and $e^{\cI}(\ti)=e(\ti)$ in $\cI_N$ and $g_a^H:=g_a$ in $\hH(B_n)$, and we will make use of a symbol $\some\in\{\cI,H\}$, in the same spirit as explained in Remark \ref{rem-symbol}.

We need to check that:
\begin{itemize}
\item[-] The defining relations of $\hH(B_n)$ are satisfied by the elements $g_a^{\cI}$ and $X_i$ in $\cI_N$. All the defining relations not involving $g_0$ were checked in Lemma \ref{lem1A} (taking $\alpha=\beta$);
\item[-] the defining relations of $\cI_N$ are satisfied by the elements $\Phi_a^H$, $X_ie_{\beta}$ and $e^H(\ti)$ in $e_{\beta}H(B_n)_{\blambda,m}$. All the defining relations not involving $\Phi_0$ were checked in Lemma \ref{lem1A} (taking $\alpha=\beta$).
\end{itemize}

For the defining relations of $\hH(B_n)$ involving $g_0$, except the braid relation  $g_0g_{1}g_0g_1=g_{1}g_0g_{1}g_0$, the verification is a direct repetition of the calculations in the proof of Lemma \ref{lem1A}. Similarly for the defining relations of $\cI_N$ involving $\Phi_0$, except the modified braid relation (\ref{rel1-Phi6}).

So it remains to check only the following relations.

$\bullet$ \textbf{Relations (\ref{rel1-Phi6}) and $\boldsymbol{g_0g_{1}g_0g_1=g_{1}g_0g_{1}g_0}$.} We again use the symbol $\some\in\{\cI,H\}$.

We claim that, for every $\ti\in\beta$,
\begin{equation}\label{eq1}
\bigl(\Phi^{\some}_0\Phi^{\some}_{1}\Phi^{\some}_0\Phi^{\some}_1-\Phi^{\some}_{1}\Phi^{\some}_0\Phi^{\some}_{1}\Phi^{\some}_0\bigr)e^{\some}(\ti) - \bigl(g_0^{\some}g^{\some}_{1}g^{\some}_0g^{\some}_1-g^{\some}_{1}g^{\some}_0g^{\some}_{1}g^{\some}_0\bigr)e^{\some}(\ti)\ 
\end{equation}
is equal to the right-hand side of (\ref{rel1-Phi6}) in terms of $\Phi^{\some}_0$ and $\Phi^{\some}_1$. This indeed implies that, first by taking $\some=H$, that (\ref{rel1-Phi6}) is satisfied by $\Phi_0^H$ and $\Phi_1^H$, and second by taking $\some=\cI$ that $g^{\cI}_0$ and $g^{\cI}_{1}$ satisfy the four-term braid relations.

We can use all the previously checked relations, namely all defining relations of $\cI_N$ and $\hH(B_n)$ except (\ref{rel1-Phi6}) and $g_0g_{1}g_0g_1=g_{1}g_0g_{1}g_0$. So this calculation is exactly the same in $\cI_N$ and in $e_{\beta}H(B_n)_{\blambda,m}$. Note moreover that we obviously have
\[\bigl(\Phi^{H}_0\Phi^{H}_{1}\Phi^{H}_0\Phi^{H}_1-\Phi^{H}_{1}\Phi^{H}_0\Phi^{H}_{1}\Phi^{H}_0\bigr)e^{H}(\ti) = \bigl(g_0g_{1}g_0g_1-g_{1}g_0g_{1}g_0\bigr)e^H(\ti)+\sum_{x,w} c_{x,w}g_wX^xe^H(\ti)\ ,\]
where the sum is over $x\in L$ and over words $g_w$ of length strictly less than 4 in $g_0$ and $g_1$. Therefore, checking Relation (\ref{rel1-Phi6}) for $\Phi_0^H$ and $\Phi_1^H$ in $e_{\beta}H(B_n)_{\blambda,m}$ at once implies our claim about (\ref{eq1}).

The different cases to consider are related to the different possible orbits of $\ti$ under the action of the Weyl group of type $B_2$ generated by $r_0$ and $r_1$. The 7 different possibilities were listed above in the proof of Proposition \ref{prop-iso1}, and we refer to this list for our different cases below.

\vskip .2cm 
Before starting the verification in each case, we recall the definition of the intertwining elements $U'_a$ in $\hH(B_n)^{loc}$ from Section \ref{subsec-int}, and we note that
\[\Phi^H_ae(\ti)=U'_ae(\ti)\ \ \ \ \ \text{if $r_a(\ti)\neq \ti$.}\]
We will use that in $\hH(B_n)^{loc}$, we have $U'_aX^x=X^{r_a(x)}U'_a$ for $x\in L$, $U'_0U'_1U'_0U'_1=U'_1U'_0U'_1U'_0$ and $(U'_a)^2=\Gamma_a$, with $\Gamma_a$ as in (\ref{rel1-Phi4}).

During the proof, set 
\[Br_{\Phi}:=\Phi^{H}_0\Phi^{H}_{1}\Phi^{H}_0\Phi^{H}_1-\Phi^{H}_{1}\Phi^{H}_0\Phi^{H}_{1}\Phi^{H}_0\ .\]
\textbf{1.} Case 1 is an easy and direct calculation. One finds that
\begin{multline*}
Br_{\Phi}e^{H}(\ti)\\
=\Bigl((g_0+p^{-1})(g_1+q^{-1})(g_0+p^{-1})(g_1+q^{-1})-(g_1+q^{-1})(g_0+p^{-1})(g_1+q^{-1})(g_0+p^{-1})\Bigr)e^H(\ti)
\\
=(p^{-1}q+pq^{-1})(g_0g_1-g_1g_0)e^H(\ti)\ ,
\end{multline*}
and on the other hand
\[\bigl(\Phi^{H}_0\Phi^{H}_{1}-\Phi^{H}_{1}\Phi^{H}_0\bigr)e^{H}(\ti)=\Bigl((g_0+p^{-1})(g_1+q^{-1})-(g_1+q^{-1})(g_0+p^{-1})\Bigr)e^H(\ti)=(g_0g_1-g_1g_0)e^H(\ti)\ .\]

\noindent\textbf{2. and 3.} In these cases, we have
\[Br_{\Phi}e^{H}(\ti)=\Bigl((g_0+C)U'_1(g_0+p^{-1})U'_1-U'_1(g_0+p^{-1})U'_1(g_0+C)\Bigr)e^H(\ti)\ ,\]
with $C=p^{-1}$ in Case 2 and $C=\displaystyle-\frac{p-p^{-1}}{1-X^{-\alpha_0}}$ in Case 3. 

Note first that in $\hH(B_n)^{loc}$, $C$ commutes with $U'_1(g_0+p^{-1})U'_1$. Indeed this is obvious for Case 2, and for Case 3 it follows from the fact that $r_0r_1(\alpha_0)=r_1(\alpha_0)$.

Note also that $g_0+p^{-1}=\displaystyle U'_0+p^{-1}+\frac{p-p^{-1}}{1-X^{-\alpha_0}}=U'_0+\frac{p-p^{-1}X^{-\alpha_0}}{1-X^{-\alpha_0}}$. So we need to calculate in $\hH(B_n)^{loc}$ the expression
\begin{equation}\label{eq-case2-1}
g_0U'_1\bigl(U'_0+\frac{p-p^{-1}X^{-\alpha_0}}{1-X^{-\alpha_0}}\bigr)U'_1-U'_1\bigl(U'_0+\frac{p-p^{-1}X^{-\alpha_0}}{1-X^{-\alpha_0}}\bigr)U'_1g_0\ .
\end{equation}
Now note that $X^{\alpha_0}$ commutes with $U'_1\bigl(U'_0+\frac{p-p^{-1}X^{-\alpha_0}}{1-X^{-\alpha_0}}\bigr)U'_1$ since $r_0r_1(\alpha_0)=r_1(\alpha_0)$, and therefore the expression (\ref{eq-case2-1}) is equal to
\[U'_0U'_1\bigl(U'_0+\frac{p-p^{-1}X^{-\alpha_0}}{1-X^{-\alpha_0}}\bigr)U'_1-U'_1\bigl(U'_0+\frac{p-p^{-1}X^{-\alpha_0}}{1-X^{-\alpha_0}}\bigr)U'_1U'_0\ .\]
Expanding this, we find
\begin{multline*}
U'_0(U'_1)^2\frac{p-p^{-1}X^{-r_1(\alpha_0)}}{1-X^{-r_1(\alpha_0)}}-(U'_1)^2\frac{p-p^{-1}X^{-r_1(\alpha_0)}}{1-X^{-r_1(\alpha_0)}}U'_0=U'_0(\Gamma_1-{}^{r_0}\Gamma_1)\frac{p-p^{-1}X^{-r_1(\alpha_0)}}{1-X^{-r_1(\alpha_0)}}\\
= U'_0(1-X^{-\alpha_0})Y_0\ ,
\end{multline*}
where $Y_0$ is defined after Relation (\ref{rel1-Phi6}). To conclude, we put back the idempotent $e^H(\ti)$ and express $U'_0$ in terms of $\Phi_0$ depending on which case (2 or 3) we are considering. In Case 2, we have $\Phi_0e^H(\ti)=\displaystyle (U'_0+\frac{p-p^{-1}X^{-\alpha_0}}{1-X^{-\alpha_0}})e^H(\ti)$ and in Case 3, we have $\Phi_0e^H(\ti)=U'_0e^H(\ti)$. So we find that (\ref{rel1-Phi6}) is satisfied in these cases. 

 \noindent\textbf{4.} Here we have
\[Br_{\Phi}e^{H}(\ti)=\Bigl(U'_0(g_1+q^{-1})U'_0U'_1-U'_1U'_0(g_1+q^{-1})U'_0\Bigr)e^H(\ti)\ ,\]
and we reproduce exactly the same reasoning as in Case 3 above, with the role of the indices $1$ and $0$ exchanged, as well as the role of $p$ and $q$. Thus we have 
\[Br_{\Phi}e^{H}(\ti)= -U'_1(\Gamma_0-{}^{r_1}\Gamma_0)\frac{q-q^{-1}X^{-r_0(\alpha_1)}}{1-X^{-r_0(\alpha_1)}}e^H(\ti)= U'_1(1-X^{-\alpha_1})Y_1e^H(\ti)\ ,\]
and we conclude that (\ref{rel1-Phi6}) is satisfied here as well since $\Phi_1e^H(\ti)=U'_1e^H(\ti)$.

\noindent\textbf{5.} Here we have
\[Br_{\Phi}e^{H}(\ti)=\Bigl((g_0+p^{-1})U'_1U'_0U'_1-U'_1U'_0U'_1(g_0+p^{-1})\Bigr)e^H(\ti)\ .\]
We have that $g_0=\displaystyle U'_0+\frac{p-p^{-1}}{1-X^{-\alpha_0}}$ commutes with $U'_1U'_0U'_1$. Indeed, $U'_0$ commutes with $U'_1U'_0U'_1$ from the braid relation for the intertwining elements, and $X^{\alpha_0}$ commutes with $U'_1U'_0U'_1$ since $r_0r_1(\alpha_0)=r_1(\alpha_0)$. So we have indeed in this case that $Br_{\Phi}e^{H}(\ti)=0$.

\noindent\textbf{6.} Here we have
\[Br_{\Phi}e^{H}(\ti)=\Bigl(U'_0U'_1U'_0(g_1+q^{-1})-(g_1+q^{-1})U'_0U'_1U'_0\Bigr)e^H(\ti)\ ,\]
and we reproduce exactly the same reasoning as in Case 5 above, with the role of the indices $1$ and $0$ exchanged, to find that $Br_{\Phi}e^{H}(\ti)=0$ here as well.

\noindent\textbf{7.} Here we have at once
\[Br_{\Phi}e^{H}(\ti)=\Bigl(U'_0U'_1U'_0U'_1-U'_1U'_0U'_1U'_0\Bigr)e^H(\ti)=0\ ,\]
using the braid relation for the intertwining elements.
\end{proof}

\begin{proof}[Proof of Lemma \ref{lem2}]
The proof is entirely similar to the proof of Lemma \ref{lem2A}.
\end{proof}

\noindent L. P.d'A.: { \sl \small Laboratoire de Math\'ematiques de Reims FRE 2011, Universit\'e de Reims Champagne-Ardenne, UFR Sciences exactes et naturelles,  Moulin de la Housse BP 1039, 51100 Reims, France} 
\newline \noindent {\tt \small email: loic.poulain-dandecy@univ-reims.fr}\\

\noindent R. W.: { \sl \small Universit\'{e} Paris Diderot-Paris VII, B\^{a}timent Sophie Germain, 75205 Paris Cedex 13 France} 
\newline \noindent {\tt \small email: ruari.walker@imj-prg.fr}


\begin{thebibliography}{99}

\bibitem{BK} J. Brundan and A. Kleshchev, \emph{Blocks of cyclotomic Hecke algebras and Khovanov--Lauda algebras}, Invent. math. 178(3) (2009), 451--484. arxiv 0808.2032

\bibitem{DM} R. Dipper and A. Mathas, \emph{Morita equivalences of Ariki--Koike algebras}. Math.
Zeit. 240 (2002), 579--610. arxiv math/0004014

\bibitem{EK} N. Enomoto, M. Kashiwara, \emph{Symmetric crystals and affine Hecke algebras of type B}, Proc. Japan Acad. 82, Ser. A, (2006) no. 8, 131--136. arxiv math/0608.079

\bibitem{Ro2} S. Rostam, \emph{Cyclotomic quiver Hecke algebras and Hecke algebra of $G(r,p,n)$},  arxiv:1609.08908\ .

\bibitem{VV} M. Varagnolo and E. Vasserot, \emph{Canonical bases and affine Hecke algebras of type B}, Invent. math. 185(3) (2011), 593--693. arxiv 0911.5209

\bibitem{KL}
M.~Khovanov and A.~Lauda, \emph{A diagrammatic approach to categorification of
  quantum groups. {I}}, Represent. Theory \textbf{13} (2009), 309--347.
  MR2525917 (2010i:17023)

\bibitem{Rou}
R.~Rouquier, \emph{2-{K}ac-{M}oody algebras}, preprint arXiv:0812.5023v1
  [math.{RT}].

\bibitem{SVV} P.~Shan, M.~Varagnolo, and E.~Vasserot, \emph{Canonical bases and affine {H}ecke algebras of type {$D$}}, Adv. Math. \textbf{227} (2011), no.~1,  267--291. MR2782195 (2012a:20007)

\bibitem{Walker}
R.~Walker, \emph{On {M}orita equivalences between {KLR} algebras and {VV}
  algebras}, preprint arXiv:1603.00796 [math.{RT}].


\end{thebibliography}
\end{document}